\documentclass[11pt]{article}
\usepackage{geometry}
\usepackage{graphicx}
\usepackage{caption}
\usepackage{subcaption}
\usepackage{amssymb}
\usepackage{amsthm}
\usepackage{csquotes}
\usepackage{amsmath}
\usepackage{mathtools}
\usepackage{float}
\usepackage{physics}
\usepackage{enumitem}
\usepackage[hidelinks]{hyperref}

\MakeOuterQuote{"}
\newtheorem{theorem}{Theorem}[section]
\newtheorem{definition}[theorem]{Definition}
\newtheorem{proposition}[theorem]{Proposition}
\newtheorem{lemma}[theorem]{Lemma}
\newtheorem{corollary}[theorem]{Corollary}
\newtheorem{hypothesis}[theorem]{Hypothesis}
\theoremstyle{remark}
\newtheorem{remark}[theorem]{Remark}

\newcommand{\R}{\mathbb{R}}
\newcommand{\C}{\mathbb{C}}
\newcommand{\Z}{\mathbb{Z}}
\newcommand{\N}{\mathbb{N}}

\newcommand{\KBS}{K_{BS}}
\newcommand{\Knear}{K_{BS}^{\mathrm{nr}}}
\newcommand{\Kfar}{K_{BS}^{\mathrm{fr}}}
\newcommand{\er}{\textbf{e}_r}
\newcommand{\etheta}{\textbf{e}_\theta}
\newcommand{\DT}{\mathcal{D}_T}
\newcommand{\LipO}{\mathrm{Lip}_0}

\newcommand{\Ren}{\Upsilon}
\newcommand{\Circ}{\Gamma}
\newcommand{\jap}[1]{\langle #1\rangle}
\newcommand{\Ip}[2]{I_{#1}(#2)}
\newcommand{\pout}{p^{\mathrm{out}}}
\newcommand{\pin}{p^{\mathrm{in}}}
\newcommand{\Qm}{Q_m}

\newcommand{\Orot}{\mathsf{O}}
\DeclareMathOperator{\sgn}{sgn}
\DeclareMathOperator{\divop}{div}
\DeclareMathOperator{\supp}{supp}
\DeclareMathOperator{\spec}{spec}

\newcommand{\xiL}{\xi_L}
\newcommand{\xia}{\xi_a}
\newcommand{\xib}{\xi_b}
\newcommand{\xiR}{\xi_R}
\newcommand{\prof}{\varpi}
\newcommand{\omc}{\overline{\omega}_{\mathrm{c}}}%
\newcommand{\Rc}{R_{\mathrm{c}}}
\newcommand{\Gone}{\overline{G}}

\newcommand{\bomega}{\overline{\omega}}

\newcommand{\Lin}{L}
\newcommand{\Lsec}[1]{L_{#1}}
\newcommand{\Green}{\mathcal{G}}
\newcommand{\kcoef}{k}
\newcommand{\Vric}{V}
\newcommand{\Evans}{\mathcal{E}}
\newcommand{\Wr}{\mathcal{W}}
\newcommand{\Mres}{M}
\newcommand{\lamst}{\lambda_*}
\newcommand{\phist}{\phi_*}

\newcommand{\ellst}{\ell_*}
\newcommand{\Riesz}{\Pi}
\newcommand{\Jsimp}{J}
\newcommand{\Thl}[1]{\Theta_{#1}}
\newcommand{\epsst}{\epsilon_*}
\newcommand{\Zedge}{Z}

\newcommand{\sigb}{\sigma_\flat}

\newcommand{\sbar}{\bar s}

\newcommand{\Cls}{C_{\ell_*}}

\newcommand{\Cp}{C_{\mathrm{p}}}
\newcommand{\Cnab}{C_{\nabla}}
\newcommand{\cout}{c_{\mathrm{out}}}

\newcommand{\Tstar}{T_\star}

\title{Non-Uniqueness of Weak Solutions to the Unforced\\
Two-Dimensional Euler Equations\\ with Vorticity in $L^1\cap L^p$,\\
I: A Compactness Criterion and the Linear Analysis}
\hypersetup{pdftitle={Non-Uniqueness of Weak Solutions to the Unforced
Two-Dimensional Euler Equations with Vorticity in L1 and Lp,
I: A Compactness Criterion and the Linear Analysis},
pdfauthor={Matei P. Coiculescu}}
\author{Matei P. Coiculescu\thanks{
Department of Mathematics, New York University,
Email: \texttt{m.coiculescu@nyu.edu}
}}

\begin{document}
\maketitle

\begin{abstract}
This paper and its companion Part II prove that the incompressible Euler
equations in the plane, without external force, admit two distinct renormalized
weak solutions with the same initial vorticity $f$, where $f$ is radial and
nonnegative and lies in $(L^1\cap L^q)(\mathbb{R}^2)$ for every $q<\infty$ but is
not bounded. One solution is radially symmetric and stationary. The other is
$3$-fold rotationally symmetric, is equimeasurable with $f$ at every time, and
departs from $f$ in finite time; the two solutions are distinct in
$L^\infty_t(L^1\cap L^q)_x$ on $[0,T]\times\mathbb{R}^2$ for every $q\in(4/3,\infty)$
and every sufficiently large $T$. In this first part we prove two ingredients. One is a
compactness criterion: for a radial datum, a sequence of $3$-fold rotationally
symmetric solutions with converging data, whose non-radial moment is bounded
below at some bounded time, has a subsequence converging to a second distinct
renormalized weak solution. The other is the linear analysis for the building
blocks of our multiscale datum. We show that the linearization of the Euler
equations around the basic building block plus a small point vortex has an
unstable eigenvalue in a space of $3$-fold rotationally symmetric
perturbations. This latter argument depends on rigorous interval arithmetic
that we discuss more fully in Part II.
\end{abstract}

\section{Introduction}

We consider the incompressible Euler equations in the plane, in vorticity form:
\begin{equation}
\label{eq:euler}
\partial_t \omega + v\cdot\nabla \omega = 0,\qquad
v = \KBS\ast \omega,\qquad
\omega(t=0) = \omega_0 .
\end{equation}
Here $\KBS(x)=x^\perp/(2\pi|x|^2)$, $x^\perp=(-x_2,x_1)$, denotes the
Biot-Savart kernel, so that $v$ is divergence free and
$\partial_1v_2-\partial_2v_1=\omega$. There is no external force, and we call
$\omega_0$ the datum.

For a smooth decaying datum, Equation~\eqref{eq:euler} has a unique global
smooth solution (Wolibner \cite{Wolibner}, H\"older \cite{Holder}; see Kato
\cite{Kato1967}; a self-contained proof for compactly supported data is in
Appendix~B of Part~II~\cite{PartII}). Yudovich proved in \cite{Yudovich1963}, for
bounded domains, that a weak solution with bounded vorticity is unique: its
velocity is log-Lipschitz, which closes an energy estimate for the difference of
two solutions. For the
plane, with vorticity in $(L^1\cap L^\infty)(\R^2)$ and square-integrable
velocity, a complete proof is given in Theorem~0.3 and Appendix~B.2 of \cite{ABCDGJK}.
For a datum in $(L^1\cap L^p)(\Omega)$, $1<p<\infty$, global weak solutions exist
(Yudovich for a bounded domain $\Omega$ \cite{Yudovich1963}, see the introduction of
\cite{Yudovich1995}; DiPerna and Majda for $\Omega=\R^2$ \cite{DiPernaMajda}). When $p>4/3$ the product $v\,\omega$ is
locally integrable and Equation~\eqref{eq:euler} has a weak formulation
(Lemma~\ref{lem:velocity}, Definition~\ref{def:weak}), but the velocity is only
Sobolev regular and Yudovich's uniqueness argument does not apply. Whether his theorem
extends to vorticity in $L^\infty_tL^p_x$, $p<\infty$, is the question of
Section~1.1.1 of \cite{BrueColomboKumar}. This paper and Part~II~\cite{PartII} prove that, without force in the
plane, uniqueness fails among renormalized weak solutions with vorticity in
$L^\infty_t(L^1\cap L^q)_x$, for every $q\in(4/3,\infty)$ at once and for an
explicit radial datum. The velocity of our datum is not square integrable (see
below), whereas the planar theorem of \cite{ABCDGJK} assumes
it is, and the question of \cite{BrueColomboKumar} is posed on $\mathbb T^2$.

Uniqueness is known in several classes of unbounded vorticity. Yudovich proved
in Theorem~1 of \cite{Yudovich1995}, in a bounded domain $D$, uniqueness among
solutions whose vorticity lies in every $L^p(D)$ with $p<\infty$, with $\|\omega(t)\|_{L^p(D)}$ bounded, for all large $p$, by a constant
multiple of a growth function of $p$ satisfying a divergence condition of Osgood
type. Taniuchi introduced uniformly local versions of these spaces on $\R^2$
and constructed global solutions in them \cite{Taniuchi}; uniqueness there
under the same kind of divergence condition is proved by Chen, Miao and Zheng
\cite{ChenMiaoZheng} and stated as Theorem~1.4 of \cite{CrippaStefani}, both
crediting Taniuchi, and Crippa and Stefani gave a Lagrangian proof of
uniqueness in localized Yudovich spaces (Theorem~1.6 of \cite{CrippaStefani}).
Vishik proved uniqueness in Besov-type spaces containing the local space $\mathrm{bmo}$, under an
additional integrability condition \cite{Vishik1999}. Bernicot and Keraani
\cite{BernicotKeraani}, Bernicot and Hmidi \cite{BernicotHmidi}, and Chen, Miao
and Zheng \cite{ChenMiaoZheng} gave spaces of unbounded vorticity with global
existence, uniqueness or propagation of the class; Elgindi, Murray and Said
\cite{ElgindiMurraySaid} gave a local-in-time existence and uniqueness class
whose solutions can leave it in finite time. Our datum satisfies none of these
uniqueness criteria: its $L^p(\R^2)$ norms grow exponentially in $p$, it is not in
$\mathrm{BMO}$, its dyadic blocks grow too fast for the spaces of
\cite{Vishik1999}, and it is not odd in the angle. We leave the verification of
these facts to the interested reader. So Theorem~\ref{thm:main} contradicts
none of these results.

With a radial force, Vishik \cite{Vishik1}, \cite{Vishik2} proved
non-uniqueness with vorticity in an $L^p(\R^2)$ space, breaking the radial symmetry. In the
account of Albritton, Bru\`e, Colombo, De Lellis, Giri, Janisch and Kwon
\cite{ABCDGJK}, for every $p\in(2,\infty)$ there are a datum in $(L^1\cap L^p)(\R^2)$
(possibly zero) with velocity in $L^2(\R^2)$ and a force in $L^1_t(L^1\cap L^p)_x$ whose
Biot-Savart velocity lies in $L^1_tL^2_x$, with uncountably many solutions with vorticity in
$L^\infty_t(L^1\cap L^p)_x$ and velocity in $L^\infty_tL^2_x$. The unstable
manifold of a steady vortex gives non-uniqueness only at time $-\infty$;
Vishik's self-similar background moves time $-\infty$ to time $0$, but it is
not an exact solution and requires the force (Section~0.1 of \cite{ABCDGJK}).
Without force, convex integration gives non-uniqueness for vorticity in $L^p(\R^2)$ and
$L^p(\mathbb T^2)$:
Mengual \cite{Mengual} showed, for $2<p<\infty$, that a truncated power-law
vortex, a steady solution with vorticity in $(L^1\cap L^p)(\R^2)$, is the datum of
infinitely many bounded admissible weak solutions, with velocity in $C_tL^2_x$.
In that result only the initial vorticity lies in $L^p$: at positive times the
vorticity of these solutions is known only as a distribution inside a growing disk,
so they are not known to lie in the class $L^\infty_tL^p_x$ of the question above.
Both solutions of Theorem~\ref{thm:main}, in contrast, have vorticity in
$L^\infty_t(L^1\cap L^q)_x$ for every $q\in(4/3,\infty)$, so Theorem~\ref{thm:main}
answers that question in the plane, in the class in which it is posed. Bru\`e,
Colombo and Kumar (Theorem~1.2 of \cite{BrueColomboKumar}) obtained, on
$\mathbb T^2$, for some $p>1$ and a dense set of initial velocities, infinitely
many weak solutions with vorticity in $C([0,1];L^p(\mathbb T^2))$, which do not conserve
energy. Outside the $L^p$ scale, convex integration also gives unforced
non-uniqueness: \cite{BrueColombo}
($L^{1,\infty}$ vorticity, bounded energy) and \cite{BuckModena} (Hardy spaces
$H^p(\R^2)$, $2/3<p<1$).

In an instability construction without force, the background must be an exact
solution. Ours is radial, hence stationary: infinitely many rescaled copies of
one unstable vortex ring on disjoint annuli accumulating at the origin, with
amplitudes increasing toward it. Inner rings rotate faster, and the task is to
carry an instability from an inner ring outward through every ring in a time
independent of the depth, uniformly in the number of interacting rings.

For $T>0$ and $1\le q\le\infty$ we write $X_q(T)$ for the space
$L^\infty_t(L^1\cap L^q)_x$ of functions on $(0,T)\times\R^2$. Renormalized weak solutions in $X_q(T)$, $q>4/3$ (after
DiPerna and Lions \cite{DiPernaLions}), and equimeasurability are defined in
Definitions~\ref{def:weak} and~\ref{def:equimeasurable}. The ring $\Phi$ of
Definition~\ref{def:ring} is a closed formula in $\log|y|$ times a fixed smooth
cutoff; it is radial, $\Phi\in C_c^\infty(\R^2)$, $0\le\Phi<1$,
$\Phi\not\equiv0$, and $\supp\Phi\subset\{a\le|y|\le1\}$ for an explicit
$a\in(0,1)$ (Lemma~\ref{lem:ring-facts}). For $A_1>0$, $\mu\in(1,2]$ and
$\sbar\in(0,a/16]$ we put
\begin{equation}
\label{eq:datum}
\begin{gathered}
f=\sum_{k\ge1}A_k\,\Phi(\cdot/r_k),\qquad A_k=A_1\mu^{k-1},\\
r_1=1,\qquad r_{k+1}=s_kr_k,\qquad s_k=\sbar(2\mu)^{-k}\qquad(k\ge1).
\end{gathered}
\end{equation}
The main theorem of the two papers is the following.

\begin{theorem}
\label{thm:main}
Let $\Phi$ and $a$ be as in Definition~\ref{def:ring}. There are
$\mu\in(1,\frac{17}{16}]$ and $\sbar\in(0,a/16]$ such that, for every $A_1>0$,
the function $f$ of Equation~\eqref{eq:datum} has the following properties. It
is radial and nonnegative, $f\in (L^1\cap L^q)(\R^2)$ for every $q\in[1,\infty)$,
and $f\notin L^\infty(\R^2)$. There is $\Tstar\in(0,\infty)$ such that, for
every $T>\Tstar$ and every $q\in(4/3,\infty)$, Equation~\eqref{eq:euler} with
$\omega_0=f$ has two distinct renormalized weak solutions in $X_q(T)$, in the
sense of Definition~\ref{def:weak}:
\begin{enumerate}[label=\textup{(\roman*)}]
\item the stationary solution $\omega(t)=f$, $t\in[0,T]$;
\item a solution $\omega_\infty$ such that
  $\omega_\infty\in C([0,\infty);L^p(\R^2))$ for every $p\in[1,\infty)$ and,
  for every $t\ge0$, $\omega_\infty(t)$ is equimeasurable with $f$ and
  $3$-fold rotationally symmetric.
\end{enumerate}
More precisely, there is $t_0\in[0,\Tstar)$ such that $\omega_\infty$ and $f$
are distinct elements of $X_q(T')$ for every $T'>t_0$ and every
$q\in(4/3,\infty)$.
\end{theorem}

We do not treat $q\le4/3$. No symmetry is
imposed on the class; the $3$-fold rotational symmetry of $\omega_\infty$ comes from the
construction, where it removes the angular modes $\ell\notin3\Z$, among them the
modes $\ell=\pm1$ that contain the translations of a ring, and leaves the
modes $\ell\in3\Z$, whose spectra we control under Hypothesis~\ref{hyp:N}.

Theorem~\ref{thm:main} concerns one explicit family of data. It says nothing about the kinetic
energy or the vanishing viscosity limit. The velocity of $f$ is not square integrable: by
Lemma~\ref{lem:radial-stationary}, since $\supp f$ lies in the closed unit disk,
$|\KBS\ast f(x)|=(2\pi|x|)^{-1}\int f$ for almost every $x$ with $|x|\ge1$, and
$\int f>0$.

The proof of Theorem~\ref{thm:main} is divided between this paper and
Part~II~\cite{PartII}. This paper proves a non-uniqueness criterion
(Section~\ref{sec:compact}) and the linear analysis of the ring
(Section~\ref{sec:ring}), and it uses no machine computation. Part~II proves
that \textup{(N1)}--\textup{(N3)} of Hypothesis~\ref{hyp:N} imply
Theorem~\ref{thm:main} (Proposition~1.2 there), by a nonlinear construction
followed by the criterion of Section~\ref{sec:compact}, and it proves
\textup{(N1)}--\textup{(N3)} (Theorem~1.3 there): \textup{(N1)} and
\textup{(N2)} with rigorous interval arithmetic and \textup{(N3)} by hand, in
its Appendix~A. Its Appendix~B gives the existence theory for the global smooth solutions
that the proof uses.

Section~\ref{sec:compact} follows the compactness theory of DiPerna and Lions for
sequences of renormalized solutions, in a self-contained form for vorticity in
$L^\infty_t(L^1\cap L^2)_x$, and proves a compactness criterion for
non-uniqueness (Theorem~\ref{thm:criterion}). The form of it that Part~II uses is
the following conditional non-uniqueness theorem (Corollary~\ref{cor:defect} and
Remark~\ref{rem:conditional}).

\begin{theorem}[Conditional non-uniqueness]
\label{thm:conditional}
Let $f$ be radial with $f\in (L^1\cap L^q)(\R^2)$ for every $q\in[1,\infty)$.
Suppose that there are data $f^{(n)}\to f$ in $(L^1\cap L^2)(\R^2)$ and, for each
$n$, a $3$-fold rotationally symmetric renormalized weak solution $\omega^{(n)}$ of
Equation~\eqref{eq:euler} on $[0,\infty)$ with datum $f^{(n)}$ and vorticity in
$L^\infty_t(L^1\cap L^2)_x$, with the following property: for some integer
$m\ge1$, some real $\chi\in C_c^\infty((0,\infty))$, some $c_0>0$ and
$\bar T\in(0,\infty)$, and some times $T^{(n)}\in[0,\bar T]$,
\[
\Big|\int_{\R^2}\omega^{(n)}(T^{(n)},x)\,\chi(|x|)\,(\bar x)^m\dd x\Big|\ge c_0
\qquad\text{for all large }n,
\]
with $x$ written as a complex number. Then Equation~\eqref{eq:euler} with
$\omega_0=f$ has, besides the stationary solution $\omega(t)=f$, a renormalized
weak solution $\omega_\infty$ with vorticity in $L^\infty_t(L^1\cap L^2)_x$ that is
$3$-fold rotationally symmetric and equimeasurable with $f$ at every time and lies
in $C([0,\infty);L^q(\R^2))$ for every $q\in[1,\infty)$. For some
$t_0\in[0,\bar T)$, $\omega_\infty$ and $f$ are distinct elements of $X_q(T')$ for
every $T'>t_0$ and every $q\in(4/3,\infty)$.
\end{theorem}

Since the moment in Theorem~\ref{thm:conditional} vanishes for the radial
function $f$, its hypothesis says that the approximate solutions leave $f$ by a
fixed amount, in one non-radial moment, within a time bounded independently of
$n$. The second solution is a limit of a subsequence of the $\omega^{(n)}$.

For the datum $f$ of Theorem~\ref{thm:main}, Part~II takes for $f^{(n)}$ the partial
sum of the first $n$ rings of $f$, with the innermost ring replaced by a seeded
ring: a $3$-fold rotationally symmetric copy of $\Phi$, equimeasurable with it and
close to it, that carries a small component along the unstable eigenfunction of
Theorem~\ref{thm:linear} below. Figure~\ref{fig:datum} shows, for $n=10$, the
partial sum $f_{10}=\sum_{k=1}^{10}A_k\Phi(\cdot/r_k)$ of the actual datum $f$ from
which non-uniqueness occurs; $f^{(10)}$ differs from it only by the seed, which is
invisible at this scale.

\begin{figure}[ht]
\centering
\includegraphics[width=0.95\textwidth]{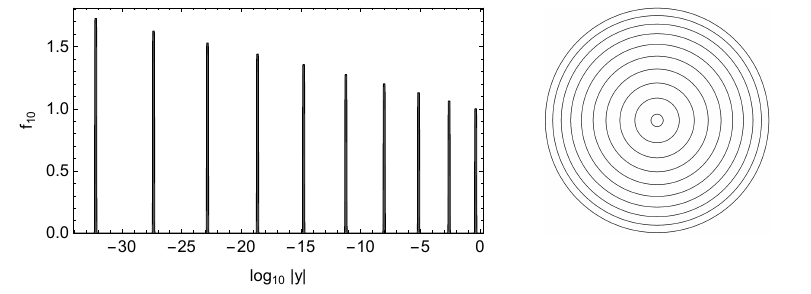}
\caption[The partial sum of ten rings of the datum]{The partial sum $f_{10}=\sum_{k=1}^{10}A_k\Phi(\cdot/r_k)$ of the datum $f$ of
Equation~\eqref{eq:datum}, for the illustrative values $A_1=1$, $\mu=\frac{17}{16}$ and
$\sbar=a/16$ (Theorem~\ref{thm:main} uses some $\mu\in(1,\frac{17}{16}]$ and
$\sbar\in(0,a/16]$ that are not computed; a smaller $\sbar$ separates the scales
further), and for the function $\chi$ constructed in the proof of
Lemma~\ref{lem:ring-facts}. Left: $f_{10}$ against $\log_{10}|y|$; the rings sit at
the scales $r_1=1>r_2>\dots>r_{10}\approx10^{-32}$, and their amplitudes
$A_k=\mu^{k-1}$ grow toward the origin. Right: the same ten rings in a logarithmic
polar picture, in which the distance from the center is an affine function of
$\log|y|$, so that all of them are visible; each ring is drawn without its amplitude,
darker where $\Phi(\cdot/r_k)$ is larger. Drawn with Mathematica from the exact
formulas; the figure is an illustration and is not used in any proof.}
\label{fig:datum}
\end{figure}

The proof of Theorem~\ref{thm:main} uses three explicit statements, collected in Section~\ref{sec:ring}
as Hypothesis~\ref{hyp:N}: \textup{(N1)} and \textup{(N2)} locate the zeros in
the mode $\ell=3$ of the Evans function of the ring, an explicit holomorphic
function, for every $G\in[0,\frac1{50}]$, where $2\pi G$ is the circulation of
a point vortex at the origin in the reference units of
Definition~\ref{def:ring}; \textup{(N3)} is an inequality between explicit
numbers controlling the modes $\ell=\pm6$ and $|\ell|\ge9$. Apart from the
thresholds stated in \textup{(N1)}--\textup{(N3)}, the only explicit numbers that
the computer-assisted proof of these statements uses are the constants of the ring
in Definition~\ref{def:ring}, which fix the building block of the multiscale datum;
all other constants of this paper are used only through their existence. Without
this hypothesis, Section~\ref{sec:ring} decomposes the linearization into angular
modes (Lemma~\ref{lem:sector}), reduces each mode to a linear second-order
equation in $\xi=\log r$, and defines the Evans function of each mode
$\ell\ge1$ (Definition~\ref{def:evans}). The Evans function is, up to sign, the Wronskian of
the solutions of this equation that decay at the two ends
(Lemma~\ref{lem:evans-basic}), and its zeros in $\{\Re\lambda>0\}$ are exactly
the eigenvalues there (Lemma~\ref{lem:resolvent}). A simple zero of the Evans
function gives an algebraically simple eigenvalue with an explicit Riesz
projection of rank one (Theorem~\ref{thm:evans-simple}).
Section~\ref{sec:ring} also bounds the real part of every unstable eigenvalue
of the mode $\ell$, for a radial vorticity that increases and then decreases
in $r$, by half the square root of an explicit double integral of the radial
derivative of the
vorticity over the region where it increases and the region where it decreases,
which is nonincreasing in $\ell$ (Lemma~\ref{lem:edge}). For the ring this
integral is computed in closed form (Lemma~\ref{lem:closed-form}), and
\textup{(N3)} makes the bound smaller than $\frac9{200}$ for $\ell=\pm6$ and
$|\ell|\ge9$. The section also proves a semigroup bound, uniform in $G$, for all
sufficiently large $|\ell|$ (Proposition~\ref{prop:large-sectors}), and constructs the
semigroup of the linearization (Lemma~\ref{lem:semigroup-construct}).

The main result of the linear analysis is the linear instability of the ring,
in the ring units of Definition~\ref{def:ring}, in the presence of a point vortex
of small circulation $2\pi G$ at the origin, which stands for everything inside
the ring. For an integer $s\ge0$ let $X^s$ be the space of $3$-fold rotationally
symmetric $w\in H^s(\R^2;\C)$ supported in the annulus $\{\frac a2\le|y|\le2\}$.

\begin{theorem}[Linear instability of the ring]
\label{thm:linear}
Assume \textup{(N1)}--\textup{(N3)} of Hypothesis~\ref{hyp:N}, which Part~II
proves. For $G\ge0$ let
\[
\Lin(G)w=-\Big(\KBS\ast\Phi+\frac{G\,y^\perp}{|y|^2}\Big)\cdot\nabla w-(\KBS\ast w)\cdot\nabla\Phi
\]
be the linearization of Equation~\eqref{eq:euler} at the steady state $\Phi$ plus a
point vortex of circulation $2\pi G$ at the origin. Let $\Gone\approx0.006024$ be
the constant of Definition~\ref{def:ring}. Then, for every $G\in[0,\Gone]$ and
every integer $s\ge0$:
\begin{enumerate}[label=\textup{(\roman*)}]
\item $\Lin(G)$, defined on $X^{s+1}$, extends to the generator of a
  $C_0$-semigroup $e^{t\Lin(G)}$ on $X^s$. Its spectrum in
  $\{\Re\lambda>\frac9{200}\}$ is $\{\lamst(G),\overline{\lamst(G)}\}$, with
  $\Re\lamst(G)>\frac{211}{2500}$. The eigenvalue $\lamst(G)$ has the
  eigenfunction $\phist(G)=g(|y|)e^{3i\theta}$, with $g$ smooth and supported in
  $[a,1]$ and $\|\phist(G)\|_{L^2(\R^2)}=1$, and it is algebraically simple on the
  functions in $X^s$ of angular mode $3$; $\overline{\lamst(G)}$ has the
  eigenfunction $\overline{\phist(G)}$, of angular mode $-3$.
\item $\lamst(G)$ depends continuously on $G$.
\item There is $C$, depending only on $s$, such that
  $\|e^{t\Lin(G)}\|_{X^s\to X^s}\le Ce^{\Re\lamst(G)\,t}$ for $t\ge0$. For every
  $\sigma'\in(\frac9{200},\min_{G'}\Re\lamst(G'))$ there is $C'$, depending only
  on $s$ and $\sigma'$, such that
  $\|e^{t\Lin(G)}(1-\Riesz(G))\|_{X^s\to X^s}\le C'e^{\sigma't}$ for $t\ge0$, where
  $\Riesz(G)$ is the spectral projection of $\Lin(G)$ for
  $\{\lamst(G),\overline{\lamst(G)}\}$, whose range is spanned by $\phist(G)$ and
  $\overline{\phist(G)}$.
\item The moment $\int_{\R^2}\phist(G)(y)\,(\bar y)^3\dd y$, with $y$ written as
  a complex number, has modulus at least a constant $c>0$ independent of $G$.
\end{enumerate}
\end{theorem}

Theorem~\ref{thm:linear} collects Proposition~\ref{prop:SU},
Lemma~\ref{lem:resolvent}(2), Propositions~\ref{prop:U} and~\ref{prop:semigroup},
and Lemma~\ref{lem:uniform}. In Section~\ref{sec:ring} these are stated with the
precise eigendata that Part~II uses (Definition~\ref{def:eigendata}).

Part~II uses these conclusions in a nonlinear construction, which we describe
heuristically. Each ring of $f$ sees the rings inside it essentially as a point
vortex of small circulation, so its linearization is the operator $\Lin(G)$ of
Theorem~\ref{thm:linear} for a small $G$, and a perturbation along its unstable eigenfunction grows at a rate
proportional to the amplitude of the ring: inner rings are faster. An approximation of $f$ that keeps finitely many rings
carries a small multiple of the unstable eigenfunction on its innermost ring.
This perturbation grows until the ring saturates; the non-radial moment of the saturated ring then
kicks the next ring outward along its unstable eigenfunction, and that ring
grows and saturates in turn. Ring $k$ needs a time of order $k/A_k$ to saturate, and $A_k$ grows
geometrically toward the origin, so this relay reaches the outermost ring in a
time bounded independently of the number of rings, and one bootstrap keeps every ring close to its
expected evolution. This is why the datum is multiscale. A single unstable ring
would not suffice: a perturbation of size $\varepsilon$ of a steady state needs a
time of order $\log(1/\varepsilon)$ to reach a fixed size, which tends to infinity
as $\varepsilon\to0$, so approximations that converge to the steady datum would
not leave it in a bounded time; this is the reason why the unstable manifold of a
steady vortex gives non-uniqueness only at time $-\infty$. In our datum the
approximation differs from $f$ only at the scale of its innermost ring, where the
seed has a fixed size relative to the ring; it is small in $(L^1\cap L^q)(\R^2)$
only because that ring is small. So the instability starts at order one on the
time scale of that ring, which tends to $0$ with the depth, and the relay carries
it to the outermost ring in a bounded time. The instability is what produces a non-radial moment bounded
away from zero. At saturation the perturbation of the outermost ring is, to
leading order, a multiple of fixed size of its unstable eigenfunction and its
conjugate, which lie in the angular modes $3$ and $-3$. The moment
$\int\omega\Psi$ with $\Psi(x)=\chi(|x|)(\bar x)^3$, where $\chi=1$ on the annulus
of the ring, sees only the angular mode $3$, like $\pout_3$
(Definition~\ref{def:multipole}), and on the unstable eigenfunction it is the
moment of Theorem~\ref{thm:linear}(iv), which is bounded away from zero uniformly
in $G$. So this moment of the outermost ring has modulus at least a constant
$c_0>0$ independent of the number of rings, while it vanishes for the radial datum
$f$. This is the hypothesis of Theorem~\ref{thm:conditional}, which gives the
second solution as a limit of the approximate solutions. Section~1.2 of Part~II lists the definitions and results of this paper
that it uses.

The main idea of a linear instability in physical coordinates driving the
non-uniqueness mechanism, the use of non-monotone unstable radial vortices and of
a multiscale datum, and the overall strategy of the proof are due to the author.
Within this strategy, the detailed proofs and estimates, the code for the
computer-assisted verification in Part~II and the text were largely produced with
Claude Opus 5.5 (Anthropic), under the direction of the author. The author has
checked all proofs by hand and takes full responsibility for the paper.

\paragraph{Acknowledgments.} I am supported at New York University by a Junior
Fellowship of the Simons Society of Fellows, funded by Simons Foundation
International and administered by the Simons Foundation. I would also like to thank
Tristan Buckmaster for generously offering his time and advice on this and many
other projects.

\section{Preliminaries}
\label{sec:prelim}

\subsection{Conventions}

We identify $\R^2$ with $\C$, writing $x=(x_1,x_2)=x_1+ix_2$ when convenient,
and $\bar x$ for the complex conjugate; $|x|$ is the Euclidean norm; for a matrix or an array of numbers, $|\cdot|$ is the
Euclidean norm of the array of its entries, and for a matrix $M$ we write
$\|M\|_{\mathrm{op}}$ for its operator norm. For $r>0$, $B_r=\{|x|<r\}$ and
$\overline B_r$ is its closure; $|S|$ is the Lebesgue measure of a measurable
set $S\subset\R^2$. The overline thus denotes complex conjugation, of numbers and
of functions, and the closure of a set; it is also part of the names of some fixed
objects, such as the radial background $\bomega$ of Section~\ref{sec:ring} and the
constants $\Gone$, $\sbar$ and $\bar T$. Which meaning is intended will always be
clear from the context. We use polar
coordinates $x=r(\cos\theta,\sin\theta)$, with $\er=(\cos\theta,\sin\theta)$
and $\etheta=(-\sin\theta,\cos\theta)$. For $\theta\in\R$, $\Orot_\theta$ is
the rotation of $\R^2$ by the angle $\theta$ (multiplication by $e^{i\theta}$),
and $\Orot=\Orot_{2\pi/3}$. For an integer $m\ge1$, a function $g$ on $\R^2$ is
\textit{$m$-fold rotationally symmetric} if and only if $g\circ\Orot_{2\pi/m}=g$ almost
everywhere, and \textit{radial} if and only if $g\circ\Orot_\theta=g$ almost
everywhere for all $\theta$. If $\theta\mapsto g(re^{i\theta})$ is integrable
on $[0,2\pi]$ for almost every $r>0$ (by Fubini's theorem in polar
coordinates, this holds for every $g\in L^1_{\mathrm{loc}}(\R^2)$) and
$\ell\in\Z$, the \textbf{angular mode} $\ell$ of $g$ is
$g_\ell(r)=\frac1{2\pi}\int_0^{2\pi}g(re^{i\theta})e^{-i\ell\theta}\dd\theta$,
defined for almost every $r>0$. If $g\in L^2(\R^2)$, then
$g=\sum_{\ell\in\Z}g_\ell(r)e^{i\ell\theta}$ in $L^2(\R^2)$, and $g$ is
$m$-fold rotationally symmetric if and only if $g_\ell=0$ for all $\ell\notin m\Z$.

The Biot-Savart kernel is $\KBS(x)=x^\perp/(2\pi|x|^2)$ with
$x^\perp=(-x_2,x_1)$, and $\nabla^\perp=(-\partial_2,\partial_1)$. For a
vorticity $\omega$ we write $v=\KBS\ast\omega$ whenever the integral converges.
If $\omega\in C_c^\infty(\R^2)$, then the \textbf{stream function}
$\psi(x)=\frac1{2\pi}\int\log|x-y|\,\omega(y)\dd y$ satisfies
$v=\nabla^\perp\psi$, $\Delta\psi=\omega$ and $\partial_1v_2-\partial_2v_1=\omega$
(Lemma~\ref{lem:BS-basic}(d) below). In polar coordinates these sign
conventions read $v\cdot\er=-r^{-1}\partial_\theta\psi$ and
$v\cdot\etheta=\partial_r\psi$. For a vector field $u$ we write $\nabla u$ for
the matrix $(\partial_lu_i)_{i,l=1,2}$, so that $\operatorname{tr}\nabla
u=\divop u$.

Functions are complex valued unless we say otherwise; vorticities and
velocities are real. For a multi-index $\gamma=(\gamma_1,\gamma_2)\in\N_0^2$ we
write $\partial^\gamma=\partial_1^{\gamma_1}\partial_2^{\gamma_2}$,
$|\gamma|=\gamma_1+\gamma_2$, $\zeta^{2\gamma}=\zeta_1^{2\gamma_1}\zeta_2^{2\gamma_2}$
for $\zeta\in\R^2$,
$\beta\le\gamma$ if $\beta_1\le\gamma_1$ and $\beta_2\le\gamma_2$, and
$\binom\gamma\beta=\binom{\gamma_1}{\beta_1}\binom{\gamma_2}{\beta_2}$; here
$\N=\{1,2,\dots\}$, $\N_0=\N\cup\{0\}$ and $m\N=\{m,2m,\dots\}$. For an integer
$k\ge0$ and an open set $V\subset\R^2$, $H^k(V)$ is the Sobolev space with norm
$\|g\|_{H^k(V)}^2=\sum_{|\gamma|\le k}\|\partial^\gamma g\|_{L^2(V)}^2$, and
$H^k=H^k(\R^2)$. Sobolev norms of vector and matrix fields are summed over the
entries; for instance $\|\nabla u\|_{H^k}^2=\sum_{i,l}\|\partial_lu_i\|_{H^k}^2$.
For a vector or matrix field $F$ on $V$ and $1\le q\le\infty$ we put
$\|F\|_{L^q(V)}=\|\,|F|\,\|_{L^q(V)}$, where $|F|$ is the Euclidean norm of the
array of entries; for $q=2$ this agrees with the sum over the entries. For a
vector field $F$ on $V$ we put
$\|F\|_{W^{1,\infty}(V)}=\operatorname{ess\,sup}_V|F|+\operatorname{ess\,sup}_V|\nabla F|$.

For a function or vector field $h$ of class $C^q$ near a point $x$, we write
$\nabla^qh(x)$ for the array of all partial derivatives of order $q$ of the
components of $h$ at $x$, one entry for each ordered $q$-tuple of directions,
so that $|\nabla^qh(x)|$ is its Euclidean norm. For an integer $k\ge0$, a closed
set $U\subset\R^2$ and $h$ of class $C^k$ on a neighborhood of $U$, we put
$\|h\|_{C^k(U)}=\sum_{q=0}^k\sup_{x\in U}|\nabla^qh(x)|$;
in particular $|\partial^\gamma h(x)|\le\|h\|_{C^k(U)}$ for $x\in U$ and
$|\gamma|\le k$, and $C^k(U)$ denotes the space of such $h$ with finite norm.
The Fourier transform is $\hat g(\zeta)=\int g(x)e^{-ix\cdot\zeta}\dd x$, so that
$\|\partial^\gamma g\|_{L^2(\R^2)}^2=(2\pi)^{-2}\int\zeta^{2\gamma}|\hat g(\zeta)|^2\dd\zeta$
by Plancherel's theorem.

We put $\jap{t}=(1+t^2)^{1/2}$ for $t\in\R$ and $\log_+x=\max\{\log x,0\}$ for
$x>0$. For $x>0$ and an integer $p\ge0$ we put
\begin{equation}
\label{eq:Ip}
\Ip{p}{x}=\int_0^\infty e^{-xu}(1+u)^p\dd u,
\end{equation}
a finite (as $x>0$), decreasing function of $x$.

Constants that enter a later choice of parameters carry their own names, and
their dependence is stated where they are defined. The letter $C$ denotes a
constant that enters no later choice; its dependence is stated each time, and
within one proof its value may change from line to line. For a function $u$
and a nonnegative function $g$, $u=O(g)$ as $r\to r_0$ means that $|u|\le Cg$ in
a neighborhood of $r_0$ for some constant $C$.

Let $I$ be an interval and $X$ a Banach space. A function $F\colon I\to X$ is
\textit{strongly measurable} if it is an almost everywhere limit of simple
functions $\sum_{i=1}^k\mathbf 1_{E_i}(t)x_i$, with $E_i\subset I$ measurable
and $x_i\in X$, and \textit{weakly measurable} if $t\mapsto\Lambda(F(t))$ is
measurable for every $\Lambda\in X^*$. We use Pettis's theorem \cite{Pettis}:
if $X$ is separable, then $F$ is strongly measurable if and only if it is weakly
measurable. The Bochner spaces $L^\infty(I;X)$ and $L^1(I;X)$ consist of the
strongly measurable $F$ for which $t\mapsto\|F(t)\|_X$ is essentially bounded,
respectively integrable, modulo equality almost everywhere.

Let $\Omega\subset\R^2$ be a measurable set. For
$b\in[1,\infty)$, $L^\infty(I;L^b(\Omega))$ and $L^1(I;L^b(\Omega))$ are the
Bochner spaces, whose elements we identify with measurable functions on
$I\times\Omega$; a measurable $u$ on $I\times\Omega$ lies in them if and only if
$t\mapsto\|u(t)\|_{L^b(\Omega)}$ is essentially bounded, respectively
integrable, by Pettis's theorem: $L^b(\Omega)$ is separable, and for
$\phi\in L^{b'}(\Omega)$, $1/b+1/b'=1$, the function $t\mapsto\int u(t)\phi$,
defined for almost every $t$, is measurable by Tonelli's theorem applied to the
positive and negative parts of the real and imaginary parts of $u\phi$. We write $L^\infty(I;L^\infty(\Omega))$
for $L^\infty(I\times\Omega)$, and $L^1(I;L^\infty(\Omega))$ for the set of
measurable $u$ on $I\times\Omega$ with $\int_I\|u(t)\|_{L^\infty(\Omega)}\dd t<\infty$.
Intersections such as $L^\infty(I;(L^1\cap L^\infty)(\Omega))$ are intersections of these
spaces; a space with $W^{1,a}(\Omega)$ in place of $L^b(\Omega)$ consists of the
$w$ for which $w$ and its distributional gradient in $x$ lie in the
corresponding space with $L^a(\Omega)$; the subscript $\mathrm{loc}$ on the time
interval means that the condition holds on $J\times\Omega$ for every compact
interval $J\subset I$, on $\R^2$ that it holds on $I\times B_R$ for every $R>0$,
and on both that it holds on $J\times B_R$ for all such $J$ and $R$.

For $T>0$ and $1\le q\le\infty$, the space $X_q(T)$ of Section~1, written
$L^\infty_t(L^1\cap L^q)_x$ there, is the Bochner space
$L^\infty\big(0,T;(L^1\cap L^q)(\R^2)\big)$ just described, with
$\|\omega\|_{(L^1\cap L^q)(\R^2)}=\|\omega\|_{L^1(\R^2)}+\|\omega\|_{L^q(\R^2)}$; and
$\DT=\{\varphi\in C_c^\infty(\R\times\R^2):\supp\varphi\subset(-\infty,T)\times\R^2\}$.
$\LipO$ is the set of Lipschitz functions $\Ren\colon\R\to\R$ with $\Ren(0)=0$.
For $g\in L^1(\R^2)$ we put $\Circ_g(r)=\int_{B_r}g$.

A $C_0$-semigroup $(S(t))_{t\ge0}$ on a Banach space $X$ is a family of bounded
operators with $S(0)$ the identity, $S(t+s)=S(t)S(s)$, and $t\mapsto S(t)x$ continuous for
all $x\in X$; its generator $A$ is $Ax=\lim_{h\downarrow0}h^{-1}(S(h)x-x)$ on the
set $D(A)$ of $x$ for which the limit exists.

\subsection{The Biot-Savart Law}

We split $\KBS=\Knear+\Kfar$ with $\Knear=\mathbf 1_{B_1}\KBS$. By
Lemma~\ref{lem:BS-basic}(a) below,
$\|\Knear\|_{L^s(\R^2)}^s=(2\pi)^{1-s}/(2-s)$ for $s\in[1,2)$, and
$|\Kfar(z)|\le\min\{1,|z|^{-1}\}/(2\pi)$ for $z\in\R^2$.

\begin{lemma}
\label{lem:BS-basic}
The kernel $\KBS$ has the following properties.
\begin{enumerate}[label=\textup{(\alph*)}]
\item For $z\ne0$, $i,l\in\{1,2\}$ and $\rho>0$, $|\KBS(z)|=(2\pi|z|)^{-1}$,
  $|\partial_l(\KBS)_i(z)|\le(2\pi|z|^2)^{-1}$ and
  $\int_{B_\rho}|\KBS|\dd z=\rho$.
\item Let $g\in (L^1\cap L^\infty)(\R^2)$. Then for every $x\in\R^2$,
  \[
  \big(|\KBS|\ast|g|\big)(x)\le\Big(\frac{2\|g\|_{L^1(\R^2)}\|g\|_{L^\infty(\R^2)}}{\pi}\Big)^{1/2} .
  \]
  In particular the integral defining $\KBS\ast g(x)$ converges absolutely for
  every $x$, and
  $\|\KBS\ast g\|_{L^\infty(\R^2)}\le\big(2\|g\|_{L^1(\R^2)}\|g\|_{L^\infty(\R^2)}/\pi\big)^{1/2}$.
\item Let $\rho,R>0$ and let $g\in L^2(\R^2)$ vanish outside $B_\rho$. Then
  $\KBS\ast g$ is defined almost everywhere, and
  $\|\KBS\ast g\|_{L^2(B_R)}\le(R+\rho)\|g\|_{L^2(\R^2)}$.
\item Let $g\in C_c^\infty(\R^2)$ and $u=\KBS\ast g$. Then $u\in C^\infty(\R^2)$,
  $\partial^\gamma u=\KBS\ast\partial^\gamma g$ for every $\gamma$,
  $u=\nabla^\perp\psi$ with $\psi(x)=\frac1{2\pi}\int\log|x-y|\,g(y)\dd y$,
  $\Delta\psi=g$, $\divop u=0$ and $\partial_1u_2-\partial_2u_1=g$. Moreover
  $\sum_{i,l=1}^2\|\partial_l\partial^\gamma u_i\|_{L^2(\R^2)}^2=\|\partial^\gamma g\|_{L^2(\R^2)}^2$
  for every $\gamma$, and hence $\|\nabla u\|_{H^k}=\|g\|_{H^k}$ for every
  integer $k\ge0$.
\item Let $g\in L^1(\R^2)$ and $R>0$. Then
  $\int_{B_R}\big(|\KBS|\ast|g|\big)\dd x\le R\,\|g\|_{L^1(\R^2)}$. In particular the
  integral defining $\KBS\ast g$ converges absolutely almost everywhere,
  $\KBS\ast g\in L^1_{\mathrm{loc}}(\R^2)$, and $\divop(\KBS\ast g)=0$ in
  $\mathcal D'(\R^2)$.
\end{enumerate}
\end{lemma}

\begin{proof}
(a) The first identity holds because $|x^\perp|=|x|$, and integration in polar
coordinates gives $\int_{B_\rho}(2\pi|z|)^{-1}\dd z=\rho$. Since
$2\pi\KBS(z)=(-z_2,z_1)/|z|^2$, the derivatives $2\pi\partial_l(\KBS)_i$ are,
up to sign, the numbers
$\partial_l(z_k/|z|^2)=\delta_{lk}|z|^{-2}-2z_lz_k|z|^{-4}$, whose modulus is
$|z_{k'}^2-z_k^2|/|z|^4\le|z|^{-2}$ if $l=k$ (where $k'\ne k$) and
$2|z_1z_2|/|z|^4\le|z|^{-2}$ if $l\ne k$.

(b) We may assume $g\ne0$. For arbitrary $r>0$, (a) gives, after splitting the
integral into the near field $\{|z|<r\}$ and the far field $\{|z|\ge r\}$,
\[
\begin{aligned}
2\pi\big(|\KBS|\ast|g|\big)(x)&\le\|g\|_{L^\infty(\R^2)}\int_{B_r}\frac{\dd z}{|z|}
+\frac1r\int_{|z|\ge r}|g(x-z)|\dd z\\
&\le2\pi r\|g\|_{L^\infty(\R^2)}+\frac{\|g\|_{L^1(\R^2)}}{r},
\end{aligned}
\]
and choosing $r=(\|g\|_{L^1(\R^2)}/(2\pi\|g\|_{L^\infty(\R^2)}))^{1/2}$ gives
\[
2\pi(|\KBS|\ast|g|)(x)\le2(2\pi\|g\|_{L^1(\R^2)}\|g\|_{L^\infty(\R^2)})^{1/2},
\]
which is the claim.
The rest follows from $|\KBS\ast g|\le|\KBS|\ast|g|$.

(c) For $|x|<R$ and $|x-z|<\rho$, we have $|z|<\rho+R$, so, as $g$ vanishes
outside $B_\rho$, $\KBS\ast g(x)=\int_{B_{R+\rho}}\KBS(z)g(x-z)\dd z$ for $x\in B_R$, and Minkowski's
integral inequality and (a) give
$\big\|\int_{B_{R+\rho}}|\KBS(z)||g(\cdot-z)|\dd z\big\|_{L^2(B_R)}\le(R+\rho)\|g\|_{L^2(\R^2)}$.
In particular the integral converges absolutely almost everywhere on $B_R$, for
every $R>0$.

(d) Let $\supp g\subset B_\rho$. Since $\KBS$ and $\log|\cdot|$ are locally
integrable, differentiation under the integral sign shows that $u,\psi\in
C^\infty$, $\partial^\gamma u=\KBS\ast\partial^\gamma g$ and
$\partial^\gamma\psi=\frac1{2\pi}\log|\cdot|\ast\partial^\gamma g$. Writing
$\partial_jg(x-z)=-\partial_{z_j}[g(x-z)]$ and integrating by parts on
$\{|z|>\varepsilon\}$, with a boundary term at most
$\varepsilon|\log\varepsilon|\sup|g|\to0$, we get
$\partial_j\psi(x)=\int\frac{z_j}{2\pi|z|^2}g(x-z)\dd z$. Hence
$u=\nabla^\perp\psi$, so $\divop u=-\partial_1\partial_2\psi+\partial_2\partial_1\psi=0$
and $\Delta\psi=\partial_1u_2-\partial_2u_1$. Fix $x$, put
$h(z)=g(x-z)$, and use polar coordinates $z=re^{i\theta}$, in which
$z\cdot\nabla_z=r\partial_r$.
Since $\partial_lu=\KBS\ast\partial_lg$ and $(\partial_lg)(x-z)=-\partial_{z_l}h(z)$,
\[
(\partial_1u_2-\partial_2u_1)(x)=-\int\frac{z\cdot\nabla_zh}{2\pi|z|^2}\dd z
=-\frac1{2\pi}\int_0^{2\pi}\!\!\int_0^\infty\partial_rh\dd r\dd\theta=g(x),
\]
because $(\KBS)_2\partial_1h-(\KBS)_1\partial_2h=z\cdot\nabla_zh/(2\pi|z|^2)$ and
$\int_0^\infty\partial_rh\dd r=-h(0)=-g(x)$; the integrand is integrable, as
$|\partial_rh|\le|\nabla h|$ and $h$ has compact support. Hence $\Delta\psi=g$.

For $|x|\ge2\rho$ and $|y|<\rho$ we have $|x-y|\ge|x|/2$; differentiating under
the integral sign, with $|\nabla\log|z||=|z|^{-1}$ and
$|\partial_i\partial_j\log|z||=|\partial_j(z_i/|z|^2)|\le|z|^{-2}$ (proof of
(a)), gives $|\nabla\psi(x)|\le\|g\|_{L^1(\R^2)}/(\pi|x|)$ and
$|\partial_i\partial_j\psi(x)|\le2\|g\|_{L^1(\R^2)}/(\pi|x|^2)$ for $|x|\ge2\rho$.
Two integrations by parts on $B_r$
give $\int_{B_r}\sum_{i,j}(\partial_i\partial_j\psi)^2\dd x
=\int_{B_r}(\Delta\psi)^2\dd x+\mathcal B_r$, where $\mathcal B_r$ is a sum of
integrals over $\partial B_r$ of products of a first and a second derivative of
$\psi$; for $r\ge2\rho$ the bounds above give
$|\mathcal B_r|\le Cr\cdot r^{-1}r^{-2}\|g\|_{L^1(\R^2)}^2$ with $C$ absolute. Letting
$r\to\infty$ gives $\sum_{i,j}\|\partial_i\partial_j\psi\|_{L^2(\R^2)}^2=\|g\|_{L^2(\R^2)}^2$.
The entries of $\nabla u$ are, up to sign, the second derivatives of $\psi$, so
$\sum_{i,l}\|\partial_lu_i\|_{L^2(\R^2)}^2=\|g\|_{L^2(\R^2)}^2$. Replacing $g$ by
$\partial^\gamma g$, whose velocity is $\partial^\gamma u$, and summing over
$|\gamma|\le k$ gives the rest. For complex $g$ we apply all of this to the
real and imaginary parts.

(e) Let $D$ be a disk of radius $R$. Then $|D\setminus B_R|=|B_R\setminus D|$.
Since $|\KBS(z)|=(2\pi|z|)^{-1}$ by (a), $|\KBS|\le(2\pi R)^{-1}$ on $D\setminus B_R$
and $|\KBS|>(2\pi R)^{-1}$ on $B_R\setminus D$; so
$\int_D|\KBS|\le\int_{B_R}|\KBS|=R$, by (a). For fixed $y$, the substitution
$z=x-y$ maps $B_R$ onto the disk $D_y=\{z:|z+y|<R\}$ of radius $R$, so Tonelli's
theorem gives
$\int_{B_R}(|\KBS|\ast|g|)\dd x=\int|g(y)|\int_{D_y}|\KBS(z)|\dd z\dd y\le R\|g\|_{L^1(\R^2)}$.
Hence $|\KBS|\ast|g|$ is finite almost everywhere, and
$\KBS\ast g\in L^1_{\mathrm{loc}}(\R^2)$ because $|\KBS\ast g|\le|\KBS|\ast|g|$.
Take $g_n\in C_c^\infty(\R^2)$ with $g_n\to g$ in $L^1(\R^2)$. This bound, applied to
$g_n-g$, gives $\KBS\ast g_n\to\KBS\ast g$ in $L^1(B_R)$ for every $R>0$, hence
in $\mathcal D'(\R^2)$; since $\divop(\KBS\ast g_n)=0$ by (d),
$\divop(\KBS\ast g)=0$ in $\mathcal D'(\R^2)$.
\end{proof}

\begin{lemma}
\label{lem:velocity}
Let $4/3<p<\infty$, $p'=p/(p-1)$, and $R>0$. There is $C=C(p,R)$ such that
the following hold.
\begin{enumerate}[label=\textup{(\alph*)}]
\item For $\omega\in (L^1\cap L^p)(\R^2)$ the integral defining $\KBS\ast\omega$
  converges absolutely almost everywhere and, with $v=\KBS\ast\omega$,
  \[
  \|v\|_{L^{p'}(B_R)}+\|v\|_{L^2(B_R)}\le C\|\omega\|_{(L^1\cap L^p)(\R^2)},\qquad
  \int_{\R^2}|\omega|\,\big(|\KBS|\ast|\omega|\big)\dd x\le C\|\omega\|_{(L^1\cap L^p)(\R^2)}^2 ;
  \]
  the constant in the second inequality can be chosen depending only on $p$,
  independently of $R$.
\item Let $\omega_n$ be real and $\omega_n\rightharpoonup\omega$ weakly in $L^p(\R^2)$, with
  $\sup_n\|\omega_n\|_{L^1(\R^2)}\le M$. Then $\|\omega\|_{L^1(\R^2)}\le M$ and
  $\KBS\ast\omega_n\to\KBS\ast\omega$ strongly in $L^{p'}(B_R)$.
\item If $\omega\in (L^1\cap L^2)(\R^2)$, then $v=\KBS\ast\omega$ satisfies $\nabla v\in L^2(\R^2)$ and
  $\|\nabla v\|_{L^2(\R^2)}=\|\omega\|_{L^2(\R^2)}$.
\end{enumerate}
\end{lemma}

\begin{proof}
(a) Put $r(p)=p'$ and $s(p)=p/(2p-2)$ if $p\le2$, and $r(p)=p$ and $s(p)=1$ if
$p>2$. Since $p>4/3$, we have $s(p)\in[1,2)$, so $\Knear\in L^{s(p)}(\R^2)$. Moreover
$1+1/r(p)=1/s(p)+1/p$, so Young's inequality gives
$\||\Knear|\ast|\omega|\|_{L^{r(p)}(\R^2)}\le\|\Knear\|_{L^{s(p)}(\R^2)}\|\omega\|_{L^p(\R^2)}$.
Also $\||\Kfar|\ast|\omega|\|_{L^\infty(\R^2)}\le\|\omega\|_{L^1(\R^2)}/(2\pi)$. Hence
$|\KBS|\ast|\omega|$ is finite almost everywhere. Since
$r(p)\ge\max\{2,p'\}$, H\"older's inequality on $B_R$ gives the velocity bound.
For the last bound, $r(p)'\in[1,p]$, so
$\|\omega\|_{L^{r(p)'}(\R^2)}\le\|\omega\|_{(L^1\cap L^p)(\R^2)}$ by interpolation, and
$\int|\omega|(|\KBS|\ast|\omega|)\le\|\omega\|_{L^{r(p)'}(\R^2)}\||\Knear|\ast|\omega|\|_{L^{r(p)}(\R^2)}
+\|\omega\|_{L^1(\R^2)}^2/(2\pi)$.

(b) For $L>0$ and $g\in L^\infty(B_L)$, the function $g\mathbf 1_{B_L}$ lies in
$L^{p'}(\R^2)$, so $|\int_{B_L}\omega g|=\lim_n|\int_{B_L}\omega_ng|\le
M\|g\|_{L^\infty(\R^2)}$. The limit $\omega$ is real, as a weak limit of real functions.
Taking $g=\sgn\omega$ (with $\sgn s=s/|s|$ for $s\ne0$ and $\sgn0=0$) and
letting $L\to\infty$ gives $\|\omega\|_{L^1(\R^2)}\le M$.

By hypothesis and the bound just proved, $\omega_n$ and $\omega$ lie in
$(L^1\cap L^p)(\R^2)$, with $L^1(\R^2)$ norms at most $M$. Put
$N=\sup_n\|\omega_n\|_{L^p(\R^2)}+\|\omega\|_{L^p(\R^2)}<\infty$ (weakly convergent
sequences are bounded), let $0<\varepsilon<1<L$, and split
$\KBS=K^{(1)}+K^{(2)}+K^{(3)}$ with $K^{(1)}=\KBS\mathbf 1_{B_\varepsilon}$,
$K^{(2)}=\KBS\mathbf 1_{B_L\setminus B_\varepsilon}$ and
$K^{(3)}=\KBS\mathbf 1_{\R^2\setminus B_L}$. By Young's inequality, as in (a), and
H\"older's inequality on $B_R$,
$\|K^{(1)}\ast(\omega_n-\omega)\|_{L^{p'}(B_R)}\le|B_R|^{1/p'-1/r(p)}\|K^{(1)}\|_{L^{s(p)}(\R^2)}N$,
and $\|K^{(1)}\|_{L^{s(p)}(\R^2)}\to0$ as $\varepsilon\to0$ because $s(p)<2$. Next,
$|K^{(3)}\ast(\omega_n-\omega)|\le2M/(2\pi L)$ everywhere, since
$\|\omega_n\|_{L^1(\R^2)},\|\omega\|_{L^1(\R^2)}\le M$; so its $L^{p'}(B_R)$ norm is at most
$|B_R|^{1/p'}M/(\pi L)$, uniformly in $n$. Finally, for
$x\in B_R$ the function $K^{(2)}(x-\cdot)$ is bounded by $(2\pi\varepsilon)^{-1}$,
since $|\KBS(z)|=(2\pi|z|)^{-1}$ and $|z|\ge\varepsilon$ on the support of
$K^{(2)}$, and vanishes outside $\{y:|x-y|<L\}$; hence it lies in $L^{p'}(\R^2)$, so $K^{(2)}\ast\omega_n(x)\to K^{(2)}\ast\omega(x)$, with
$|K^{(2)}\ast\omega_n(x)|\le\|K^{(2)}\|_{L^{p'}(\R^2)}N$; dominated convergence gives
$K^{(2)}\ast\omega_n\to K^{(2)}\ast\omega$ in $L^{p'}(B_R)$. Letting first
$n\to\infty$ and then $\varepsilon\to0$ and $L\to\infty$ gives (b).

(c) Choose $\omega_j\in C_c^\infty(\R^2)$ with $\omega_j\to\omega$ in $L^1(\R^2)$ and in
$L^2(\R^2)$, and put $v_j=\KBS\ast\omega_j$. By Lemma~\ref{lem:BS-basic}(d), applied
to $\omega_j-\omega_i$ and to $\omega_j$, $\|\nabla v_j-\nabla v_i\|_{L^2(\R^2)}=\|\omega_j-\omega_i\|_{L^2(\R^2)}$
and $\|\nabla v_j\|_{L^2(\R^2)}=\|\omega_j\|_{L^2(\R^2)}$, so $\nabla v_j\to F$ in $L^2(\R^2)$ with
$\|F\|_{L^2(\R^2)}=\lim_j\|\omega_j\|_{L^2(\R^2)}=\|\omega\|_{L^2(\R^2)}$. By
Lemma~\ref{lem:BS-basic}(e), $v_j\to v$ in $L^1_{\mathrm{loc}}(\R^2)$, so
$\nabla v_j\to\nabla v$ in $\mathcal D'$, and $\nabla v=F$.
\end{proof}

\begin{definition}
\label{def:multipole}
Let $m\ge1$ be an integer and let $g\in L^1(\R^2;\C)$ have compact support in
$\R^2\setminus\{0\}$. Writing $y\in\R^2$ as a complex number, we put
\[
\pout_m[g]=\frac1{2\pi}\int g(y)\,(\bar y)^m\dd y,\quad
\pin_m[g]=\frac1{2\pi}\int g(y)\,y^{-m}\dd y,\quad
\Qm[g]=\int g(y)\,y^{m}\dd y .
\]
\end{definition}

These are the multipole moments of $g$. In polar coordinates $y=se^{i\vartheta}$,
since $(\bar y)^m=s^me^{-im\vartheta}$, $y^{-m}=s^{-m}e^{-im\vartheta}$ and
$y^m=s^me^{im\vartheta}$, they read, with the angular modes $g_\ell$ of
Section~\ref{sec:prelim},
\[
\pout_m[g]=\int_0^\infty g_m(s)\,s^{m+1}\dd s,\qquad
\pin_m[g]=\int_0^\infty g_m(s)\,s^{1-m}\dd s,\qquad
\Qm[g]=2\pi\int_0^\infty g_{-m}(s)\,s^{m+1}\dd s ;
\]
so they see only the angular modes $\pm m$ of $g$. Let $g$ be real with
$\supp g\subset\{\rho_1\le|y|\le\rho_2\}$, $0<\rho_1\le\rho_2$, and let
$\psi(x)=\frac1{2\pi}\int\log|x-y|\,g(y)\dd y$ be its stream function. Expanding
$\log|x-y|=\log|x|+\Re\log(1-y/x)$ for $|x|>\rho_2$, and
$\log|x-y|=\log|y|+\Re\log(1-x/y)$ for $|x|<\rho_1$, in powers of $y/x$ and
$x/y$ gives
\[
\begin{aligned}
\psi(x)&=\frac{\int g}{2\pi}\log|x|-\Re\sum_{m\ge1}\frac{\pout_m[g]}m(\bar x)^{-m}&&(|x|>\rho_2),\\
\psi(x)&=\frac1{2\pi}\int g(y)\log|y|\dd y-\Re\sum_{m\ge1}\frac{\pin_m[g]}m\,x^m&&(|x|<\rho_1),
\end{aligned}
\]
since $\overline{\Qm[g]}=2\pi\pout_m[g]$ for real $g$. Thus $\pout_m[g]$ and
$\pin_m[g]$ give the field of $g$ in the angular modes $\pm m$ outside and inside its
support.

\subsection{Weak Solutions}

\begin{definition}
\label{def:equimeasurable}
Two real functions $g,h\in L^1(\R^2)$ are \textit{equimeasurable} if and only if
$|\{g>s\}|=|\{h>s\}|$ and $|\{g<-s\}|=|\{h<-s\}|$ for every $s>0$.
\end{definition}

\begin{definition}
\label{def:weak}
Let $4/3<q<\infty$, $T>0$ and $\omega^{\mathrm{in}}\in (L^1\cap L^q)(\R^2)$. A
\textit{weak solution with datum $\omega^{\mathrm{in}}$} in $X_q(T)$ is a function
$\omega\in X_q(T)$ such that, with $v(t)=\KBS\ast\omega(t)$,
\begin{equation}
\label{eq:weak}
\int_0^T\!\!\int\omega\,(\partial_t\varphi+v\cdot\nabla\varphi)\dd x\dd t
+\int\omega^{\mathrm{in}}\varphi(0,\cdot)\dd x=0
\qquad\text{for all }\varphi\in\DT .
\end{equation}
It is \textit{renormalized} if and only if, in addition, for every
$\Ren\in C^1(\R)$ with $\Ren'$ bounded and $\Ren(0)=0$,
\begin{equation}
\label{eq:renorm}
\int_0^T\!\!\int\Ren(\omega)\,(\partial_t\varphi+v\cdot\nabla\varphi)\dd x\dd t
+\int\Ren(\omega^{\mathrm{in}})\varphi(0,\cdot)\dd x=0
\qquad\text{for all }\varphi\in\DT .
\end{equation}
For $T_0\in(0,\infty]$, a (renormalized) weak solution \textit{on $[0,T_0)$} is a
function on $[0,T_0)$ whose restriction to $[0,T)$ is a (renormalized) weak
solution in $X_q(T)$ for every $T<T_0$.
\end{definition}

The integrals in Equations~\eqref{eq:weak} and~\eqref{eq:renorm} converge
absolutely, because $|\Ren(\omega)|\le\|\Ren'\|_{L^\infty(\R)}|\omega|$ and
$|\omega v|\le|\omega|(|\KBS|\ast|\omega|)$, and Lemma~\ref{lem:velocity}(a),
with $p=q$ and $R=1$, gives
$\|\omega(t)v(t)\|_{L^1(\R^2)}\le C(q,1)\operatorname{ess\,sup}_{0<s<T}\|\omega(s)\|_{(L^1\cap L^q)(\R^2)}^2$
for almost every $t\in(0,T)$.

\section{Compactness and a Non-Uniqueness Criterion}
\label{sec:compact}

Compactness results for sequences of renormalized solutions go back to DiPerna and Lions
\cite{DiPernaLions}; we give a self-contained proof of the one we need. We prove
it only for vorticity in $L^\infty_t(L^1\cap L^2)_x$, which is all that the proof
of Theorem~\ref{thm:main} needs. The arguments control $\nabla v$ by $\omega$:
for vorticity in $L^2(\R^2)$ this is the identity $\|\nabla v\|_{L^2(\R^2)}=\|\omega\|_{L^2(\R^2)}$
of Lemma~\ref{lem:velocity}(c), while for vorticity in $L^q(\R^2)$, $q\ne2$, it would
require the $L^q(\R^2)$ bounds of the Calder\'on-Zygmund (Riesz transform) estimate. Equations~\eqref{eq:weak}
and~\eqref{eq:renorm} do not involve $q$, so, for $q\in(4/3,\infty)$, a
renormalized weak solution in $X_2(T)$ that lies in $X_q(T)$, with datum in $(L^1\cap L^q)(\R^2)$, is a
renormalized weak solution in $X_q(T)$; Theorem~\ref{thm:criterion}(ii) gives
this membership.

All functions in this section are real valued, except $\Psi$ in
Corollary~\ref{cor:defect}. Let $\varsigma\in C_c^\infty(B_2)$ with
$0\le\varsigma\le1$ and $\varsigma=1$ on $B_1$, and put
$\varsigma_R=\varsigma(\cdot/R)$, so that
$|\nabla\varsigma_R|\le\|\nabla\varsigma\|_{L^\infty(\R^2)}/R$. Let
$\rho\in C_c^\infty(B_1)$ and $\rho^{\mathrm t}\in C_c^\infty((-1,1))$ be
nonnegative with integral $1$, and put
$\rho_\varepsilon=\varepsilon^{-2}\rho(\cdot/\varepsilon)$ and
$\rho^{\mathrm t}_\delta=\delta^{-1}\rho^{\mathrm t}(\cdot/\delta)$; convolution
with $\rho_\varepsilon$ acts in the space variable only, and convolution with
$\rho^{\mathrm t}_\delta$ in the time variable only. Let
$\varkappa\in C^\infty(\R)$ be nondecreasing with $\varkappa=0$ on
$(-\infty,0]$ and $\varkappa=1$ on $[1,\infty)$. We write $C_{\mathrm v}$ for
the constant $C(2,1)$ of Lemma~\ref{lem:velocity}(a); thus
$\int|g|(|\KBS|\ast|g|)\dd x\le C_{\mathrm v}\|g\|_{(L^1\cap L^2)(\R^2)}^2$ for
$g\in (L^1\cap L^2)(\R^2)$.

\subsection{The Linear Transport Equation}

Let $T>0$ and $t_0\in[0,T)$, let $w$ be a real measurable vector field and $u$ a
real measurable function on $(t_0,T)\times\R^2$, and let
$u_0\in L^1_{\mathrm{loc}}(\R^2)$ be real, such that $u$, $w$ and $uw$ are
integrable on $(t_0,T')\times B_R$ for all $T'<T$ and $R>0$. We say that $u$
satisfies Equation~\eqref{eq:te} on $[t_0,T)$ with field $w$ and datum $u_0$ if
\begin{equation}
\label{eq:te}
\int_{t_0}^T\!\!\int u\,(\partial_t\varphi+w\cdot\nabla\varphi)\dd x\dd t
+\int u_0\,\varphi(t_0,\cdot)\dd x=0
\qquad\text{for all }\varphi\in\DT ;
\end{equation}
every integral here converges. Thus $\omega\in X_2(T)$ is a weak solution of the
Euler equations with datum
$\omega^{\mathrm{in}}\in (L^1\cap L^2)(\R^2)$ if and only if it satisfies
Equation~\eqref{eq:te} on $[0,T)$ with field $v=\KBS\ast\omega$ and datum
$\omega^{\mathrm{in}}$, and it is renormalized if and only if, in addition,
$\Ren(\omega)$ satisfies Equation~\eqref{eq:te} on $[0,T)$ with field $v$ and
datum $\Ren(\omega^{\mathrm{in}})$ for every $\Ren\in C^1(\R)$ with $\Ren'$
bounded and $\Ren(0)=0$. The integrability conditions hold by
Lemma~\ref{lem:velocity}(a).

\begin{lemma}
\label{lem:dbr}
Let $t_0<T$, $B\in L^\infty(t_0,T)$ and $b_0\in\R$, and suppose that, for
every $\alpha\in C_c^\infty((-\infty,T))$,
$\int_{t_0}^T\alpha'(t)B(t)\dd t+\alpha(t_0)b_0=0$. Then $B=b_0$ almost
everywhere on $(t_0,T)$.
\end{lemma}

\begin{proof}
For $0<\delta<(T-t_0)/2$ and $t\in(t_0+\delta,T-\delta)$ put
$B_\delta(t)=\int B(s)\rho^{\mathrm t}_\delta(t-s)\dd s$. The function
$\alpha(s)=\rho^{\mathrm t}_\delta(t-s)$ vanishes outside $(t-\delta,t+\delta)$,
which lies in $(t_0,T)$ because $t-t_0>\delta$ and $T-t>\delta$; so
$\alpha\in C_c^\infty((t_0,T))$ and $\alpha(t_0)=0$, and the hypothesis gives
$B_\delta'(t)=-\int B(s)\,\partial_s\big(\rho^{\mathrm t}_\delta(t-s)\big)\dd s=0$,
so $B_\delta$ is constant on $(t_0+\delta,T-\delta)$. Since $B_\delta\to B$ in
$L^1_{\mathrm{loc}}((t_0,T))$ as $\delta\to0$, the constants converge to some
$c\in\R$ and $B=c$ almost everywhere. For $\alpha$ with $\alpha(t_0)=1$ we have
$\int_{t_0}^T\alpha'c\dd t=-c$, because $\alpha$ vanishes near $T$; so the
hypothesis gives $b_0=c$.
\end{proof}

\begin{lemma}
\label{lem:weak-cont}
Let $T>0$. Let $w$ be a vector field
with $w\in L^\infty(0,T;L^2(B_R))$ for every $R>0$, let
$u\in L^\infty(0,T;L^2(\R^2))$ and $u_0\in L^2(\R^2)$, and suppose that
$u$ satisfies Equation~\eqref{eq:te} on $[0,T)$ with field $w$ and datum $u_0$.
Put $N=\operatorname{ess\,sup}_t\|u(t)\|_{L^2(\R^2)}$.
Then $u$ has a representative $\tilde u\colon[0,T]\to L^2(\R^2)$ with the following
properties.
\begin{enumerate}[label=\textup{(\alph*)}]
\item $\tilde u$ is weakly continuous, $\tilde u(0)=u_0$, and
  $\|\tilde u(t)\|_{L^2(\R^2)}\le N$ for every $t$. If also
  $u\in L^\infty(0,T;L^1(\R^2))$, then
  $\|\tilde u(t)\|_{L^1(\R^2)}\le\operatorname{ess\,sup}_s\|u(s)\|_{L^1(\R^2)}$ for every $t$.
\item For $\psi\in C_c^\infty(\R^2)$ with $\supp\psi\subset B_R$ and every
  $t\in[0,T]$,
  $\int\tilde u(t)\psi=\int u_0\psi+\int_0^t\!\int u\,w\cdot\nabla\psi$. In
  particular, for $s,t\in[0,T]$,
  \[
  \Big|\int(\tilde u(t)-\tilde u(s))\psi\Big|
  \le N\,\|w\|_{L^\infty(0,T;L^2(B_R))}\|\nabla\psi\|_{L^\infty(\R^2)}|t-s| .
  \]
\item For every $t_0\in[0,T)$, $u$ satisfies Equation~\eqref{eq:te} on $[t_0,T)$ with
  field $w$ and datum $\tilde u(t_0)$.
\end{enumerate}
\end{lemma}

\begin{proof}
For $j\in\N$, $C_c^\infty(B_j)$ is a subset of the separable Banach space
$C^1(\overline{B_j})$, hence separable: there is a countable set
$\mathcal F_j\subset C_c^\infty(B_j)$ such that every $\psi\in C_c^\infty(B_j)$ is the
limit in $C^1$ of elements of $\mathcal F_j$. The countable union
$\mathcal F=\bigcup_{j\in\N}\mathcal F_j\subset C_c^\infty(\R^2)$ does not depend on $j$,
and every $\psi\in C_c^\infty(B_j)$, $j\in\N$, is the limit in $C^1$ of elements of
$\mathcal F\cap C_c^\infty(B_j)$.
$C_c^\infty$ is dense in $L^2(\R^2)$, and convergence in
$C^1$ with supports in a fixed $B_j$ implies convergence in $L^2(\R^2)$; so
$\mathcal F$ is dense in $L^2(\R^2)$. Write
$W_R=\|w\|_{L^\infty(0,T;L^2(B_R))}$.

For $\psi\in\mathcal F$ with $\supp\psi\subset B_R$ put
$a_\psi(t)=\int u(t)\psi$ and $c_\psi(t)=\int u(t)\,w(t)\cdot\nabla\psi$. These
are measurable by Fubini's theorem, and $|c_\psi|\le NW_R\|\nabla\psi\|_{L^\infty(\R^2)}$
almost everywhere by H\"older's inequality, so
$A_\psi(t)=\int u_0\psi+\int_0^tc_\psi$ is Lipschitz on $[0,T]$. For
$\alpha\in C_c^\infty((-\infty,T))$, Equation~\eqref{eq:te} with the test
function $\alpha(t)\psi(x)$ gives
$\int_0^T\alpha'a_\psi+\int_0^T\alpha c_\psi+\alpha(0)\int u_0\psi=0$. Since
$\alpha$ vanishes near $T$ and $A_\psi'=c_\psi$ almost everywhere, integration
by parts gives
$\int_0^T\alpha c_\psi=-\int_0^T\alpha'A_\psi-\alpha(0)\int u_0\psi$. Hence
$\int_0^T\alpha'(a_\psi-A_\psi)=0$, and Lemma~\ref{lem:dbr}, with $t_0=0$ and
$b_0=0$, gives $a_\psi=A_\psi$ almost everywhere. Let $\mathcal N$ be the union
of $\{0,T\}$, of the exceptional null sets of these identities for
$\psi\in\mathcal F$, of the set of $t$ with $u(t)\notin L^2(\R^2)$ or
$\|u(t)\|_{L^2(\R^2)}>N$, and, if $u\in L^\infty(0,T;L^1(\R^2))$, of the set of $t$ with
$\|u(t)\|_{L^1(\R^2)}>N_1:=\operatorname{ess\,sup}_s\|u(s)\|_{L^1(\R^2)}$. As $\mathcal F$
is countable, $\mathcal N$ is a null set.

For $t\notin\mathcal N$ put $\tilde u(t)=u(t)$. For $t\in\mathcal N$ take
$t_j\notin\mathcal N$ with $t_j\to t$. Then $\|u(t_j)\|_{L^2(\R^2)}\le N$ and
$\int u(t_j)\psi=A_\psi(t_j)\to A_\psi(t)$ for $\psi\in\mathcal F$. Since
$\mathcal F$ is dense in $L^2(\R^2)$, $\int u(t_j)\phi$ converges for every
$\phi\in L^2(\R^2)$, to a linear functional of norm at most $N$; by the Riesz
representation theorem in $L^2(\R^2)$, it is $\phi\mapsto\int\tilde u(t)\phi$ for a unique
$\tilde u(t)\in L^2(\R^2)$, which does not depend on the sequence, because
$\int\tilde u(t)\psi=A_\psi(t)$ for $\psi\in\mathcal F$. Thus $\tilde u=u$ almost
everywhere, $\|\tilde u(t)\|_{L^2(\R^2)}\le N$ for every $t$, and
$\int\tilde u(t)\psi=A_\psi(t)$ for all $t\in[0,T]$ and $\psi\in\mathcal F$.

We prove (a) and (b). For $\phi\in L^2(\R^2)$, $\psi\in\mathcal F$ and $s,t\in[0,T]$,
$|\int(\tilde u(t)-\tilde u(s))\phi|\le2N\|\phi-\psi\|_{L^2(\R^2)}+|A_\psi(t)-A_\psi(s)|$.
Given $\phi$ and $\eta>0$, first choose $\psi\in\mathcal F$ with
$2N\|\phi-\psi\|_{L^2(\R^2)}<\eta/2$, which is possible as $\mathcal F$ is dense in $L^2(\R^2)$;
then, since $A_\psi$ is Lipschitz, $|A_\psi(t)-A_\psi(s)|<\eta/2$ whenever
$|t-s|$ is small enough. So $t\mapsto\int\tilde u(t)\phi$ is continuous for every
$\phi\in L^2(\R^2)$, that is, $\tilde u$ is weakly continuous; and $\tilde u(0)=u_0$, because
$A_\psi(0)=\int u_0\psi$. If $u\in L^\infty(0,T;L^1(\R^2))$ and $t\in\mathcal N$, then
$u(t_j)\rightharpoonup\tilde u(t)$ weakly in $L^2(\R^2)$ and $\|u(t_j)\|_{L^1(\R^2)}\le N_1$,
so $\|\tilde u(t)\|_{L^1(\R^2)}\le N_1$ by Lemma~\ref{lem:velocity}(b) with $p=2$. For
(b), let $\psi\in C_c^\infty(B_j)$ and $\psi_i\in\mathcal F\cap C_c^\infty(B_j)$
with $\psi_i\to\psi$ in $C^1$. Then $\int\tilde u(t)\psi_i\to\int\tilde u(t)\psi$,
since $\|\psi_i-\psi\|_{L^2(\R^2)}\le|B_j|^{1/2}\|\psi_i-\psi\|_{L^\infty(\R^2)}$, and
$A_{\psi_i}(t)\to\int u_0\psi+\int_0^t\!\int u\,w\cdot\nabla\psi$, since
$\int u_0\psi_i\to\int u_0\psi$ and
$|\int_0^t\!\int u\,w\cdot\nabla(\psi_i-\psi)|\le TNW_j\|\nabla(\psi_i-\psi)\|_{L^\infty(\R^2)}$. This is the
identity in (b), and the bound on $c_\psi$ gives the Lipschitz bound.

For (c), the case $t_0=0$ is the hypothesis, since $\tilde u(0)=u_0$. Let $t_0\in(0,T)$ and
$\varphi\in\DT$, and for $0<h<T-t_0$ put $\varkappa_h(t)=\varkappa((t-t_0)/h)$.
Equation~\eqref{eq:te} with the test function $\varkappa_h\varphi\in\DT$, which
vanishes for $t\le t_0$, gives
\[
\int_0^T\!\!\int\varkappa_h\,u\,(\partial_t\varphi+w\cdot\nabla\varphi)\dd x\dd t
+\int_0^T\varkappa_h'(t)\,g(t)\dd t=0,\qquad g(t)=\int\tilde u(t)\varphi(t)\dd x .
\]
The function $g$ is continuous, since
$|g(t)-g(s)|\le N\|\varphi(t)-\varphi(s)\|_{L^2(\R^2)}+|\int(\tilde u(t)-\tilde u(s))\varphi(s)|$
and $\tilde u$ is weakly continuous. As $h\to0$, the second term tends to $g(t_0)$: since $\varkappa_h'\ge0$,
$\int\varkappa_h'=1$ and $\supp\varkappa_h'\subset[t_0,t_0+h]$, the functions
$\varkappa_h'$ converge to the Dirac mass at $t_0$ in the sense of distributions,
indeed against every continuous function, and $g$ is continuous; and, since
$\varkappa_h\to\mathbf 1_{(t_0,\infty)}$ pointwise, the first term tends to $\int_{t_0}^T\!\int u\,(\partial_t\varphi+w\cdot\nabla\varphi)$ by
dominated convergence, with the integrable dominating function
$|u|(|\partial_t\varphi|+|w||\nabla\varphi|)$. This is (c).
\end{proof}

\subsection{Weakly Continuous Representatives}
\label{sec:weak-rep}

Let $\omega$ be a weak solution in $X_2(T)$ with datum
$\omega^{\mathrm{in}}$, and put $v=\KBS\ast\omega$. By
Lemma~\ref{lem:velocity}(a), $g\mapsto\KBS\ast g$ is linear and bounded from
$(L^1\cap L^2)(\R^2)$ to $L^2(B_R)$, and such a map preserves strong
measurability; so $v\in L^\infty(0,T;L^2(B_R))$ for every
$R>0$, and $\|\omega(t)v(t)\|_{L^1(\R^2)}\le C_{\mathrm v}\|\omega(t)\|_{(L^1\cap L^2)(\R^2)}^2$ for
almost every $t$, so $\omega v\in L^1((0,T)\times\R^2)$. Hence Lemma~\ref{lem:weak-cont} applies with
$u=\omega$, $w=v$ and $u_0=\omega^{\mathrm{in}}$. From now on every such weak
solution is identified with its weakly continuous representative, so that
$\omega(t)$ is defined for every $t\in[0,T]$, $\omega(0)=\omega^{\mathrm{in}}$,
and $\|\omega(t)\|_{L^1(\R^2)}$ and $\|\omega(t)\|_{L^2(\R^2)}$ are bounded by the
corresponding essential suprema for every $t$. Two weakly continuous functions
that agree for almost every $t$ agree for every $t$; hence the representatives
on the intervals $[0,T]$, $T<T_0$, of a weak solution on $[0,T_0)$ with $q=2$
define a weakly continuous representative on $[0,T_0)$. Moreover
$\|\nabla v(t)\|_{L^2(\R^2)}=\|\omega(t)\|_{L^2(\R^2)}$
by Lemma~\ref{lem:velocity}(c) and $\divop v(t)=0$ by
Lemma~\ref{lem:BS-basic}(e); thus $v\in L^\infty(0,T;W^{1,2}(B_R))$ for every
$R>0$, and $v$ is divergence-free.

\subsection{Renormalization}

For an open interval $I$, a vector field $w$ and a function $u$ on
$I\times\R^2$ with $u,wu\in L^1_{\mathrm{loc}}(I\times\R^2)$, put
\[
C_\varepsilon[u,w]=w\cdot\nabla(u\ast\rho_\varepsilon)-\big(\divop(wu)\big)\ast\rho_\varepsilon ,
\]
where $\big(\divop(wu)\big)\ast\rho_\varepsilon(t,x)=\int(wu)(t,y)\cdot\nabla\rho_\varepsilon(x-y)\dd y$.
Both terms of $C_\varepsilon$ are defined by these formulas, not derived: the
divergence $\divop(wu)$ is a distribution,
and its convolution in $x$ with $\rho_\varepsilon$, which acts by pairing it with
$\rho_\varepsilon(x-\cdot)$, is the function displayed above, since $wu\in L^1_{\mathrm{loc}}(I\times\R^2)$.
The function $u\ast\rho_\varepsilon(t,\cdot)$ is smooth, and
$w\cdot\nabla(u\ast\rho_\varepsilon)$ is its pointwise product with $w$. Writing
$w(t,x)\cdot\nabla(u\ast\rho_\varepsilon)(t,x)=\int u(t,y)\,w(t,x)\cdot\nabla\rho_\varepsilon(x-y)\dd y$
and substituting $y=x-\varepsilon z$, we get, for almost every $(t,x)$,
\begin{equation}
\label{eq:comm-formula}
\begin{aligned}
C_\varepsilon[u,w](t,x)&=\int u(t,y)\,\big(w(t,x)-w(t,y)\big)\cdot\nabla\rho_\varepsilon(x-y)\dd y\\
&=\int_{B_1}u(t,x-\varepsilon z)\,\frac{w(t,x)-w(t,x-\varepsilon z)}{\varepsilon}\cdot\nabla\rho(z)\dd z ,
\end{aligned}
\end{equation}
because $\nabla\rho_\varepsilon(\varepsilon z)=\varepsilon^{-3}\nabla\rho(z)$. The map
$(u,w)\mapsto C_\varepsilon[u,w]$ is bilinear.

\begin{lemma}[cf. Lemma II.1 in \cite{DiPernaLions}]
\label{lem:commutator}
Let $I\subset\R$ be an open interval, let
$w\in L^1_{\mathrm{loc}}(I;W^{1,2}_{\mathrm{loc}}(\R^2;\R^2))$ with
$\divop w=0$, and let $u\in L^\infty_{\mathrm{loc}}(I;L^2_{\mathrm{loc}}(\R^2))$.
Then $C_\varepsilon[u,w]\to0$ in $L^1_{\mathrm{loc}}(I\times\R^2)$ as
$\varepsilon\to0$.
\end{lemma}

\begin{proof}
Note that $wu\in L^1_{\mathrm{loc}}(I\times\R^2)$, by the Cauchy-Schwarz
inequality in $x$. Fix a compact interval
$J\subset I$ and $R>0$, and write $\|\cdot\|_{q,r}$ for the norm of $L^q(B_r)$.

We first prove a bound that is uniform in $\varepsilon\in(0,1]$ and does not use
$\divop w=0$. If $w(t,\cdot)\in C^1(\R^2;\R^2)$, then
$w(x)-w(x-\varepsilon z)=\varepsilon\int_0^1\nabla w(x-s\varepsilon z)z\dd s$,
and Equation~\eqref{eq:comm-formula} with $|z|\le1$ and the Cauchy-Schwarz
inequality in $x$ give
\begin{align*}
\|C_\varepsilon[u,w](t)\|_{1,R}
&\le\|\nabla\rho\|_{L^\infty(\R^2)}\int_{B_1}\!\int_0^1\!\int_{B_R}
|u(t,x-\varepsilon z)|\,|\nabla w(t,x-s\varepsilon z)|\dd x\dd s\dd z\\
&\le\pi\|\nabla\rho\|_{L^\infty(\R^2)}\|u(t)\|_{2,R+1}\|\nabla w(t)\|_{2,R+1}.
\end{align*}
For general $w$, apply this to $w_\delta=w\ast\rho_\delta$, $\delta\in(0,1]$,
for which $\|\nabla w_\delta(t)\|_{2,R+1}\le\|\nabla w(t)\|_{2,R+2}$ by Young's
inequality. For fixed $\varepsilon$, the first line of
Equation~\eqref{eq:comm-formula} and the inequalities of H\"older and Young,
with $\|\nabla\rho_\varepsilon\|_{L^1(\R^2)}=\varepsilon^{-1}\|\nabla\rho\|_{L^1(\R^2)}$,
give
\[
\|C_\varepsilon[u,w-w_\delta](t)\|_{1,R}\le2\varepsilon^{-1}\|\nabla\rho\|_{L^1(\R^2)}
\|u(t)\|_{2,R+1}\|(w-w_\delta)(t)\|_{2,R+1},
\]
which tends to $0$ as
$\delta\to0$ for almost every $t$, because mollification converges in
$L^2_{\mathrm{loc}}(\R^2)$. Hence, for almost every $t\in J$ and all
$\varepsilon\in(0,1]$,
\begin{equation}
\label{eq:comm-bound}
\|C_\varepsilon[u,w](t)\|_{1,R}\le C_\rho\,\|u(t)\|_{2,R+1}\,\|\nabla w(t)\|_{2,R+2},
\qquad C_\rho=\pi\|\nabla\rho\|_{L^\infty(\R^2)}.
\end{equation}

We now use $\divop w=0$. For $\delta\in(0,1]$ put $u_\delta(t)=u(t)\ast\rho_\delta$.
For almost every $t\in J$, $u_\delta(t)$ is smooth with
$\|\nabla u_\delta(t)\|_{\infty,R+1}\le\|\nabla\rho_\delta\|_{L^\infty(\R^2)}\|u(t)\|_{1,R+2}$,
and, for $\phi\in C_c^\infty(\R^2)$,
$\int wu_\delta\cdot\nabla\phi=\int w\cdot\nabla(u_\delta\phi)-\int\phi\,w\cdot\nabla u_\delta=-\int\phi\,w\cdot\nabla u_\delta$,
because $\divop w(t)=0$. Thus $\divop(wu_\delta)=w\cdot\nabla u_\delta$, and
\[
C_\varepsilon[u_\delta,w]=w\cdot\big((\nabla u_\delta)\ast\rho_\varepsilon-\nabla u_\delta\big)
+\big(w\cdot\nabla u_\delta-(w\cdot\nabla u_\delta)\ast\rho_\varepsilon\big).
\]
As $\varepsilon\to0$, for almost every $t\in J$ the first term tends to $0$ in
$L^1(B_R)$, because $\nabla u_\delta(t)$ is uniformly continuous on
$\overline{B_{R+1}}$, and so does the second, because
$w\cdot\nabla u_\delta(t)\in L^1(B_{R+1})$. Since
$\|C_\varepsilon[u_\delta,w](t)\|_{1,R}\le4\|w(t)\|_{1,R+1}\|\nabla u_\delta(t)\|_{\infty,R+1}$,
which is integrable on $J$ because $w\in L^1(J;L^1(B_{R+1}))$ and
$u\in L^\infty(J;L^1(B_{R+2}))$, $C_\varepsilon[u_\delta,w]\to0$ in
$L^1(J\times B_R)$ for each $\delta$. Finally, by bilinearity and
Equation~\eqref{eq:comm-bound},
\[
\sup_{\varepsilon\le1}\|C_\varepsilon[u,w]-C_\varepsilon[u_\delta,w]\|_{L^1(J\times B_R)}
\le C_\rho\int_J\|(u-u_\delta)(t)\|_{2,R+1}\|\nabla w(t)\|_{2,R+2}\dd t ,
\]
Call the right side $\eta(\delta)$; it does not depend on $\varepsilon$, and
$\eta(\delta)\to0$ as $\delta\to0$ by dominated convergence, since
$\|(u-u_\delta)(t)\|_{2,R+1}\to0$ for almost every $t$ and is at most
$2\|u\|_{L^\infty(J;L^2(B_{R+2}))}$. Thus
$\|C_\varepsilon[u,w]\|_{L^1(J\times B_R)}\le\|C_\varepsilon[u_\delta,w]\|_{L^1(J\times B_R)}+\eta(\delta)$
for all $\varepsilon\in(0,1]$ and $\delta\in(0,1]$. Letting first $\varepsilon\to0$
with $\delta$ fixed gives
$\limsup_{\varepsilon\to0}\|C_\varepsilon[u,w]\|_{L^1(J\times B_R)}\le\eta(\delta)$, and
letting then $\delta\to0$ gives $C_\varepsilon[u,w]\to0$ in $L^1(J\times B_R)$.
\end{proof}

\begin{lemma}
\label{lem:transport}
Let $T>0$ and $t_0\in[0,T)$. Let $w\in L^\infty(0,T;W^{1,2}_{\mathrm{loc}})$
with $\divop w=0$, let $u\in X_2(T)$ with
$uw\in L^1((0,T)\times\R^2)$, and let $u_0\in (L^1\cap L^2)(\R^2)$. Assume that $u$
satisfies Equation~\eqref{eq:te} on $[t_0,T)$ with field $w$ and datum $u_0$.
Then, for every $\Ren\in\LipO$:
\begin{enumerate}[label=\textup{(\roman*)}]
\item $\Ren(u)$ satisfies Equation~\eqref{eq:te} on $[t_0,T)$ with field $w$
  and datum $\Ren(u_0)$;
\item $\int\Ren(u(t))\dd x=\int\Ren(u_0)\dd x$ for almost every $t\in(t_0,T)$.
\end{enumerate}
\end{lemma}

\begin{proof}
We first prove (i) for $\Ren\in C^1(\R)$ with $|\Ren'|\le L$ and $\Ren(0)=0$.
Put $I=(t_0-1,T)$, and redefine $w$ and $u$ on $(t_0-1,t_0]$ by $w=0$ and
$u=u_0$; on $(t_0,T)$ they are unchanged. For $\varphi\in C_c^\infty(I\times\R^2)$, extended by zero to an
element of $\DT$, the interval $(t_0-1,t_0)$ contributes
$\int_{t_0-1}^{t_0}\!\int u_0\partial_t\varphi=\int u_0\varphi(t_0)$ to
$\int_I\!\int u(\partial_t\varphi+w\cdot\nabla\varphi)$, and by
Equation~\eqref{eq:te} the interval $(t_0,T)$ contributes
$-\int u_0\varphi(t_0)$. Thus $\partial_tu+\divop(wu)=0$ in
$\mathcal D'(I\times\R^2)$, and the extended pair satisfies the hypotheses of
Lemma~\ref{lem:commutator} on $I$.

Let $u_\varepsilon=u\ast\rho_\varepsilon$. Then $\nabla u_\varepsilon\in
L^\infty_{\mathrm{loc}}(I\times\R^2)$, and convolving the equation
$\partial_tu+\divop(wu)=0$ in $x$ with $\rho_\varepsilon$ gives
$\partial_tu_\varepsilon=-\big(\divop(wu)\big)\ast\rho_\varepsilon\in L^1_{\mathrm{loc}}(I\times\R^2)$, where,
as $wu\in L^1_{\mathrm{loc}}(I\times\R^2)$, the right side is the function defined at the
start of this subsection,
so $u_\varepsilon\in W^{1,1}_{\mathrm{loc}}(I\times\R^2)$ and
$\partial_tu_\varepsilon+w\cdot\nabla u_\varepsilon=C_\varepsilon[u,w]$. For
almost every $t$, $\Ren(u_\varepsilon(t))$ is a $C^1$ function of $x$ with
locally bounded gradient, so
$\divop\big(w\Ren(u_\varepsilon)\big)=w\cdot\nabla\Ren(u_\varepsilon)$ because
$\divop w=0$; with the chain rule in $W^{1,1}_{\mathrm{loc}}$ this gives
\[
\partial_t\Ren(u_\varepsilon)+\divop\big(w\Ren(u_\varepsilon)\big)=\Ren'(u_\varepsilon)C_\varepsilon[u,w]
\qquad\text{in }\mathcal D'(I\times\R^2).
\]
As $\varepsilon\to0$, the right side tends to $0$ in $L^1_{\mathrm{loc}}(I\times\R^2)$ by
Lemma~\ref{lem:commutator}, since $|\Ren'|\le L$. For almost every $t$, $u_\varepsilon(t)\to u(t)$ in
$L^2(B_R)$, with the bound $2\|u\|_{L^\infty(I;L^2(\R^2))}$; since
$|\Ren(u_\varepsilon)-\Ren(u)|\le L|u_\varepsilon-u|$ and
$w\in L^\infty(I;L^2(B_R))$, dominated convergence gives
$\Ren(u_\varepsilon)\to\Ren(u)$ and $w\Ren(u_\varepsilon)\to w\Ren(u)$ in
$L^1_{\mathrm{loc}}(I\times\R^2)$. Therefore $\partial_t\Ren(u)+\divop\big(w\Ren(u)\big)=0$
in $\mathcal D'(I\times\R^2)$. Now let $\varphi\in\DT$ and
$\kappa_0(t)=\varkappa\big(3(t-t_0)+2\big)$, so that $\kappa_0=0$ on
$(-\infty,t_0-2/3]$ and $\kappa_0=1$ on $[t_0-1/3,\infty)$. Test the last
equation with $\kappa_0\varphi\in C_c^\infty(I\times\R^2)$. On $(t_0-1,t_0)$,
$w=0$ and $\Ren(u)=\Ren(u_0)$, so this interval contributes
$\int_{t_0-1}^{t_0}\!\int\Ren(u_0)\partial_t(\kappa_0\varphi)=\int\Ren(u_0)\varphi(t_0)$,
while $\kappa_0=1$ on $(t_0,T)$. This is (i) for $\Ren$.

Let now $\Ren\in\LipO$ have Lipschitz constant $L$, and put
$\Ren_j=\Ren\ast\rho^{\mathrm t}_{1/j}-(\Ren\ast\rho^{\mathrm t}_{1/j})(0)$.
Then $\Ren_j\in C^\infty(\R)$, $|\Ren_j'|\le L$, $\Ren_j(0)=0$, and
$|\Ren_j-\Ren|\le2L/j$, because $|\Ren\ast\rho^{\mathrm t}_{1/j}-\Ren|\le L/j$
and $\Ren(0)=0$. Since $|\Ren(u)|\le L|u|$, all integrals in
Equation~\eqref{eq:te} for $\Ren(u)$ converge. For $\varphi\in\DT$ with
$\supp\varphi\subset(-\infty,T']\times B_R$, $T'<T$, the left sides of
Equation~\eqref{eq:te} for $\Ren_j(u)$ (which vanish) and for $\Ren(u)$ differ
by at most
$\frac{2L}{j}\big(\int_{t_0}^{T'}\!\int_{B_R}(|\partial_t\varphi|+|w||\nabla\varphi|)+\int_{B_R}|\varphi(t_0)|\big)$,
which tends to $0$ because $w$ is integrable on $(t_0,T')\times B_R$. This
proves (i).

For (ii), let $\Ren\in\LipO$ have Lipschitz constant $L$, and let
$\alpha\in C_c^\infty((-\infty,T))$. Since $\Ren(u)$ solves
Equation~\eqref{eq:te} by (i), we test it against $\alpha(t)\varsigma_R(x)$ and get
\[
\int_{t_0}^T\alpha'(t)\Big(\int\Ren(u(t))\varsigma_R\dd x\Big)\dd t
+\alpha(t_0)\int\Ren(u_0)\varsigma_R\dd x
=-\int_{t_0}^T\alpha(t)\int\Ren(u)\,w\cdot\nabla\varsigma_R\dd x\dd t ,
\]
whose right side is at most
$\|\alpha\|_{L^\infty(\R)}L\|\nabla\varsigma\|_{L^\infty(\R^2)}R^{-1}\|uw\|_{L^1((0,T)\times\R^2)}$.
As $R\to\infty$, dominated convergence, with $|\Ren(u)|\le L|u|$ and
$u\in L^\infty(0,T;L^1(\R^2))$, gives $\int_{t_0}^T\alpha'B+\alpha(t_0)b_0=0$ with
$B(t)=\int\Ren(u(t))$, a measurable function bounded by
$L\|u\|_{L^\infty(0,T;L^1(\R^2))}$, and $b_0=\int\Ren(u_0)$. Lemma~\ref{lem:dbr} gives
(ii).
\end{proof}

\begin{lemma}
\label{lem:renorm-all-t}
Let $T>0$ and let $\omega$ be a renormalized weak solution of the Euler
equations in $X_2(T)$ with datum $\omega^{\mathrm{in}}$. Then for every
$t\in[0,T)$ and every $\Ren\in\LipO$,
\begin{equation}
\label{eq:conservation}
\int\Ren(\omega(t))\dd x=\int\Ren(\omega^{\mathrm{in}})\dd x .
\end{equation}
Consequently, for every $t\in[0,T)$,
$\|\omega(t)\|_{L^q(\R^2)}=\|\omega^{\mathrm{in}}\|_{L^q(\R^2)}$ for every $q\in[1,2]$, and
$\omega(t)$ is equimeasurable with $\omega^{\mathrm{in}}$. Moreover
$\omega\in C([0,T);L^q(\R^2))$ for every $q\in[1,2]$, and, with
$v=\KBS\ast\omega$, for every $R>0$ and every $t\in[0,T)$,
\begin{equation}
\label{eq:tail}
\int_{\{|x|\ge2R\}}|\omega(t)|\dd x\le\int_{\{|x|\ge R\}}|\omega^{\mathrm{in}}|\dd x
+\frac{\|\nabla\varsigma\|_{L^\infty(\R^2)}}{R}\,\|\omega v\|_{L^1((0,T)\times\R^2)} ,
\end{equation}
where $\|\omega v\|_{L^1((0,T)\times\R^2)}\le C_{\mathrm v}TM^2$ with
$M=\sup_{t\in[0,T]}\|\omega(t)\|_{(L^1\cap L^2)(\R^2)}$.
\end{lemma}

\begin{proof}
The number $M$ is finite by Lemma~\ref{lem:weak-cont}(a), and
$\|\omega(t)v(t)\|_{L^1(\R^2)}\le C_{\mathrm v}M^2$ for almost every $t$ by
Lemma~\ref{lem:velocity}(a); this gives the bound on
$\|\omega v\|_{L^1((0,T)\times\R^2)}$. By the remark at the start of
Section~\ref{sec:weak-rep},
$v\in L^\infty(0,T;W^{1,2}(B_R))$ for every $R>0$ and $\divop v=0$.

We first prove Equation~\eqref{eq:conservation}. Let $t_0\in[0,T)$ and
$\Ren\in\LipO$. By Lemma~\ref{lem:weak-cont}(a) and (c),
$\omega(t_0)\in (L^1\cap L^2)(\R^2)$, and $\omega$ satisfies Equation~\eqref{eq:te} on
$[t_0,T)$ with field $v$ and datum $\omega(t_0)$, as well as on $[0,T)$ with
datum $\omega^{\mathrm{in}}$. Lemma~\ref{lem:transport}(ii), applied on
$[t_0,T)$ and on $[0,T)$ and compared at a time $t\in(t_0,T)$ that is
exceptional for neither, gives
$\int\Ren(\omega(t_0))=\int\Ren(\omega(t))=\int\Ren(\omega^{\mathrm{in}})$. For
$s,\delta>0$ the functions $z\mapsto\min\{1,(\pm z-s)_+/\delta\}$ lie in
$\LipO$ and increase to $\mathbf 1_{\{\pm z>s\}}$ as $\delta\downarrow0$, so
monotone convergence gives the equimeasurability. Integration of the
distribution function, $\int|g|^q=\int_0^\infty qs^{q-1}\big(|\{g>s\}|+|\{g<-s\}|\big)\dd s$,
$1\le q<\infty$, gives the norms.

Next, $\omega\in C([0,T);L^2(\R^2))$: the norm $\|\omega(t)\|_{L^2(\R^2)}$ is
constant on $[0,T)$, and $\omega$ is weakly continuous into $L^2(\R^2)$, so
$\|\omega(t)-\omega(s)\|_{L^2(\R^2)}^2=\|\omega(t)\|_{L^2(\R^2)}^2+\|\omega(s)\|_{L^2(\R^2)}^2-2\int\omega(t)\omega(s)\to0$
as $t\to s$. By Lemma~\ref{lem:transport}(i) with $\Ren=|\cdot|$, $|\omega|$
satisfies Equation~\eqref{eq:te} on $[0,T)$ with field $v$ and datum
$|\omega^{\mathrm{in}}|$. So Lemma~\ref{lem:weak-cont} applies with
$u=|\omega|$, $w=v$ and $u_0=|\omega^{\mathrm{in}}|$; since
$t\mapsto|\omega(t)|$ is continuous into $L^2(\R^2)$ on $[0,T)$, it agrees there with
the representative of that lemma, because two weakly continuous functions that
agree for almost every $t$ agree for every $t$. Part (b) of that lemma with
$\psi=\varsigma_R$, together with Equation~\eqref{eq:conservation} for
$\Ren=|\cdot|$, gives for every $t\in[0,T)$
\[
\int|\omega(t)|(1-\varsigma_R)\dd x
=\int|\omega^{\mathrm{in}}|(1-\varsigma_R)\dd x-\int_0^t\!\!\int|\omega|\,v\cdot\nabla\varsigma_R\dd x\dd s .
\]
Since $\mathbf 1_{\{|x|\ge2R\}}\le1-\varsigma_R\le\mathbf 1_{\{|x|\ge R\}}$ and
$|\nabla\varsigma_R|\le\|\nabla\varsigma\|_{L^\infty(\R^2)}/R$, this gives
Equation~\eqref{eq:tail}. Its right side does not depend on $t$ and tends to $0$
as $R\to\infty$. We split the $L^1(\R^2)$ norm into the parts on $B_{2R}$ and
outside it: for $s,t\in[0,T)$, by the Cauchy-Schwarz inequality on $B_{2R}$ and
Equation~\eqref{eq:tail} for $\omega(t)$ and $\omega(s)$,
\[
\|\omega(t)-\omega(s)\|_{L^1(\R^2)}
\le|B_{2R}|^{1/2}\|\omega(t)-\omega(s)\|_{L^2(\R^2)}
+2\sup_{\tau\in[0,T)}\int_{\{|x|\ge2R\}}|\omega(\tau)|\dd x .
\]
Given $\eta>0$, choose $R$ so that the last term is less than $\eta/2$; then the
$L^2(\R^2)$ continuity makes the first term less than $\eta/2$ for $t$ close to
$s$. Hence $\omega\in C([0,T);L^1(\R^2))$. Interpolation between $L^1(\R^2)$ and
$L^2(\R^2)$ gives $\omega\in C([0,T);L^q(\R^2))$ for $q\in(1,2)$.
\end{proof}

\subsection{Radial Data and the Criterion}

Recall from Section~\ref{sec:prelim} that, for $g\in L^1(\R^2)$,
$\Circ_g(r)=\int_{B_r}g$ is the integral of $g$ over the disk $B_r$; for smooth
$g$ it is the circulation of $\KBS\ast g$ around the circle $\{|x|=r\}$.

\begin{lemma}
\label{lem:radial-stationary}
Let $4/3<q<\infty$, $T>0$, and let $f\in (L^1\cap L^q)(\R^2)$ be radial. Then
$\KBS\ast f(x)=\Circ_f(|x|)\,x^\perp/(2\pi|x|^2)$ for almost every $x$, and
$\omega(t)=f$, $t\in[0,T]$, is a renormalized weak solution in $X_q(T)$ with
datum $f$.
\end{lemma}

\begin{proof}
We first prove the formula for radial $h=h_0(|\cdot|)\in L^1(\R^2)$. We
identify $\R^2$ with $\C$, so that $z^\perp=iz$ and $\KBS(z)=i/(2\pi\bar z)$.
By Lemma~\ref{lem:BS-basic}(e), for almost every $x\ne0$ the integral defining
$\KBS\ast h(x)$ converges absolutely, and in polar coordinates
$y=se^{i\vartheta}$,
\[
\KBS\ast h(x)=\frac{i}{2\pi}\int_0^\infty h_0(s)\,s
\int_0^{2\pi}\frac{\dd\vartheta}{\bar x-se^{-i\vartheta}}\dd s .
\]
The inner integral equals $2\pi/\bar x$ for $s<|x|$, by the uniformly
convergent expansion
$(\bar x-se^{-i\vartheta})^{-1}=\bar x^{-1}\sum_{l\ge0}(s/\bar x)^le^{-il\vartheta}$,
and $0$ for $s>|x|$, by the uniformly convergent expansion
$(\bar x-se^{-i\vartheta})^{-1}=-s^{-1}e^{i\vartheta}\sum_{l\ge0}(\bar x/s)^le^{il\vartheta}$,
integrated term by term: $\int_0^{2\pi}e^{il\vartheta}\dd\vartheta=0$ for every
integer $l\ne0$, so only the term $l=0$ of the first expansion survives, and no
term of the second.
Since $\Circ_h(r)=2\pi\int_0^rh_0(s)s\dd s$, we get
\[
\KBS\ast h(x)=\frac{i}{\bar x}\cdot\frac{\Circ_h(|x|)}{2\pi}
=\Circ_h(|x|)\frac{x^\perp}{2\pi|x|^2} .
\]

Now let $f=f_0(|\cdot|)$, $v_f=\KBS\ast f$, and $\Ren\in C^1(\R)$ with
$|\Ren'|\le L$ and $\Ren(0)=0$. The constant function $\omega(t)=f$ lies in
$X_q(T)$ and is weakly continuous, $\Ren(f)=\Ren(f_0(|\cdot|))$ is radial,
and $|\Ren(f)v_f|\le L|f||v_f|$ is
integrable by Lemma~\ref{lem:velocity}(a). Let $\varphi\in\DT$. In polar
coordinates $x=r(\cos\theta,\sin\theta)$ we have
$x^\perp\cdot\nabla=\partial_\theta$, so the formula for $v_f$ gives
$v_f\cdot\nabla\varphi(t)=\Circ_f(r)(2\pi r^2)^{-1}\partial_\theta\varphi(t)$,
and Fubini's theorem gives, for every $t$,
\[
\int\Ren(f)\,v_f\cdot\nabla\varphi(t)\dd x
=\int_0^\infty\Ren(f_0(r))\frac{\Circ_f(r)}{2\pi r}
\Big(\int_0^{2\pi}\partial_\theta\varphi(t)\dd\theta\Big)\dd r=0 .
\]
Since $\varphi(t)=0$ for $t$ near $T$, we also have
$\int_0^T\!\int\Ren(f)\partial_t\varphi\dd x\dd t=-\int\Ren(f)\varphi(0)\dd x$. Hence
Equation~\eqref{eq:renorm} holds, and Equation~\eqref{eq:weak} is the case
$\Ren(s)=s$.
\end{proof}

\begin{definition}
\label{def:adm}
Let $f\in (L^1\cap L^2)(\R^2)$. A
\textit{$2$-admissible sequence for $f$ on $[0,\infty)$} is a sequence of pairs
$(f^{(n)},\omega^{(n)})_{n\ge1}$ such that $f^{(n)}\in (L^1\cap L^2)(\R^2)$,
$f^{(n)}\to f$ in $(L^1\cap L^2)(\R^2)$, and $\omega^{(n)}$ is a renormalized weak
solution on $[0,\infty)$ with datum $f^{(n)}$, in the sense of
Definition~\ref{def:weak} with $q=2$.
\end{definition}

\begin{theorem}
\label{thm:criterion}
Let $f\in (L^1\cap L^2)(\R^2)$ be radial, and let
$(f^{(n)},\omega^{(n)})$ be a $2$-admissible sequence for $f$ on $[0,\infty)$.
Assume that for some $T>0$
\begin{equation}
\label{eq:defect}
c:=\limsup_{n\to\infty}\sup_{0\le t\le T}\|\omega^{(n)}(t)-f\|_{L^1(\R^2)}>0 .
\end{equation}
Then there exists a subsequence $(n_k)$, together with a function
$\omega_\infty\colon[0,\infty)\to (L^1\cap L^2)(\R^2)$, with the following properties.
\begin{enumerate}[label=\textup{(\roman*)}]
\item $\omega_\infty$ is a renormalized weak solution on $[0,\infty)$ with datum $f$,
  in the sense of Definition~\ref{def:weak} with $q=2$.
\item For every $t\in[0,\infty)$, $\omega_\infty(t)$ is equimeasurable with $f$.
  Hence $\|\omega_\infty(t)\|_{L^q(\R^2)}=\|f\|_{L^q(\R^2)}$ for every $q\in[1,\infty)$,
  and, whenever $f\in L^q(\R^2)$ with $q\in(4/3,\infty)$, $\omega_\infty\in X_q(T')$ for
  every $T'>0$.
\item $\omega_\infty\in C([0,\infty);L^q(\R^2))$ whenever $f\in L^q(\R^2)$ with
  $q\in[1,\infty)$.
\item $\sup_{0\le t\le T}\|\omega^{(n_k)}(t)-\omega_\infty(t)\|_{L^1(\R^2)}\to0$.
\item $\sup_{0\le t\le T}\|\omega_\infty(t)-f\|_{L^1(\R^2)}=c$. The set
  $\mathcal O=\{t\in[0,\infty):\omega_\infty(t)\ne f\}$ is open, $0\notin\mathcal O$,
  and $\mathcal O\cap(0,T]\ne\emptyset$. With $t_0=\inf\mathcal O\in[0,T)$,
  $\omega_\infty$ and the
  stationary solution $f$ are distinct elements of $X_q(T')$ for every
  $T'>t_0$, whenever $f\in L^q(\R^2)$ with $q\in(4/3,\infty)$.
\item If $\omega^{(n)}(t)$ is $3$-fold rotationally symmetric for all $n$ and $t$, then so is
  $\omega_\infty(t)$ for every $t$.
\end{enumerate}
\end{theorem}

\begin{proof}
Write $\omega_n=\omega^{(n)}$ and $f_n=f^{(n)}$. We first construct the limit.
Pass to a subsequence along which $\sup_{t\le T}\|\omega_n(t)-f\|_{L^1(\R^2)}\to c$,
and put $T_j=T+j$ for $j\ge1$. For $t\ge0$ choose $j$ with $T_j>t$; by
Definition~\ref{def:adm}, $\omega_n$ is a renormalized weak solution in
$X_2(T_j)$, so Lemma~\ref{lem:renorm-all-t} gives
$\|\omega_n(t)\|_{L^r(\R^2)}=\|f_n\|_{L^r(\R^2)}$ for $r\in\{1,2\}$. Hence
$M:=\sup_n\sup_{t\ge0}\|\omega_n(t)\|_{(L^1\cap L^2)(\R^2)}=\sup_n\|f_n\|_{(L^1\cap L^2)(\R^2)}<\infty$,
since $f_n$ converges in $(L^1\cap L^2)(\R^2)$; in particular
$c\le M+\|f\|_{L^1(\R^2)}<\infty$.

Write $v_n=\KBS\ast\omega_n$, and let $\mathcal F$ be as in the proof of
Lemma~\ref{lem:weak-cont}. Since
$\|v_n(t)\|_{L^2(B_R)}\le C(2,R)M$ by Lemma~\ref{lem:velocity}(a),
Lemma~\ref{lem:weak-cont}(b) on $[0,T_j]$ shows that the functions
$t\mapsto\int\omega_n(t)\psi$, $\psi\in\mathcal F$, are bounded by
$M\|\psi\|_{L^2(\R^2)}$ and Lipschitz on every $[0,T_j]$ with a constant that depends
neither on $n$ nor on $j$. By the Arzel\`a-Ascoli theorem and a diagonal
argument over $\psi\in\mathcal F$ and $j$, a subsequence $(n_k)$ converges
uniformly on every $[0,T_j]$ for every $\psi\in\mathcal F$; below, $n$ runs
through $(n_k)$. Since
$|\int(\omega_n(t)-\omega_m(t))\psi|\le2M\|\psi-\psi'\|_{L^2(\R^2)}+|\int(\omega_n(t)-\omega_m(t))\psi'|$
for $\psi\in L^2(\R^2)$ and $\psi'\in\mathcal F$, and $\mathcal F$ is dense in $L^2(\R^2)$,
$\int\omega_n(t)\psi$ is uniformly Cauchy on every $[0,T_j]$ for every
$\psi\in L^2(\R^2)$. For each $t$ the limit is a linear functional of norm at most
$M$, hence $\psi\mapsto\int\omega_\infty(t)\psi$ for a unique
$\omega_\infty(t)\in L^2(\R^2)$. This gives the uniform weak convergence
\[
\sup_{t\le T_j}\Big|\int(\omega_n(t)-\omega_\infty(t))\psi\Big|\to0
\qquad\text{for every }\psi\in L^2(\R^2)\text{ and every }j,
\]
and $\omega_\infty$ is weakly continuous into $L^2(\R^2)$, as a locally uniform limit
of the continuous functions $t\mapsto\int\omega_n(t)\psi$.

Thus $\omega_n(t)\rightharpoonup\omega_\infty(t)$ weakly in $L^2(\R^2)$ for every $t$.
Fix $t\ge0$. Weak lower semicontinuity of the $L^2(\R^2)$ norm gives
$\|\omega_\infty(t)\|_{L^2(\R^2)}\le\liminf_n\|\omega_n(t)\|_{L^2(\R^2)}$. For the $L^1(\R^2)$
norm, put $\lambda=\liminf_n\|\omega_n(t)\|_{L^1(\R^2)}$ and let $\varepsilon>0$. There
is a further subsequence $(\omega_{n_i}(t))_{i\ge1}$ with
$\|\omega_{n_i}(t)\|_{L^1(\R^2)}\le\lambda+\varepsilon$ for every $i$; its terms are
real, and it still converges weakly in $L^2(\R^2)$ to $\omega_\infty(t)$. The first
conclusion of Lemma~\ref{lem:velocity}(b), with $p=2$, applied to this
subsequence with $\lambda+\varepsilon$ in the role of the bound $M$ there, gives
$\|\omega_\infty(t)\|_{L^1(\R^2)}\le\lambda+\varepsilon$. As $\varepsilon>0$ is
arbitrary, $\|\omega_\infty(t)\|_{L^1(\R^2)}\le\liminf_n\|\omega_n(t)\|_{L^1(\R^2)}$. Hence
$\|\omega_\infty(t)\|_{L^1(\R^2)}+\|\omega_\infty(t)\|_{L^2(\R^2)}\le\liminf_n\big(\|\omega_n(t)\|_{L^1(\R^2)}+\|\omega_n(t)\|_{L^2(\R^2)}\big)\le M$,
so $\sup_{t\ge0}\|\omega_\infty(t)\|_{(L^1\cap L^2)(\R^2)}\le M$.
For $\phi\in L^\infty(\R^2)$, $t\mapsto\int\omega_\infty(t)\phi$ is the pointwise
limit of the continuous functions $t\mapsto\int\omega_\infty(t)\phi\mathbf 1_{B_\ell}$
as $\ell\to\infty$. So $\omega_\infty$ is weakly measurable into $L^1(\R^2)$ and into
$L^2(\R^2)$, hence strongly measurable into both by Pettis's theorem
(Section~\ref{sec:prelim}), as $L^1(\R^2)$ and $L^2(\R^2)$ are separable; with the bound by
$M$ and the convention on intersections in Section~\ref{sec:prelim},
$\omega_\infty\in X_2(T_j)$ for every $j$.

Lemma~\ref{lem:velocity}(b) with $p=2$ and the uniform weak convergence give
$v_n(t)\to v(t):=\KBS\ast\omega_\infty(t)$ in $L^2(B_R)$ for every $t$, and
$\|v_n(t)-v(t)\|_{L^2(B_R)}\le2C(2,R)M$; so dominated convergence in $t$ gives
$\int_0^{T_j}\|v_n-v\|_{L^2(B_R)}\dd t\to0$ for every $j$ and $R>0$. Now let
$\varphi\in\mathcal D_{T_j}$ with $\supp\varphi\subset\R\times B_R$. By the
uniform weak convergence and dominated convergence in $t$, with the bound
$2M\|\partial_t\varphi(t)\|_{L^2(\R^2)}$,
$\int_0^{T_j}\!\int\omega_n\partial_t\varphi\to\int_0^{T_j}\!\int\omega_\infty\partial_t\varphi$;
and $\int f_n\varphi(0)\to\int f\varphi(0)$, because $f_n\to f$ in $L^2(\R^2)$. In
\[
\int_0^{T_j}\!\!\int(\omega_nv_n-\omega_\infty v)\cdot\nabla\varphi
=\int_0^{T_j}\!\!\int\omega_n(v_n-v)\cdot\nabla\varphi
+\int_0^{T_j}\!\!\int(\omega_n-\omega_\infty)\,v\cdot\nabla\varphi ,
\]
the first term on the right is at most
$M\|\nabla\varphi\|_{L^\infty(\R\times\R^2)}\int_0^{T_j}\|v_n-v\|_{L^2(B_R)}\dd t\to0$.
In the second term on the right, the inner integral tends to $0$ for each $t$ by the uniform weak
convergence, because $v(t)\cdot\nabla\varphi(t)\in L^2(\R^2)$, and is bounded by
$2C(2,R)M^2\|\nabla\varphi\|_{L^\infty(\R\times\R^2)}$; so dominated convergence applies.
Thus Equation~\eqref{eq:weak} for $\omega_n$ on $[0,T_j)$ passes to the limit,
and $\omega_\infty$ is a weak solution in $X_2(T_j)$ with datum $f$ for every
$j$; being weakly continuous, it is its own representative in the sense of
Lemma~\ref{lem:weak-cont}. For $T'>0$ choose $T_j>T'$; since
$\mathcal D_{T'}\subset\mathcal D_{T_j}$, the restriction of $\omega_\infty$ to
$[0,T']$ is a weak solution in $X_2(T')$ with datum $f$, and it is
renormalized: by the remark at the start of Section~\ref{sec:weak-rep},
Lemma~\ref{lem:transport}(i) with $t_0=0$, $u=\omega_\infty$ and
$w=\KBS\ast\omega_\infty$ gives Equation~\eqref{eq:te} on $[0,T')$ for
$\Ren(\omega_\infty)$ with datum $\Ren(f)$, for every $\Ren\in\LipO$, in
particular for every $\Ren\in C^1(\R)$ with $\Ren'$ bounded and $\Ren(0)=0$;
this is Equation~\eqref{eq:renorm}. This proves (i).

We next prove (ii) and (iii). By Lemma~\ref{lem:renorm-all-t}, applied to
$\omega_\infty$ on each $[0,T_j)$, $\omega_\infty(t)$ and $f$ are equimeasurable
for every $t\ge0$, and integration of the distribution function gives
$\|\omega_\infty(t)\|_{L^q(\R^2)}=\|f\|_{L^q(\R^2)}$ for $q\in[1,\infty)$. For (iii), the
case $q\in[1,2]$ is part of Lemma~\ref{lem:renorm-all-t}. Let $s_i\to t$ in
$[0,\infty)$, and put $g_i=\omega_\infty(s_i)$ and $g=\omega_\infty(t)$. Then
$g_i\rightharpoonup g$ weakly in $L^2(\R^2)$ by weak continuity of $\omega_\infty$, and $\|g_i\|_{L^2(\R^2)}=\|g\|_{L^2(\R^2)}$, so
$\|g_i-g\|_{L^2(\R^2)}^2=2\|g\|_{L^2(\R^2)}^2-2\int g_ig\to0$. Let $q\in(2,\infty)$ with
$f\in L^q(\R^2)$, and let $K>0$. Where $|g_i|\le K$ and $|g|\le K$ we have
$|g_i-g|^q\le(2K)^{q-2}|g_i-g|^2$. On the remaining set, where $|g_i|>K$ or
$|g|>K$, we have
$|g_i-g|^q\le2^q\max\{|g_i|,|g|\}^q\le2^q\big(|g_i|^q\mathbf 1_{\{|g_i|>K\}}+|g|^q\mathbf 1_{\{|g|>K\}}\big)$.
Integrating these two bounds gives
\[
\|g_i-g\|_{L^q(\R^2)}^q\le(2K)^{q-2}\|g_i-g\|_{L^2(\R^2)}^2
+2^q\Big(\int_{\{|g_i|>K\}}|g_i|^q\dd x+\int_{\{|g|>K\}}|g|^q\dd x\Big).
\]
The last two integrals depend only on distribution functions. For $h\in L^q(\R^2)$,
write $|h|^q=K^q+\int_K^{|h|}qs^{q-1}\dd s$ on $\{|h|>K\}$ and integrate in $x$;
Tonelli's theorem gives
\[
\int_{\{|h|>K\}}|h|^q\dd x=K^q\,|\{|h|>K\}|+\int_K^\infty qs^{q-1}\,|\{|h|>s\}|\dd s .
\]
Since $g_i$ and $g$ are equimeasurable with $f$ by (ii), the measure
$|\{|h|>s\}|=|\{h>s\}|+|\{h<-s\}|$ is the same for $h=g_i$, $g$ and $f$, for
every $s>0$. Hence both integrals equal $\int_{\{|f|>K\}}|f|^q\dd x$, and
\[
\|g_i-g\|_{L^q(\R^2)}^q\le(2K)^{q-2}\|g_i-g\|_{L^2(\R^2)}^2+2^{q+1}\int_{\{|f|>K\}}|f|^q\dd x ,
\]
and letting $i\to\infty$ and then $K\to\infty$ gives $g_i\to g$ in $L^q(\R^2)$.
This proves (iii). Now let
$q\in(4/3,\infty)$ with $f\in L^q(\R^2)$ and $T'>0$. By (iii) and the norm
identities, $\omega_\infty$ is continuous from $[0,T']$ into $(L^1\cap L^q)(\R^2)$, so
$\omega_\infty\in X_q(T')$. This proves (ii).

We now prove (iv). Fix $j$ with $T<T_j$. We claim: if $(n_i)$ is
a subsequence of $(n_k)$ and $t_i\in[0,T]$ with $t_i\to t$, then
$\omega_{n_i}(t_i)\to\omega_\infty(t)$ in $L^1(\R^2)$. Granting the
claim, if (iv) failed, there would be $\epsilon>0$, a
subsequence $(n_i)$ and $t_i\in[0,T]$ with
$\|\omega_{n_i}(t_i)-\omega_\infty(t_i)\|_{L^1(\R^2)}\ge\epsilon$; along a further
subsequence $t_i\to t\in[0,T]$, and the claim and (iii) give
$\|\omega_{n_i}(t_i)-\omega_\infty(t_i)\|_{L^1(\R^2)}\to0$, a contradiction.

To prove the claim, write $n$ for $n_i$. For $\psi\in L^2(\R^2)$,
\[
\Big|\int(\omega_n(t_i)-\omega_\infty(t))\psi\Big|
\le\sup_{s\le T_j}\Big|\int(\omega_n(s)-\omega_\infty(s))\psi\Big|
+\Big|\int(\omega_\infty(t_i)-\omega_\infty(t))\psi\Big|\to0 ,
\]
by the uniform weak convergence on $[0,T_j]$ and weak continuity of
$\omega_\infty$. By the norm identities for $\omega_n$ above and part (ii),
$\|\omega_n(t_i)\|_{L^2(\R^2)}=\|f_n\|_{L^2(\R^2)}\to\|f\|_{L^2(\R^2)}=\|\omega_\infty(t)\|_{L^2(\R^2)}$;
expanding $\|\omega_n(t_i)-\omega_\infty(t)\|_{L^2(\R^2)}^2$ shows that
$\omega_n(t_i)\to\omega_\infty(t)$ in $L^2(\R^2)$, hence in $L^1(B_{2R})$ for every
$R>0$. Equation~\eqref{eq:tail} on $[0,T_j)$,
applied to $\omega_n$ and to $\omega_\infty$, gives
\[
\int_{\{|x|\ge2R\}}\big(|\omega_n(t_i)|+|\omega_\infty(t)|\big)\dd x
\le\|f_n-f\|_{L^1(\R^2)}+2\int_{\{|x|\ge R\}}|f|\dd x
+\frac{2\|\nabla\varsigma\|_{L^\infty(\R^2)}}{R}C_{\mathrm v}T_jM^2 .
\]
Letting $n\to\infty$ and then $R\to\infty$ proves the claim.

We now prove (v). By (iv),
\[
\Big|\sup_{t\le T}\|\omega_n(t)-f\|_{L^1(\R^2)}-\sup_{t\le T}\|\omega_\infty(t)-f\|_{L^1(\R^2)}\Big|
\le\sup_{t\le T}\|\omega_n(t)-\omega_\infty(t)\|_{L^1(\R^2)}\to0 ,
\]
so the second supremum equals $c>0$. By (iii) with $q=1$,
$t\mapsto\|\omega_\infty(t)-f\|_{L^1(\R^2)}$ is continuous on $[0,\infty)$, so
$\mathcal O$ is relatively open in $[0,\infty)$, and it meets $[0,T]$. Since
$\omega_\infty(0)=f$ by Lemma~\ref{lem:weak-cont}(a), $0\notin\mathcal O$; so
$\mathcal O$ is open in $\R$, it contains some $t_1\in(0,T]$, and
$0\le t_0=\inf\mathcal O<t_1\le T$, because $\mathcal O$ contains a
neighborhood of $t_1$. For $T'>t_0$ the set $\mathcal O\cap(0,T')$ is
nonempty and open, hence of positive measure, and $\omega_\infty(t)\ne f$ on it.
Thus, for every $q\in(4/3,\infty)$ with $f\in L^q(\R^2)$, $\omega_\infty$, which lies
in $X_q(T')$ by (ii), and the constant function $f$, which lies in $X_q(T')$
and is a solution by Lemma~\ref{lem:radial-stationary}, are distinct
elements of $X_q(T')$.

Finally we prove (vi). Let $t\ge0$ and $\psi\in C_c^\infty(\R^2)$. Since
$|\det\Orot|=1$ and $\omega_n(t)\circ\Orot=\omega_n(t)$ almost everywhere, by
the $3$-fold rotational symmetry of $\omega_n(t)$, the uniform weak convergence gives
\begin{align*}
\int\omega_\infty(t,\Orot x)\psi(x)\dd x&=\int\omega_\infty(t,y)\psi(\Orot^{-1}y)\dd y
=\lim_n\int\omega_n(t,y)\psi(\Orot^{-1}y)\dd y\\
&=\lim_n\int\omega_n(t,\Orot x)\psi(x)\dd x=\int\omega_\infty(t,x)\psi(x)\dd x ,
\end{align*}
so $\omega_\infty(t)\circ\Orot=\omega_\infty(t)$ almost everywhere.
\end{proof}

\begin{corollary}
\label{cor:defect}
Let $f\in (L^1\cap L^q)(\R^2)$ for every $q\in[1,\infty)$ be radial. Let $m\ge1$
be an integer, $\chi\in C_c^\infty((0,\infty))$ real, and
$\Psi(x)=\chi(|x|)\,(\bar x)^m$ (with $x$ written as a complex number). Let
$(f^{(n)},\omega^{(n)})$ be a $2$-admissible sequence for $f$ on $[0,\infty)$
such that $\omega^{(n)}(t)$ is $3$-fold rotationally symmetric for all $n$ and $t$. Assume
that there are $c_0>0$, $\bar T\in(0,\infty)$ and times
$T^{(n)}\in[0,\bar T]$ with
\[
\Big|\int\omega^{(n)}(T^{(n)})\,\Psi\Big|\ge c_0\qquad\text{for all large }n.
\]
Then there is a renormalized weak solution $\omega_\infty$
on $[0,\infty)$ with datum $f$, in the sense of Definition~\ref{def:weak} with
$q=2$, such that $\omega_\infty(t)$ is equimeasurable
with $f$ and $3$-fold rotationally symmetric for every $t\ge0$,
$\omega_\infty\in C([0,\infty);L^q(\R^2))$ for every $q\in[1,\infty)$, and, for
some $t_0\in[0,\bar T)$, $\omega_\infty$ and $f$ are distinct elements of
$X_q(T')$ for every $T'>t_0$ and every $q\in(4/3,\infty)$.
\end{corollary}

\begin{proof}
Write $f=f_0(|\cdot|)$. Since $\chi$ has compact support in $(0,\infty)$, the
function $\Psi$ lies in $C_c^\infty(\R^2;\C)$, and in polar coordinates
$\int f\Psi=\int_0^\infty f_0(r)\chi(r)r^{m+1}\big(\int_0^{2\pi}e^{-im\theta}\dd\theta\big)\dd r=0$,
because $m\ne0$. Hence, for all large $n$,
\[
c_0\le\Big|\int(\omega^{(n)}(T^{(n)})-f)\Psi\Big|
\le\|\Psi\|_{L^\infty(\R^2)}\sup_{t\le\bar T}\|\omega^{(n)}(t)-f\|_{L^1(\R^2)} .
\]
Thus $\|\Psi\|_{L^\infty(\R^2)}>0$, and Equation~\eqref{eq:defect} holds with
$T=\bar T$, with $c\ge c_0/\|\Psi\|_{L^\infty(\R^2)}$.
Theorem~\ref{thm:criterion} gives a subsequence and $\omega_\infty$. By its part
(i), $\omega_\infty$ is a renormalized weak solution on $[0,\infty)$ with datum
$f$, with $q=2$. Part (ii) gives equimeasurability, part (iii) gives
$\omega_\infty\in C([0,\infty);L^q(\R^2))$ for every $q\in[1,\infty)$, and part (v)
gives, with the set $\mathcal O$ defined there, $t_0=\inf\mathcal O\in[0,\bar T)$
and the distinctness in $X_q(T')$ for every $T'>t_0$ and $q\in(4/3,\infty)$.
Part (vi) gives the $3$-fold rotational symmetry of $\omega_\infty(t)$.
\end{proof}

\begin{remark}
\label{rem:conditional}
Corollary~\ref{cor:defect} is a conditional non-uniqueness theorem. Let $f$ be
radial with $f\in (L^1\cap L^q)(\R^2)$ for every $q\in[1,\infty)$. Suppose that one
can produce data $f^{(n)}\to f$ in $(L^1\cap L^2)(\R^2)$ and $3$-fold rotationally
symmetric renormalized weak solutions $\omega^{(n)}$ on $[0,\infty)$ with these
data and vorticity in $L^\infty_t(L^1\cap L^2)_x$, such that one fixed moment
$\int\omega^{(n)}(T^{(n)})\Psi$ of the form above has modulus at least $c_0>0$ at
times $T^{(n)}\le\bar T$, with $c_0$ and $\bar T$ independent of $n$. Then the
Euler equations with datum $f$ have two distinct renormalized weak solutions: the
stationary solution $f$ and $\omega_\infty$, which differ in $X_q(T')$ for every
$T'>t_0$ and every $q\in(4/3,\infty)$. Since $\int f\Psi=0$ for radial $f$, the
hypothesis says that the approximate solutions leave $f$ by a fixed amount, in
this moment, within a time bounded independently of $n$. Part~II
\cite{PartII} produces such a sequence for the datum of Theorem~\ref{thm:main}.
\end{remark}

\section{The Ring and Its Linearization}
\label{sec:ring}

We fix a certain exact radial steady state supported in an annulus, which we
call the \textbf{Ring}, and study its linearization, that of the Euler equations
at the Ring plus a point vortex of circulation $2\pi G$ at the origin.

\subsection{The Ring}
\label{sec:ring-def}

\begin{definition}
\label{def:ring}
Put
\[
\begin{gathered}
a_0=-7304210039150669\cdot2^{-54},\qquad d=5764607523034235\cdot2^{-58},\\
\xiL=-4528212286209965\cdot2^{-52},\qquad \xia=-4077852323472915\cdot2^{-52},\\
\xib=\tfrac12,\qquad \xiR=5404319552844595\cdot2^{-53},
\end{gathered}
\]
and $H(x)=(1+e^{-2x})^{-1}$ for $x\in\R$. Fix, once and for all, a function
$\chi\in C_c^\infty(\R)$ with values in $[0,1]$ that is nondecreasing on
$(-\infty,\xia]$, equal to $1$ on $[\xia,\xib]$, nonincreasing on
$[\xib,\infty)$, and supported in $[\xiL,\xiR]$ (such functions exist by
Lemma~\ref{lem:ring-facts}). Put
\[
\begin{gathered}
\prof(\xi)=\chi(\xi)\,H\Big(\frac{\xi-a_0}{d}\Big)H\Big(-\frac{\xi}{d}\Big)\quad(\xi\in\R),\\
\omc(x)=\prof(\log|x|)\quad(x\ne0),\qquad \omc(0)=0,
\end{gathered}
\]
\[
\Rc=e^{\xiR},\qquad \Phi(y)=\omc(\Rc y),\qquad a=e^{\xiL-\xiR},\qquad
\Gone=\tfrac1{50}\Rc^{-2}.
\]
To four significant figures, $a_0\approx-0.4055$, $d\approx0.02000$,
$\xiL\approx-1.005$, $\xia\approx-0.9055$, $\xiR\approx0.6000$,
$\Rc\approx1.822$, $a\approx0.2008$ and $\Gone\approx0.006024$.
We call $\Phi$ \textbf{the ring}. The functions $\omc$ and $\Phi$ describe the
same ring in the variables $x$ and $y$, related by $x=\Rc y$; we say that
$\omc$ is the ring in \textbf{reference units} and $\Phi$ the ring in
\textbf{ring units}. From now on $m=3$, except where an integer $m\ge1$ is
quantified explicitly.
\end{definition}

Figure~\ref{fig:ring} shows the ring, for the function $\chi$ constructed in the
proof of Lemma~\ref{lem:ring-facts}.

\begin{figure}[ht]
\centering
\includegraphics[width=0.85\textwidth]{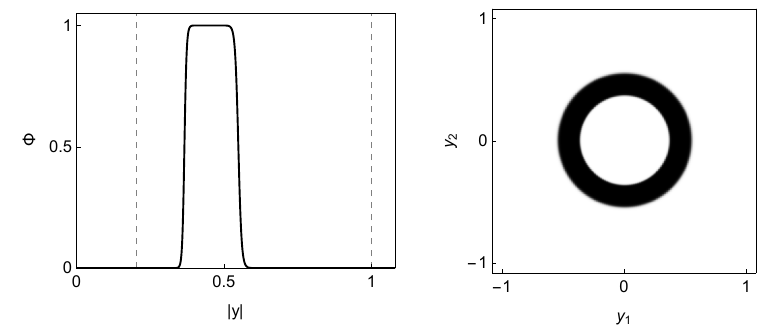}
\caption{The ring $\Phi$ of Definition~\ref{def:ring} in ring units, for the
function $\chi$ constructed in the proof of Lemma~\ref{lem:ring-facts}. Left:
the radial profile $\Phi(|y|)$; the dashed lines mark $|y|=a$ and $|y|=1$, between
which $\Phi$ is supported. Right: $\Phi$ on the plane, darker where $\Phi$ is larger
(black on the plateau, where $\Phi$ is close to $1$).
Drawn with Mathematica from the exact formulas; the figure is an illustration and
is not used in any proof.}
\label{fig:ring}
\end{figure}

\begin{lemma}
\label{lem:ring-facts}
A function $\chi$ as in Definition~\ref{def:ring} exists. Moreover
$a_0<0<d$, $\xiL<\xia\le a_0/2\le\xib<\xiR$, $\ell d<1$ for every integer
$1\le\ell\le9$, and $0<a<1$. The function $\omc$ lies in $C_c^\infty(\R^2)$, is
radial, satisfies $0\le\omc<1$, and $\supp\omc\subset\{e^{\xiL}\le|x|\le e^{\xiR}\}$.
Consequently $\Phi\in C_c^\infty(\R^2)$ is radial, $0\le\Phi<1$, $\Phi\not\equiv0$,
and $\supp\Phi\subset\{a\le|y|\le1\}$.
\end{lemma}
\begin{proof}
Let $e(x)=e^{-1/x}$ for $x>0$ and $e(x)=0$ for $x\le0$, and
$S(x)=e(x)/(e(x)+e(1-x))$. Then $S\in C^\infty(\R)$ is nondecreasing, $S=0$ on
$(-\infty,0]$ and $S=1$ on $[1,\infty)$. The function
\[
\chi(\xi)=S\Big(\frac{\xi-\xiL}{\xia-\xiL}\Big)\,S\Big(\frac{\xiR-\xi}{\xiR-\xib}\Big)
\]
is smooth with values in $[0,1]$. Its second factor is $1$ for $\xi\le\xib$ and
its first factor is $1$ for $\xi\ge\xia$; since $\xia\le\xib$, $\chi$ is the
first factor on $(-\infty,\xia]$, $1$ on $[\xia,\xib]$, the second factor on
$[\xib,\infty)$, and $0$ outside $[\xiL,\xiR]$. It has the required properties
once the order of the data is checked.

We compare the data as fractions with a common power of two. The signs of $a_0$
and $d$ are those of their numerators; $4528212286209965>4077852323472915$ gives
$\xiL<\xia$; $a_0/2=-7304210039150669\cdot2^{-55}$ and
$\xia=-32622818587783320\cdot2^{-55}$ give $\xia<a_0/2<0<\xib$; and
$\xib=4503599627370496\cdot2^{-53}<\xiR$. For $\ell d<1$ it suffices to take
$\ell=9$: $9\cdot5764607523034235=\nobreak51881467707308115<2^{58}=\nobreak288230376151711744$.
Since $\xiL<\xiR$, $0<a<1$.

Since $\prof$ is smooth and supported in $[\xiL,\xiR]$, $\omc$ is smooth on
$\R^2\setminus\{0\}$ and vanishes on $B_{e^{\xiL}}$, so $\omc\in C_c^\infty(\R^2)$
with the stated support. Since $0<H<1$ and $0\le\chi\le1$, $0\le\omc<1$. The
statements on $\Phi$ follow from $e^{\xiL}/\Rc=a$, $e^{\xiR}/\Rc=1$, and
$\Phi\not\equiv0$ because ${\prof(\xia)=H((\xia-a_0)/d)H(-\xia/d)>0}$.
\end{proof}

In ring units we use the closed annuli
\[
N=\{\tfrac a2\le|y|\le2\},\qquad N'=\{\tfrac a4\le|y|\le3\};
\]
so $\{a\le|y|\le1\}\subset N\subset N'$, and the ring $\Phi$ is supported in $N$.

\subsection{The Linearization and the Evans Function}

\begin{definition}
\label{def:sector-setting}
A \textbf{mode setting} consists of: a real radial function
$\bomega\in C_c^\infty(\R^2)$, written also as a function of $r=|x|$; numbers
$0<r_1\le r_2$ such that the radial derivative $\partial_r\bomega$ is
supported in $[r_1,r_2]$, and $\xi_1=\log r_1$,
$\xi_2=\log r_2$; a number $G\in\R$; and an annulus
$N_*=\{\rho_-\le|x|\le\rho_+\}$ with $0<\rho_-<r_1$ and $r_2<\rho_+$. In a mode
setting we put
\[
\begin{gathered}
\Omega(r;G)=r^{-2}\Big(\frac{\Circ_{\bomega}(r)}{2\pi}+G\Big),\qquad
\kcoef_\ell(r)=\frac{i\ell\,\partial_r\bomega(r)}{r},\\
(\Green_\ell g)(r)=-\frac1{2|\ell|}\int_0^\infty\Big(\frac{\min(r,\rho)}{\max(r,\rho)}\Big)^{|\ell|}g(\rho)\,\rho\dd\rho
\end{gathered}
\]
for $\ell\in\Z\setminus\{0\}$. For an integer $s\ge0$, $Y^s$ is the space of
$w\in H^s(\R^2;\C)$ with $\supp w\subset N_*$, and $Y^s_\ell$ is the space of
radial profiles $g$ with $\supp g\subset[\rho_-,\rho_+]$ and
$g(r)e^{i\ell\theta}\in H^s(\R^2)$, normed by $\|g(r)e^{i\ell\theta}\|_{H^s}$.
As a set, $Y^0_\ell=L^2([\rho_-,\rho_+],r\dd r)$, and the two norms differ by
the factor $\sqrt{2\pi}$. The \textbf{linearization} at $(\bomega,G)$, that is,
the linearization of the Euler equations~\eqref{eq:euler} at the steady state
$\bomega$ plus a point vortex of circulation $2\pi G$ at the origin, is
\[
\Lin(G)w=-\Omega(\cdot;G)\,\partial_\theta w-(\KBS\ast w)\cdot\nabla\bomega .
\]
The \textbf{mode operators} are $\Lsec{\ell}(G)g=-i\ell\Omega g+\kcoef_\ell\,\Green_\ell g$
for $\ell\ne0$ and $\Lsec{0}=0$. For $\ell\ge1$, $\Re\lambda>0$ and $\xi\in\R$,
the \textbf{Rayleigh potential} is
\[
\Vric_\ell(\xi;\lambda,G)=\ell^2+\frac{i\ell\,\nu(\xi)}{\lambda+i\ell\,\Omega(e^\xi;G)},\qquad
\nu(\xi)=r\partial_r\bomega(r)\big|_{r=e^\xi},
\]
and $\Mres_\lambda=(\lambda+i\ell\Omega(\cdot;G))^{-1}$. This is the resolvent at
$\lambda$ of the transport part $-i\ell\Omega$ of $\Lsec{\ell}(G)$, well defined
because $\Omega$ is real; with it, the equation $(\lambda-\Lsec{\ell}(G))g=h$
reads $g=\Mres_\lambda(h+\kcoef_\ell\Green_\ell g)$. The equation
$\partial_\xi^2\psi=\Vric_\ell(\cdot;\lambda,G)\,\psi$ is the Rayleigh equation of the mode
$\ell$ in the variable $\xi=\log r$.
\end{definition}

Here $\Circ_{\bomega}(r)=\int_{B_r}\bomega=2\pi\int_0^r\bomega(\rho)\rho\dd\rho$,
so $\Omega(r;G)$, smooth on $(0,\infty)$, is the angular velocity of the flow of
$\bomega$ and a point vortex of circulation $2\pi G$ at the origin. The
\textbf{mode profiles} $g_\ell$ of $w\in Y^0$ are its angular modes
(Section~\ref{sec:prelim}), $w=\sum_\ell g_\ell(r)e^{i\ell\theta}$ in $L^2(\R^2)$, and
$P_\ell w=g_\ell(r)e^{i\ell\theta}=\frac1{2\pi}\int_0^{2\pi}e^{-i\ell\phi}\,w\circ\Orot_\phi\dd\phi$.
Rotations preserve $N_*$ and are bounded on $H^s(\R^2)$, uniformly in the angle:
by the chain rule, for $|\gamma|=j$ the derivative $\partial^\gamma(w\circ\Orot_\phi)$ is a
sum of at most $2^j$ terms $c\,(\partial^\beta w)\circ\Orot_\phi$ with $|\beta|=j$ and
$|c|\le1$, so $\|w\circ\Orot_\phi\|_{H^s}\le C_s\|w\|_{H^s}$ with $C_s$ depending only
on $s$. (For $s\ge2$ they are in general not isometries of the norm of
Section~\ref{sec:prelim}.) By Minkowski's integral inequality, $P_\ell$ is therefore
a bounded projection on every $Y^s$, with norm at most $C_s$, and
$g_\ell\in Y^s_\ell$ for $w\in Y^s$.

We work in the variable $\xi=\log r$, the usual logarithmic change of variables,
which maps the half-line $r>0$ onto $\R$. Since $\partial_\xi=r\partial_r$, the
radial part of the Laplacian on the mode $\ell$ satisfies
\[
\partial_r^2\psi+r^{-1}\partial_r\psi-\ell^2r^{-2}\psi=r^{-2}\big(\partial_\xi^2\psi-\ell^2\psi\big).
\]
So $r^2$ times this operator is $\partial_\xi^2-\ell^2$, which has constant
coefficients and no first-order term, and the singular point $r=0$ of the radial
equation moves to $\xi=-\infty$. Since $\nu$ vanishes outside
$[\xi_1,\xi_2]$, so does $\Vric_\ell-\ell^2$; there the solutions of the Rayleigh
equation are the combinations of $e^{\ell\xi}=r^\ell$ and $e^{-\ell\xi}=r^{-\ell}$.
A dilation $x\mapsto cx$ becomes the translation $\xi\mapsto\xi+\log c$; this is
how Equation~\eqref{eq:evans-units} below compares the ring in reference units
with the ring in ring units.

The Evans function of Definition~\ref{def:evans} detects the eigenvalues of
$\Lsec{\ell}(G)$ in $\{\Re\lambda>0\}$, for $\ell\ge1$. If
$\Lsec{\ell}(G)g=\lambda g$ with $g\ne0$, then $g=\Mres_\lambda\kcoef_\ell\psi$
for the stream function $\psi=\Green_\ell g$. By the proof of
Lemma~\ref{lem:resolvent} below, $\psi$, as a function of $\xi$, is then a nonzero
solution of the Rayleigh equation with $\psi=O(e^{\ell\xi})$ as $\xi\to-\infty$
and $\psi=O(e^{-\ell\xi})$ as $\xi\to+\infty$; these conditions say that
$\psi(r)e^{i\ell\theta}$ is regular at the origin and decays at infinity. On
$(-\infty,\xi_1]$ the first condition leaves only the multiples of $e^{\ell\xi}$,
whose Cauchy data at $\xi_1$ are multiples of $(1,\ell)$; this is the
normalization of $\psi_\ell$ in Definition~\ref{def:evans}. On $[\xi_2,\infty)$,
$\psi_\ell=Ae^{\ell(\xi-\xi_2)}+Be^{-\ell(\xi-\xi_2)}$ with
$2\ell A=\partial_\xi\psi_\ell(\xi_2)+\ell\psi_\ell(\xi_2)$, so $\psi_\ell$ satisfies the
second condition exactly when $A=0$. Thus
$\Evans_\ell(\lambda;G)=2\ell e^{-\ell(\xi_2-\xi_1)}A$ measures the growing part
$r^\ell$ of $\psi_\ell$ at infinity, and Lemma~\ref{lem:resolvent} shows that its
zeros in $\{\Re\lambda>0\}$ are exactly the eigenvalues there.

The factor $e^{-\ell(\xi_2-\xi_1)}$ makes $\Evans_\ell$ equal to minus the
Wronskian of the solutions that equal $e^{\ell\xi}$ near $-\infty$ and
$e^{-\ell\xi}$ near $+\infty$ (Lemma~\ref{lem:evans-basic}(2) below). So
$\Evans_\ell$ depends only on $\bomega$, $G$ and $\lambda$, not on the choice of
$r_1$ and $r_2$, and for $\bomega=0$ it equals $2\ell$. Since $\Evans_\ell$ is
holomorphic in $\lambda$ (Lemma~\ref{lem:evans-basic}(1)) and is computed from an
initial value problem on the compact interval $[\xi_1,\xi_2]$, its zeros can be
counted with the argument principle. Part~II verifies \textup{(N1)} and
\textup{(N2)} of Hypothesis~\ref{hyp:N} below in this way, by certified winding
numbers of $\Evans_3$.

\begin{definition}
\label{def:evans}
In a mode setting, for an integer $\ell\ge1$ and $\Re\lambda>0$, let
$\psi_\ell(\cdot;\lambda,G)$ be the solution on $\R$ of
$\partial_\xi^2\psi=\Vric_\ell(\cdot;\lambda,G)\psi$ with $\psi(\xi_1)=1$, $\partial_\xi\psi(\xi_1)=\ell$,
and put
\[
\Evans_\ell(\lambda;G)=e^{-\ell(\xi_2-\xi_1)}\big(\partial_\xi\psi_\ell(\xi_2)+\ell\,\psi_\ell(\xi_2)\big),
\]
where $\psi_\ell=\psi_\ell(\cdot;\lambda,G)$. We call $\Evans_\ell$ the
\textbf{Evans function}.
\end{definition}

\begin{lemma}
\label{lem:sector}
In a mode setting the following hold.
\begin{enumerate}[label=\textup{(\arabic*)}]
\item For $\ell\ne0$ and $g\in Y^0_\ell$, $\Green_\ell g$ is the unique
  $\psi\in C^1((0,\infty))$ with $\partial_r\psi$ locally absolutely continuous that solves
  $\partial_r^2\psi+r^{-1}\partial_r\psi-\ell^2r^{-2}\psi=g$ almost everywhere on $(0,\infty)$ and is
  $O(r^{|\ell|})$ as $r\to0$ and $O(r^{-|\ell|})$ as $r\to\infty$; equivalently,
  $(\Green_\ell g)(r)e^{i\ell\theta}=\Delta^{-1}(g(r)e^{i\ell\theta})$, where
  $\Delta^{-1}$ is the convolution with $\frac1{2\pi}\log|x|$. The operator
  $\Green_\ell$ is bounded from $Y^0_\ell$ to $C([\rho_-,\rho_+])$.
\item $\Lin(G)$ preserves angular modes: for an integer $s\ge0$ and
  $w=\sum_\ell g_\ell(r)e^{i\ell\theta}\in Y^{s+1}$,
  $P_\ell\Lin(G)w=(\Lsec{\ell}(G)g_\ell)(r)e^{i\ell\theta}$ for every $\ell$; for the
  coupling term this holds for every $w\in Y^0$:
  $P_\ell\big(-(\KBS\ast w)\cdot\nabla\bomega\big)=(\kcoef_\ell\Green_\ell g_\ell)(r)e^{i\ell\theta}$.
  Moreover $\Lsec{-\ell}(G)\bar g=\overline{\Lsec{\ell}(G)g}$, and the coupling term
  commutes with rotations.
\item For $\ell\ge1$ and $\Re\lambda>0$, $|\lambda+i\ell\Omega(r;G)|\ge\Re\lambda$ for
  every $r>0$; $\Vric_\ell$ is smooth in $\xi$, holomorphic in $\lambda$, and
  equal to $\ell^2$ outside $[\xi_1,\xi_2]$.
\item For $\ell\ne0$, $\Lsec{\ell}(G)$ is a bounded operator on $Y^0_\ell$.
  Consequently the series
  $e^{t\Lsec{\ell}(G)}=\sum_{n\ge0}t^n\Lsec{\ell}(G)^n/n!$ converges in operator
  norm for every $t\in\R$ and defines a group of bounded operators on $Y^0_\ell$,
  continuous in $t$ in operator norm, with
  $\frac{\dd}{\dd t}e^{t\Lsec{\ell}(G)}=\Lsec{\ell}(G)e^{t\Lsec{\ell}(G)}
  =e^{t\Lsec{\ell}(G)}\Lsec{\ell}(G)$; in particular $\Lsec{\ell}(G)$ generates the
  $C_0$-semigroup $(e^{t\Lsec{\ell}(G)})_{t\ge0}$.
\end{enumerate}
\end{lemma}
\begin{proof}
(1) Write $n=|\ell|$. For $g\in Y^0_\ell$ the integrals below converge absolutely
by the Cauchy-Schwarz inequality, and
\[
(\Green_\ell g)(r)=-\frac1{2n}\Big[r^{-n}\int_0^r\rho^{n+1}g(\rho)\dd\rho+r^{n}\int_r^\infty\rho^{1-n}g(\rho)\dd\rho\Big]
\]
is continuous in $r>0$. The kernel has modulus at most $\frac1{2n}$ and $g$
vanishes outside $[\rho_-,\rho_+]$, so by the Cauchy-Schwarz inequality in
$L^2([\rho_-,\rho_+],\rho\dd\rho)$, for every $r>0$,
\[
|(\Green_\ell g)(r)|\le\frac1{2n}\int_{\rho_-}^{\rho_+}|g|\rho\dd\rho
\le\frac1{2n}\Big(\int_{\rho_-}^{\rho_+}\rho\dd\rho\Big)^{1/2}\|g\|_{L^2(r\dd r)}
\le\frac{\rho_+}{2n}\|g\|_{L^2(r\dd r)},
\]
as $\int_{\rho_-}^{\rho_+}\rho\dd\rho=\frac12(\rho_+^2-\rho_-^2)\le\rho_+^2$;
this is the last assertion. Differentiating, the terms from the variable limits
cancel, so $\psi=\Green_\ell g\in C^1$ with
$\partial_r\psi=\frac12\big[r^{-n-1}\int_0^r\rho^{n+1}g-r^{n-1}\int_r^\infty\rho^{1-n}g\big]$
absolutely continuous, and $\partial_r^2\psi+r^{-1}\partial_r\psi-n^2r^{-2}\psi=g$ almost everywhere.
Since $g$ vanishes outside $[\rho_-,\rho_+]$, for $r\le\rho_-$ the first integral in
the formula for $\psi$ vanishes and the second is $\int_{\rho_-}^{\rho_+}\rho^{1-n}g\dd\rho$,
and for $r\ge\rho_+$ the second vanishes and the first is
$\int_{\rho_-}^{\rho_+}\rho^{n+1}g\dd\rho$. So
\[
\psi(r)=c_-r^n\ (r\le\rho_-),\qquad \psi(r)=c_+r^{-n}\ (r\ge\rho_+),\qquad
c_\mp=-\frac1{2n}\int_{\rho_-}^{\rho_+}\rho^{1\mp n}g(\rho)\dd\rho ,
\]
and $\psi$ satisfies the growth conditions. For uniqueness, let $\tilde\psi$ be a
second solution with these properties and $D=\psi-\tilde\psi$. Then $D\in C^1$,
$\partial_rD$ is locally absolutely continuous, and $\partial_r^2D=-r^{-1}\partial_rD+n^2r^{-2}D$ almost everywhere;
the right side is continuous, so $D\in C^2((0,\infty))$ solves the homogeneous
equation classically. In the variable $\xi=\log r$ this equation reads
$\partial_\xi^2D=n^2D$ (by the identity displayed before
Definition~\ref{def:evans}), whose solutions are $Ae^{n\xi}+Be^{-n\xi}$; hence
$D=Ar^n+Br^{-n}$ for constants $A$ and $B$. As $r\to0$, $D=O(r^n)$, so
$B=r^n(D-Ar^n)=O(r^{2n})$ and $B=0$; then, as $r\to\infty$, $Ar^n=D=O(r^{-n})$, so
$A=0$. Thus $D=0$. For the equivalence, let first $g\in C_c^\infty((\rho_-,\rho_+))$, $w=g(r)e^{i\ell\theta}$,
$\Psi=(\Green_\ell g)(r)e^{i\ell\theta}$ and $\Psi_0=\frac1{2\pi}\log|\cdot|\ast w$.
Near the origin $\Psi$ is a multiple of the harmonic polynomial $x^n$ or $\bar
x^n$ ($x$ read as a complex number), so $\Psi\in C^\infty(\R^2)$, and $\Delta\Psi=w$
because $\Delta(\psi(r)e^{i\ell\theta})=(\partial_r^2\psi+r^{-1}\partial_r\psi-\ell^2r^{-2}\psi)e^{i\ell\theta}$.
Also $\Delta\Psi_0=w$, and $\Psi_0(x)=O(|x|^{-1})$ as $|x|\to\infty$. Indeed,
since $\ell\ne0$,
$\int w=\int_0^\infty g(r)r\dd r\int_0^{2\pi}e^{i\ell\theta}\dd\theta=0$, so
\[
\Psi_0(x)=\frac1{2\pi}\int\big(\log|x-y|-\log|x|\big)\,w(y)\dd y .
\]
Read $x$ and $y$ as complex numbers and let $|x|\ge2\rho_+$. On $\supp w$,
$|y|\le\rho_+\le|x|/2$, so $s=|1-y/x|\ge\frac12$ and $|s-1|\le|y|/|x|$; since
$|\log s|\le2|s-1|$ for $s\ge\frac12$, we get
$|\log|x-y|-\log|x||=|\log s|\le2|y|/|x|$ and
$|\Psi_0(x)|\le\frac{\rho_+}{\pi|x|}\|w\|_{L^1(\R^2)}$. So $\Psi-\Psi_0$ is harmonic and tends to $0$ at infinity, and Liouville's theorem
gives $\Psi=\Psi_0$. General $g\in Y^0_\ell$ follow by density, both sides being
continuous from $L^2([\rho_-,\rho_+],r\dd r)$ into $L^2_{\rm loc}(\R^2)$.

(2) With the conventions of Section~\ref{sec:prelim},
$\KBS\ast w=\nabla^\perp\Delta^{-1}w$ and $(\KBS\ast w)\cdot\er=-r^{-1}\partial_\theta\Delta^{-1}w$.
For $w=g(r)e^{i\ell\theta}$, $g\in C_c^\infty((\rho_-,\rho_+))$, $\ell\ne0$, part
(1) gives
$-(\KBS\ast w)\cdot\nabla\bomega=-\partial_r\bomega(r)\,(\KBS\ast w)\cdot\er
=\frac{i\ell\,\partial_r\bomega}{r}(\Green_\ell g)e^{i\ell\theta}=(\kcoef_\ell\Green_\ell g)e^{i\ell\theta}$;
for $\ell=0$, $\Delta^{-1}w$ is radial, so $\KBS\ast w=\nabla^\perp\Delta^{-1}w$
is a multiple of $\etheta$, while $\nabla\bomega=\partial_r\bomega\,\er$ is a multiple of $\er$,
and the term vanishes. Both sides are
continuous from $Y^0_\ell$ into $L^2(\R^2)$. For the left side, $w$ vanishes outside
$B_{\rho_+}$ and $\nabla\bomega$ is bounded and supported in $B_{\rho_+}$, so
Lemma~\ref{lem:BS-basic}(c) with $\rho=R=\rho_+$ gives
$\|(\KBS\ast w)\cdot\nabla\bomega\|_{L^2(\R^2)}\le2\rho_+\|\nabla\bomega\|_{L^\infty(\R^2)}\|w\|_{L^2(\R^2)}$;
for the right side, use (1). So the formula holds on $Y^0_\ell$. The coupling term commutes with rotations, since
$\KBS(\Orot_\phi x)=\Orot_\phi\KBS(x)$ and $\nabla\bomega$ is radial times $\er$;
hence it commutes with $P_\ell$, which gives the formula on $Y^0$. For $w\in Y^1$,
$P_\ell(-\Omega\partial_\theta w)=-i\ell\Omega g_\ell e^{i\ell\theta}$ by
integration by parts in $\theta$. The conjugation relation holds
because $\Omega$, $\bomega$, $\KBS$ and the kernel of $\Green_\ell=\Green_{-\ell}$ are
real.

(3) $\Re(\lambda+i\ell\Omega)=\Re\lambda$ because $\Omega$ is real; $\nu$ is smooth
and vanishes outside $[\xi_1,\xi_2]$, and $\Omega(e^\xi;G)$ is smooth in $\xi$.

(4) $\Omega(\cdot;G)$ and $\kcoef_\ell$ are bounded on $[\rho_-,\rho_+]$ and
$\kcoef_\ell$ vanishes outside $[r_1,r_2]$; so multiplication by $\Omega$, and
$\kcoef_\ell\Green_\ell$ by the last assertion of (1), are bounded on $Y^0_\ell$.
Write $L=\Lsec{\ell}(G)$. The series converges absolutely in operator norm, since
$\|t^nL^n/n!\|\le(|t|\,\|L\|)^n/n!$, uniformly for $t$ in compact sets. The group law
$e^{(t+t')L}=e^{tL}e^{t'L}$ follows from the Cauchy product of these absolutely
convergent series, and termwise differentiation of the power series in $t$, which
converges in operator norm, gives the derivative; so $L$ is the generator.
\end{proof}

Recall from Definition~\ref{def:evans} that, for an integer $\ell\ge1$ and
$\Re\lambda>0$, $\psi_\ell(\cdot;\lambda,G)$ is the solution on $\R$ of the Rayleigh
equation $\partial_\xi^2\psi=\Vric_\ell(\cdot;\lambda,G)\,\psi$ with $\psi(\xi_1)=1$ and
$\partial_\xi\psi(\xi_1)=\ell$.

\begin{lemma}
\label{lem:evans-basic}
In a mode setting, let $\ell\ge1$.
\begin{enumerate}[label=\textup{(\arabic*)}]
\item $\psi_\ell(\xi;\lambda,G)$ and its $\xi$-derivative are holomorphic in
  $\lambda\in\{\Re\lambda>0\}$, locally uniformly in $\xi$, and so is
  $\Evans_\ell(\cdot;G)$.
\item Put $\psi_L=e^{\ell\xi_1}\psi_\ell(\cdot;\lambda,G)$, so that
  $\psi_L(\xi)=e^{\ell\xi}$ for $\xi\le\xi_1$, and let $\psi_R$ be the solution of
  $\partial_\xi^2\psi=\Vric_\ell\psi$ with $\psi_R(\xi)=e^{-\ell\xi}$ for $\xi\ge\xi_2$
  \textup{(}$L$ and $R$ for the left end $\xi\le\xi_1$ and the right end
  $\xi\ge\xi_2$, where these solutions are explicit and decay\textup{)}. The
  Wronskian $\Wr(\lambda)=\psi_L\,\partial_\xi\psi_R-\partial_\xi\psi_L\,\psi_R$ is independent of $\xi$,
  and $\Wr(\lambda)=-\Evans_\ell(\lambda;G)$.
\end{enumerate}
For the ring in reference units, $\bomega=\omc$, $(\xi_1,\xi_2)=(\xiL,\xiR)$,
the Evans function of Definition~\ref{def:evans} is the function $\Evans_\ell$ of
Hypothesis~\ref{hyp:N} below.
\end{lemma}
\begin{proof}
(1) On $(-\infty,\xi_1]$, $\Vric_\ell=\ell^2$, so $\psi_\ell=e^{\ell(\xi-\xi_1)}$
there, independently of $\lambda$. On $[\xi_1,\infty)$, $Z=(\psi,\partial_\xi\psi)$
solves $\partial_\xi Z=AZ$ with
$A=\begin{psmallmatrix}0&1\\ \Vric_\ell&0\end{psmallmatrix}$; let
$Z_0=(1,\ell)$ and $Z_{n+1}(\xi)=Z_0+\int_{\xi_1}^\xi AZ_n$. For compact
$K\subset\{\Re\lambda>0\}$ and $\bar\xi>\xi_1$, Lemma~\ref{lem:sector}(3) gives $|A|\le
C_A=1+\ell^2+\ell\|\nu\|_{L^\infty(\R)}/\min_K\Re\lambda$ on $[\xi_1,\bar\xi]\times K$, and by
induction
$|Z_{n+1}-Z_n|\le|Z_0|(C_A(\xi-\xi_1))^{n+1}/(n+1)!$. The
iterates are continuous in $(\xi,\lambda)$ and holomorphic in $\lambda$ and converge
uniformly on $[\xi_1,\bar\xi]\times K$, so their limit $(\psi_\ell,\partial_\xi\psi_\ell)$ is
holomorphic in $\lambda$, uniformly in $\xi\in[\xi_1,\bar\xi]$. The same holds for $\psi_R$ of (2), with
data at $\xi_2$.

(2) The equation has no first-order term, so $\partial_\xi\Wr=0$. At $\xi_2$,
$\partial_\xi\psi_R=-\ell\psi_R=-\ell e^{-\ell\xi_2}$, so
$\Wr=-e^{-\ell\xi_2}(\partial_\xi\psi_L+\ell\psi_L)(\xi_2)
=-e^{-\ell(\xi_2-\xi_1)}(\partial_\xi\psi_\ell+\ell\psi_\ell)(\xi_2)=-\Evans_\ell(\lambda;G)$.

For the last assertion, $\omc$, $r_1=e^{\xiL}$, $r_2=e^{\xiR}$ and
$N_*=\{\frac12e^{\xiL}\le|x|\le2e^{\xiR}\}$ form a mode setting by
Lemma~\ref{lem:ring-facts}; $\nu=\partial_\xi\prof$, and
$\Circ_{\omc}(e^\xi)/(2\pi)=\int_0^{e^\xi}\omc(\rho)\rho\dd\rho
=\int_{-\infty}^\xi\prof(\xi')e^{2\xi'}\dd\xi'$, so $\Omega(e^\xi;G)=\Omega_G(\xi)$, with
$\Omega_G$ as in Hypothesis~\ref{hyp:N}.
So on $[\xiL,\xiR]$ the equation, the data and the formula for $\Evans_\ell$ are
those of Hypothesis~\ref{hyp:N}.
\end{proof}

For a closed operator $A$ and an isolated point $\lambda_0$ of its spectrum, the
\textbf{Riesz projection} at $\lambda_0$ is
$\frac1{2\pi i}\oint_{|\lambda-\lambda_0|=\epsilon}(\lambda-A)^{-1}\dd\lambda$, for any
$\epsilon>0$ such that the closed disk of radius $\epsilon$ meets the spectrum only
at $\lambda_0$; its rank is the \textbf{algebraic multiplicity} of $\lambda_0$, and
$\lambda_0$ is \textit{algebraically simple} if this rank is $1$.

\begin{lemma}
\label{lem:resolvent}
In a mode setting, let $\ell\ge1$ and $\Re\lambda>0$.
\begin{enumerate}[label=\textup{(\arabic*)}]
\item If $\Evans_\ell(\lambda;G)\ne0$, then $\lambda-\Lsec{\ell}(G)$ is invertible
  on $Y^0_\ell$ and on every $Y^s_\ell$, and for $g\in Y^0_\ell$,
  \[
  \begin{gathered}
  \big((\lambda-\Lsec{\ell})^{-1}g\big)(r)=\Mres_\lambda(r)\big(g(r)+\kcoef_\ell(r)\,\psi_g(\log r)\big),\\
  \psi_g(\xi)=\frac{1}{\Wr(\lambda)}\Big[\psi_R(\xi)\int_{-\infty}^{\xi}\psi_LF\dd\xi'+\psi_L(\xi)\int_{\xi}^{\infty}\psi_RF\dd\xi'\Big],
  \end{gathered}
  \]
  where $\psi_L$, $\psi_R$ and $\Wr$ are as in Lemma~\ref{lem:evans-basic}(2),
  and $F(\xi)=e^{2\xi}(\Mres_\lambda g)(e^\xi)$.
\item If $\Evans_\ell(\lambda;G)=0$, then $\lambda$ is an eigenvalue of
  $\Lsec{\ell}(G)$ on $Y^0_\ell$ and on every $Y^s_\ell$, with eigenfunction
  $g_\lambda=\Mres_\lambda\kcoef_\ell\,\psi_\ell(\log r;\lambda,G)\ne0$, which is
  smooth and supported in $[r_1,r_2]$.
\end{enumerate}
In particular, for each $\ell\ge1$ the spectrum of $\Lsec{\ell}(G)$ in
$\{\Re\lambda>0\}$ is the zero set of $\Evans_\ell(\cdot;G)$ and does not depend
on $s$; for $\ell\le-1$ it is the complex conjugate of that of $\Lsec{-\ell}(G)$.
\end{lemma}
\begin{proof}
For $\eta\in Y^0_\ell$ and $\psi=\Green_\ell\eta$, the equation
$(\lambda-\Lsec{\ell})\eta=g$ reads $\eta=\Mres_\lambda(g+\kcoef_\ell\psi)$. Put
$\tilde\psi(\xi)=\psi(e^\xi)$. By the identity displayed before
Definition~\ref{def:evans}, the radial equation of Lemma~\ref{lem:sector}(1) reads
$\partial_\xi^2\tilde\psi-\ell^2\tilde\psi=e^{2\xi}\eta(e^\xi)$. At $r=e^\xi$,
$e^{2\xi}\kcoef_\ell=r^2\cdot i\ell\,\partial_r\bomega/r=i\ell\nu$, so
$e^{2\xi}\Mres_\lambda\kcoef_\ell=i\ell\nu\Mres_\lambda=\Vric_\ell-\ell^2$, all
read at $r=e^\xi$; hence $u=\tilde\psi$ solves
\begin{equation}
\label{eq:bvp}
\partial_\xi^2u-\Vric_\ell u=F,\qquad u=O(e^{\ell\xi})\ (\xi\to-\infty),\qquad
u=O(e^{-\ell\xi})\ (\xi\to+\infty).
\end{equation}
Conversely, if $\tilde\psi$ solves \eqref{eq:bvp}, put $\psi(r)=\tilde\psi(\log r)$
and $\eta=\Mres_\lambda(g+\kcoef_\ell\psi)$. Then $\eta\in Y^0_\ell$, and
$\partial_\xi^2\tilde\psi-\ell^2\tilde\psi=e^{2\xi}\eta(e^\xi)$, so $\psi$ solves the
radial equation of Lemma~\ref{lem:sector}(1) with the growth conditions there;
hence $\psi=\Green_\ell\eta$ and $(\lambda-\Lsec{\ell})\eta=g$.

(1) By Lemma~\ref{lem:evans-basic}(2), $\Wr\ne0$. Put
$A(\xi)=\int_{-\infty}^\xi\psi_LF\dd\xi'$ and $B(\xi)=\int_\xi^\infty\psi_RF\dd\xi'$,
so that $\Wr\psi_g=\psi_RA+\psi_LB$, $\partial_\xi A=\psi_LF$ and
$\partial_\xi B=-\psi_RF$. Differentiating, the terms from the variable limits
cancel: $\Wr\,\partial_\xi\psi_g=\partial_\xi\psi_R\,A+\partial_\xi\psi_L\,B+(\psi_R\psi_L-\psi_L\psi_R)F
=\partial_\xi\psi_R\,A+\partial_\xi\psi_L\,B$. Differentiating again and using
$\partial_\xi^2\psi_L=\Vric_\ell\psi_L$ and $\partial_\xi^2\psi_R=\Vric_\ell\psi_R$,
\[
\Wr\,\partial_\xi^2\psi_g=\Vric_\ell(\psi_RA+\psi_LB)+(\psi_L\,\partial_\xi\psi_R-\partial_\xi\psi_L\,\psi_R)F
=\Wr\,(\Vric_\ell\psi_g+F)
\]
almost everywhere, by the definition of $\Wr$; so $\partial_\xi^2\psi_g=\Vric_\ell\psi_g+F$.
Since $F$ vanishes outside $[\log\rho_-,\log\rho_+]$, $\psi_g$ is a multiple of
$\psi_L=e^{\ell\xi}$ below $\min(\xi_1,\log\rho_-)$ and of $\psi_R=e^{-\ell\xi}$ above
$\max(\xi_2,\log\rho_+)$; so $\psi_g$ solves \eqref{eq:bvp}. As shown at the start
of the proof, this means that the function
$r\mapsto\Mres_\lambda(r)\big(g(r)+\kcoef_\ell(r)\psi_g(\log r)\big)$ solves the
resolvent equation. If $(\lambda-\Lsec{\ell})\eta=0$, then
$\xi\mapsto(\Green_\ell\eta)(e^\xi)$ solves \eqref{eq:bvp} with $F=0$,
so it is $O(e^{\ell\xi})$ solving $\partial_\xi^2\psi=\ell^2\psi$ on $(-\infty,\xi_1]$, hence a
multiple of $\psi_L$, and likewise of $\psi_R$; as $\Wr\ne0$, it vanishes, and
$\eta=\Mres_\lambda\kcoef_\ell\Green_\ell\eta=0$. For the bounds: $\Mres_\lambda$ is
smooth on $(0,\infty)$ by Lemma~\ref{lem:sector}(3), so it is a bounded
multiplier on every $Y^s_\ell$. Since $r\dd r=e^{2\xi}\dd\xi$,
$|\Mres_\lambda|\le1/\Re\lambda$ (Lemma~\ref{lem:sector}(3)) and $g$ vanishes outside
$[\rho_-,\rho_+]$, the Cauchy-Schwarz inequality, as in the proof of
Lemma~\ref{lem:sector}(1), gives
\[
\|F\|_{L^1(\R)}=\int_0^\infty|\Mres_\lambda g|\,r\dd r
\le\frac1{\Re\lambda}\int_{\rho_-}^{\rho_+}|g|\,r\dd r
\le\frac{\rho_+}{\Re\lambda}\|g\|_{L^2(r\dd r)}
=\frac{\rho_+}{\sqrt{2\pi}\,\Re\lambda}\|g\|_{Y^0_\ell},
\]
Put $C_L=\sup_{\xi\le\xi_2}(|\psi_L|+|\partial_\xi\psi_L|)$ and
$C_R=\sup_{\xi\ge\xi_1}(|\psi_R|+|\partial_\xi\psi_R|)$; they are finite, as
$\psi_L=e^{\ell\xi}$ for $\xi\le\xi_1$, $\psi_R=e^{-\ell\xi}$ for $\xi\ge\xi_2$, and
both are $C^1$ on $[\xi_1,\xi_2]$. For $\xi\in[\xi_1,\xi_2]$,
$|A(\xi)|\le C_L\|F\|_{L^1(\R)}$ and $|B(\xi)|\le C_R\|F\|_{L^1(\R)}$, so the
formulas for $\Wr\psi_g$ and $\Wr\,\partial_\xi\psi_g$ above give
$|\psi_g|+|\partial_\xi\psi_g|\le2C_LC_R|\Wr|^{-1}\|F\|_{L^1(\R)}\le C\|g\|_{Y^0_\ell}$ on
$[\xi_1,\xi_2]$, with $C$ depending on $\lambda$. If $g\in Y^s_\ell$, then
$g\in H^s((0,\infty))$, polar coordinates being smooth near $N_*$, so
$\|F\|_{H^s([\xi_1,\xi_2])}\le C\|g\|_{Y^s_\ell}$, and from $\partial_\xi^2\psi_g=\Vric_\ell\psi_g+F$
by induction $\|\psi_g\|_{H^{s+2}([\xi_1,\xi_2])}\le C\|g\|_{Y^s_\ell}$. As
$\kcoef_\ell$ is smooth and supported in $[r_1,r_2]$,
$\|\kcoef_\ell(r)\psi_g(\log r)e^{i\ell\theta}\|_{H^{s+2}}\le C\|g\|_{Y^s_\ell}$. Hence
$(\lambda-\Lsec{\ell})^{-1}$ is bounded on $Y^0_\ell$ and on every $Y^s_\ell$.

(2) Now $\Wr=0$. The functions $\psi_L$ and $\psi_R$ solve the same linear
second-order equation, and their Wronskian vanishes, so they are linearly
dependent; as $\psi_R\ne0$, $\psi_L=c\,\psi_R$ for a constant $c$. Hence
$\psi_\ell=e^{-\ell\xi_1}\psi_L=c\,e^{-\ell\xi_1}e^{-\ell\xi}$ for $\xi\ge\xi_2$, so
$\psi_\ell=O(e^{-\ell\xi})$ as $\xi\to+\infty$: the vanishing of $\Wr$ is what gives
the decay at $+\infty$. As $\psi_\ell=e^{\ell(\xi-\xi_1)}$ for $\xi\le\xi_1$, the
condition at $-\infty$ holds as well, so $\psi_\ell$ solves \eqref{eq:bvp} with
$F=0$, that is, with $g=0$. As shown at the start of the proof, if $u$ solves \eqref{eq:bvp}, then
$\eta(r)=\Mres_\lambda(r)\big(g(r)+\kcoef_\ell(r)\,u(\log r)\big)$ satisfies
$(\lambda-\Lsec{\ell})\eta=g$ and $(\Green_\ell\eta)(r)=u(\log r)$; with $g=0$ and
$u=\psi_\ell$ this gives $\eta=g_\lambda$, so $\Lsec{\ell}g_\lambda=\lambda g_\lambda$ and
$(\Green_\ell g_\lambda)(r)=\psi_\ell(\log r;\lambda,G)$, and
$g_\lambda\ne0$ since $\psi_\ell(\xi_1)=1$. As
$g_\lambda$ is smooth and supported in $[r_1,r_2]$, it lies in every $Y^s_\ell$.

The spectral statement for $\ell\ge1$ follows from (1) and (2). For $\ell\le-1$,
complex conjugation is an isometry of $Y^s_{-\ell}$ onto
$Y^s_\ell$ that conjugates $\lambda-\Lsec{\ell}$ to $\bar\lambda-\Lsec{-\ell}$, by
Lemma~\ref{lem:sector}(2).
\end{proof}

\begin{theorem}
\label{thm:evans-simple}
In a mode setting, let $\ell\ge1$ and let $\lamst$ be a zero of
$\Evans_\ell(\cdot;G)$ with $\Re\lamst>0$. Put $\psi_*=\psi_\ell(\cdot;\lamst,G)$,
$g_*=\Mres_{\lamst}\kcoef_\ell\,\psi_*(\log r)$ and
\[
\Jsimp=\int_0^\infty g_*(r)\,\psi_*(\log r)\,\Mres_{\lamst}(r)\,r\dd r
=i\ell\int_{\R}\frac{\nu(\xi)\,\psi_*(\xi)^2}{(\lamst+i\ell\Omega(e^\xi;G))^2}\dd\xi .
\]
\begin{enumerate}[label=\textup{(\arabic*)}]
\item $\psi_*(\xi_2)\ne0$ and
  $e^{\ell(\xi_2-\xi_1)}\,\psi_*(\xi_2)\,\partial_\lambda\Evans_\ell(\lamst;G)=-\Jsimp$.
\item If $\partial_\lambda\Evans_\ell(\lamst;G)\ne0$, then $\Jsimp\ne0$ by \textup{(1)},
  and the Riesz projection of $\Lsec{\ell}(G)$ at $\lamst$, on $Y^0_\ell$ and on
  each $Y^s_\ell$, is
  \[
  \Riesz g=\ellst(g)\,g_*,\qquad
  \ellst(g)=\frac1{\Jsimp}\int_0^\infty g(r)\,\psi_*(\log r)\,\Mres_{\lamst}(r)\,r\dd r ;
  \]
  it has rank one, so $\lamst$ is an algebraically simple eigenvalue.
\end{enumerate}
\end{theorem}
\begin{proof}
The two expressions for $\Jsimp$ agree because $r\dd r=e^{2\xi}\dd\xi$ and
$e^{2\xi}g_*=i\ell \nu\Mres_{\lamst}\psi_*$. Write $\Mres_*=\Mres_{\lamst}$.

(1) Put $\dot\psi=\partial_\lambda\psi_\ell(\cdot;\lambda,G)$ at $\lambda=\lamst$. It
vanishes on $(-\infty,\xi_1]$, where $\psi_\ell(\xi;\lambda,G)=e^{\ell(\xi-\xi_1)}$ for
every $\lambda$: there $\Vric_\ell=\ell^2$, and the data $\psi(\xi_1)=1$,
$\partial_\xi\psi(\xi_1)=\ell$ do not depend on $\lambda$ (proof of
Lemma~\ref{lem:evans-basic}(1)). Moreover, since differentiation in $\lambda$ and in $\xi$
commute by Lemma~\ref{lem:evans-basic}(1) and the integral form of the equation, it
solves $\partial_\xi^2\dot\psi=\Vric_\ell\dot\psi+\partial_\lambda\Vric_\ell\,\psi_*$ with
$\partial_\lambda\Vric_\ell=-i\ell \nu\Mres_*^2$. Hence
$\partial_\xi(\partial_\xi\dot\psi\,\psi_*-\dot\psi\,\partial_\xi\psi_*)=\partial_\lambda\Vric_\ell\,\psi_*^2$. Since
$\Evans_\ell(\lamst;G)=0$, $\partial_\xi\psi_*(\xi_2)=-\ell\psi_*(\xi_2)$, so at $\xi_2$
\[
\partial_\xi\dot\psi\,\psi_*-\dot\psi\,\partial_\xi\psi_*=\psi_*(\xi_2)(\partial_\xi\dot\psi+\ell\dot\psi)(\xi_2)
=\psi_*(\xi_2)e^{\ell(\xi_2-\xi_1)}\partial_\lambda\Evans_\ell(\lamst;G),
\]
while at $-\infty$ it vanishes. Integrating the identity
$\partial_\xi(\partial_\xi\dot\psi\,\psi_*-\dot\psi\,\partial_\xi\psi_*)=\partial_\lambda\Vric_\ell\,\psi_*^2$
in $\xi$ from $-\infty$ to $\xi_2$, as $\supp \nu\subset[\xi_1,\xi_2]$, gives $\psi_*(\xi_2)e^{\ell(\xi_2-\xi_1)}\allowbreak\partial_\lambda\Evans_\ell(\lamst;G)
=-i\ell\int \nu\Mres_*^2\psi_*^2\dd\xi=-\Jsimp$.
If $\psi_*(\xi_2)=0$, then $\partial_\xi\psi_*(\xi_2)=0$ and $\psi_*\equiv0$ by uniqueness,
contradicting $\psi_*(\xi_1)=1$.

(2) Let $D$ be a closed disk about $\lamst$ in $\{\Re\lambda>0\}$ containing no
other zero of $\Evans_\ell(\cdot;G)$; it exists because
$\partial_\lambda\Evans_\ell(\lamst;G)\ne0$. By Lemma~\ref{lem:resolvent}, $\lamst$
is isolated in the spectrum on $Y^0_\ell$ and on each $Y^s_\ell$. For $g\in Y^0_\ell$
and $\Re\lambda>0$ let $\mathcal N(\lambda)g$ be the bracket in the formula for
$\psi_g$ in Lemma~\ref{lem:resolvent}(1), that is,
\[
(\mathcal N(\lambda)g)(\xi)=\psi_R(\xi)\int_{-\infty}^{\xi}\psi_LF\dd\xi'+\psi_L(\xi)\int_{\xi}^{\infty}\psi_RF\dd\xi',
\qquad F(\xi)=e^{2\xi}(\Mres_\lambda g)(e^\xi),
\]
with $\psi_L$, $\psi_R$ taken at $\lambda$. Then $\mathcal N(\lambda)g=\Wr(\lambda)\psi_g
=-\Evans_\ell(\lambda;G)\psi_g$ wherever $\Evans_\ell(\lambda;G)\ne0$, so for
$\lambda\in D\setminus\{\lamst\}$
\[
\big((\lambda-\Lsec{\ell})^{-1}g\big)(r)=\Mres_\lambda(r)g(r)
-\frac{\Mres_\lambda(r)\kcoef_\ell(r)\,(\mathcal N(\lambda)g)(\log r)}{\Evans_\ell(\lambda;G)} .
\]
Unlike $\psi_g$, $\mathcal N(\lambda)g$ involves no division by $\Evans_\ell$: for fixed
$r$, $\Mres_\lambda(r)g(r)$ and $(\mathcal N(\lambda)g)(\log r)$ are holomorphic in
$\lambda$ on all of $\{\Re\lambda>0\}$, across $\lamst$ too, because $\psi_L$,
$\psi_R$ and $F$ are, with integrable bounds. So the only singularity of the
resolvent at $\lamst$ is the factor $1/\Evans_\ell$.
Since $D$ meets the spectrum only at $\lamst$, the Riesz projection is
$\Riesz g=\frac1{2\pi i}\oint_{\partial D}(\lambda-\Lsec{\ell})^{-1}g\dd\lambda$, the
$Y^0_\ell$-valued integral of a continuous function of $\lambda\in\partial D$. For
each $\lambda$ we use the representative of $(\lambda-\Lsec{\ell})^{-1}g$ given by
the formula above, which is defined at every $r$ and, for fixed $r$, continuous in
$\lambda\in\partial D$. The Riemann sums of the $Y^0_\ell$-valued integral converge
to $\Riesz g$ in $Y^0_\ell$, hence almost everywhere along a subsequence; at each
$r$ they are Riemann sums of the scalar integral
$\frac1{2\pi i}\oint_{\partial D}((\lambda-\Lsec{\ell})^{-1}g)(r)\dd\lambda$ and
converge to it. So $(\Riesz g)(r)$ equals this scalar integral for almost every
$r$. In it, $\Mres_\lambda(r)g(r)$ is holomorphic on $D$ and contributes $0$ by
Cauchy's theorem. In the other term, the numerator
$h_r(\lambda)=\Mres_\lambda(r)\kcoef_\ell(r)(\mathcal N(\lambda)g)(\log r)$ is
holomorphic on $D$, and $\Evans_\ell(\cdot;G)$ is holomorphic on $D$ with the single
zero $\lamst$, which is simple; so the integrand has at most a simple pole at $\lamst$, with
residue $-h_r(\lamst)/\partial_\lambda\Evans_\ell(\lamst;G)$. By the residue
theorem, for almost every $r$,
\[
(\Riesz g)(r)=-\frac{\Mres_*(r)\kcoef_\ell(r)\,(\mathcal N(\lamst)g)(\log r)}{\partial_\lambda\Evans_\ell(\lamst;G)} .
\]
At $\lamst$, $\Wr=0$, so $\psi_R=\kappa\psi_L$ with
$\kappa=e^{-\ell\xi_2}/(e^{\ell\xi_1}\psi_*(\xi_2))$, and
\[
\mathcal N(\lamst)g=\kappa\psi_L\int_\R\psi_LF\dd\xi'
=\kappa e^{2\ell\xi_1}\psi_*\int_0^\infty g(\rho)\,\psi_*(\log\rho)\Mres_*(\rho)\,\rho\dd\rho .
\]
Therefore
\[
\Riesz g=-\frac{e^{-\ell(\xi_2-\xi_1)}}{\psi_*(\xi_2)\,\partial_\lambda\Evans_\ell(\lamst;G)}
\Big(\int_0^\infty g\,\psi_*\Mres_*\,\rho\dd\rho\Big)g_*=\ellst(g)\,g_*,
\]
since
$\Jsimp=-e^{\ell(\xi_2-\xi_1)}\psi_*(\xi_2)\,\partial_\lambda\Evans_\ell(\lamst;G)$ by
(1). Since $\ellst(g_*)=\Jsimp/\Jsimp=1$, $\Riesz$ has rank one. The resolvents on $Y^s_\ell$ are restrictions of those on $Y^0_\ell$, so
the Riesz projection on $Y^s_\ell$ is the restriction of $\Riesz$, which maps into
$\mathrm{span}\{g_*\}\subset Y^s_\ell$.
\end{proof}

\begin{lemma}
\label{lem:adjoint}
Under the assumptions of Theorem~\ref{thm:evans-simple}, assume that $\lamst$ is
a simple zero, and put
$\pout_*=\int_0^\infty\rho^{\ell+1}g_*(\rho)\dd\rho$ and
$\pin_*=\int_0^\infty\rho^{1-\ell}g_*(\rho)\dd\rho$, so that
$\pout_*=\pout_\ell[g_*e^{i\ell\theta}]$ and $\pin_*=\pin_\ell[g_*e^{i\ell\theta}]$.
\begin{enumerate}[label=\textup{(\arabic*)}]
\item $\ellst(g_*)=1$, $\ellst(\Lsec{\ell}g)=\lamst\ellst(g)$ for $g\in Y^0_\ell$,
  and $\ellst(e^{t\Lsec{\ell}}g)=e^{\lamst t}\ellst(g)$ for $t\ge0$, where
  $e^{t\Lsec{\ell}}$ is the semigroup generated by $\Lsec{\ell}(G)$ on $Y^0_\ell$
  \textup{(}Lemma~\ref{lem:sector}(4)\textup{)}.
\item $\pin_*=-2\ell e^{-\ell\xi_1}$ and $\pout_*=-2\ell e^{\ell\xi_2}\psi_*(\xi_2)$.
\item Put $b(r)=-\frac i2r^{-\ell-1}\partial_r\bomega(r)$ and $\beta=\ellst(b)$. Then
  $\beta=-\pin_*/(2\ell\Jsimp)$.
\item $\pout_*\ne0$, $\pin_*\ne0$ and $\beta\ne0$.
\end{enumerate}
\end{lemma}
\begin{proof}
(1) Write $\Mres_*=\Mres_{\lamst}$ and put $X(r)=\psi_*(\log r)\Mres_*(r)$, so that $\kcoef_\ell X=g_*$ and
$(\lamst+i\ell\Omega)X=\psi_*=\Green_\ell g_*$ (see the proof of
Lemma~\ref{lem:resolvent}(2)), and
$\langle f,g\rangle=\int f\bar g\,r\dd r$, which is $(2\pi)^{-1}$ times the inner
product of $Y^0_\ell$. By Definition~\ref{def:sector-setting},
$(\Green_\ell g)(r)=\int_0^\infty K_\ell(r,\rho)\,g(\rho)\,\rho\dd\rho$ with
$K_\ell(r,\rho)=-\frac1{2|\ell|}\big(\min(r,\rho)/\max(r,\rho)\big)^{|\ell|}$. This
integral kernel, taken with respect to the measure $\rho\dd\rho$, is real and
symmetric, $K_\ell(r,\rho)=K_\ell(\rho,r)$. As $K_\ell$ is bounded, Fubini's
theorem gives, for the operator $\Green_\ell$, $\langle\Green_\ell f,h\rangle=\langle f,\Green_\ell h\rangle$
for $f,h\in L^2([\rho_-,\rho_+],r\dd r)$, and $\overline{\Green_\ell h}=\Green_\ell\bar h$.
The operator $\Lsec{\ell}$ is bounded on $Y^0_\ell$ (Lemma~\ref{lem:sector}(4)); we
compute its adjoint $\Lsec{\ell}^*$ for $\langle\cdot,\cdot\rangle$. For
$f,h\in Y^0_\ell$, $\langle-i\ell\Omega f,h\rangle=\langle f,i\ell\Omega h\rangle$ as
$\Omega$ is real, and
$\langle\kcoef_\ell\Green_\ell f,h\rangle=\langle\Green_\ell f,\bar\kcoef_\ell h\rangle
=\langle f,\Green_\ell(\bar\kcoef_\ell h)\rangle$. The function
$\Green_\ell(\bar\kcoef_\ell h)$ is defined on all of $(0,\infty)$, but $f$ vanishes
outside $[\rho_-,\rho_+]$, so only its restriction to $[\rho_-,\rho_+]$ enters, and
that restriction lies in $Y^0_\ell$. Hence
$\Lsec{\ell}^*h=i\ell\Omega h+\mathbf 1_{[\rho_-,\rho_+]}\Green_\ell(\bar\kcoef_\ell h)$.
Conjugating $(\lamst+i\ell\Omega)X=\Green_\ell(\kcoef_\ell X)$ gives
$\Lsec{\ell}^*\bar X=\overline{\lamst}\bar X$ on $[\rho_-,\rho_+]$. Since
$\ellst(g)=\Jsimp^{-1}\langle g,\bar X\rangle$, we get $\ellst(g_*)=1$ and
$\ellst(\Lsec{\ell}g)=\Jsimp^{-1}\langle g,\overline{\lamst}\bar X\rangle=\lamst\ellst(g)$.
The functional $\ellst$ is bounded on $Y^0_\ell$, that is, a continuous linear map
$Y^0_\ell\to\C$: $X$ is continuous on $[\rho_-,\rho_+]$, as $\psi_*$ is continuous and
$|\Mres_*|\le1/\Re\lamst$ (Lemma~\ref{lem:sector}(3)), so the Cauchy-Schwarz
inequality gives
$|\ellst(g)|\le|\Jsimp|^{-1}\|X\|_{L^2([\rho_-,\rho_+],r\dd r)}\|g\|_{L^2(r\dd r)}$.
By Lemma~\ref{lem:sector}(4), $t\mapsto e^{t\Lsec{\ell}}g$ is differentiable in
$Y^0_\ell$ with derivative $\Lsec{\ell}e^{t\Lsec{\ell}}g$, so the function
$t\mapsto\ellst(e^{t\Lsec{\ell}}g)$ has derivative
$\ellst(\Lsec{\ell}e^{t\Lsec{\ell}}g)=\lamst\ellst(e^{t\Lsec{\ell}}g)$, which gives the
last identity.

(2) As in the proof of (1), $\Green_\ell g_*=\psi_*(\log r)$, and $g_*$ vanishes
outside $[r_1,r_2]$. For $r\ge r_2$ the definition of $\Green_\ell$ gives
$\Green_\ell g_*(r)=-\frac1{2\ell}r^{-\ell}\int\rho^{\ell+1}g_*\dd\rho$, and for $r\le r_1$
it gives $-\frac1{2\ell}r^\ell\int\rho^{1-\ell}g_*\dd\rho$. Compare with
$\psi_*(\xi)=e^{\ell(\xi-\xi_1)}$ for $\xi\le\xi_1$ and
$\psi_*(\xi)=\psi_*(\xi_2)e^{-\ell(\xi-\xi_2)}$ for $\xi\ge\xi_2$. The identities
$\pout_*=\pout_\ell[g_*e^{i\ell\theta}]$ and $\pin_*=\pin_\ell[g_*e^{i\ell\theta}]$
follow from Definition~\ref{def:multipole} by integrating in $\theta$.

(3) From $g_*=\Mres_*\kcoef_\ell\psi_*$, $(\partial_r\bomega)\psi_*\Mres_*=rg_*/(i\ell)$, so
\[
\beta=\frac{1}{\Jsimp}\int b\,\psi_*\Mres_*\,r\dd r=-\frac{i}{2\Jsimp}\int r^{-\ell}(\partial_r\bomega)\psi_*\Mres_*\dd r
=-\frac{1}{2\ell\Jsimp}\int r^{1-\ell}g_*\dd r=-\frac{\pin_*}{2\ell\Jsimp}.
\]

(4) $\pin_*\ne0$ by (2). If $\pout_*=0$, then $\psi_*(\xi_2)=0$ by (2),
contradicting Theorem~\ref{thm:evans-simple}(1). As $\lamst$ is a simple zero, $\Jsimp\ne0$ by
Theorem~\ref{thm:evans-simple}(1), and $\beta\ne0$ by (3).
\end{proof}

\subsection{An A Priori Bound for Unstable Eigenvalues}

\begin{lemma}
\label{lem:edge}
In a mode setting, let $\mathcal U_+=\{r:\partial_r\bomega(r)>0\}$ and
$\mathcal U_-=\{r:\partial_r\bomega(r)<0\}$, and assume that $r<\rho$ for all
$r\in\mathcal U_+$ and $\rho\in\mathcal U_-$. For $\ell\ge1$ put
\[
\Zedge_\ell=\int_{\mathcal U_+}\int_{\mathcal U_-}\Big(\frac r\rho\Big)^{2\ell}|\partial_r\bomega(r)|\,|\partial_r\bomega(\rho)|\dd\rho\dd r .
\]
If $g\in Y^0_{\pm\ell}$, $g\ne0$, $\Lsec{\pm\ell}(G)g=\lambda g$ and
$\Re\lambda>0$, then $\Re\lambda\le\frac12\Zedge_\ell^{1/2}$, and $\Zedge_\ell$ is
nonincreasing in $\ell$.
\end{lemma}
\begin{proof}
By Lemma~\ref{lem:resolvent}, the eigenvalues of $\Lsec{-\ell}$ in
$\{\Re\lambda>0\}$ are the conjugates of those of $\Lsec{\ell}$, so we may take the
sign $+$. Let $\psi=\Green_\ell g$, which is continuous, and
$\langle f,h\rangle=\int f\bar h\,r\dd r$. Since $\Re\lambda>0$, $\lambda+i\ell\Omega\ne0$
(Lemma~\ref{lem:sector}(3)), so $g=\Mres_\lambda\kcoef_\ell\psi$ vanishes outside
$\{\partial_r\bomega\ne0\}$; there put $W=r/|\partial_r\bomega|$ and
$\|g\|_W^2=\int|g|^2Wr\dd r=\int\ell^2|\partial_r\bomega|\,|\psi|^2|\lambda+i\ell\Omega|^{-2}\dd r$,
which is finite, and positive as $g\ne0$. Multiplying $\lambda g=-i\ell\Omega g+\kcoef_\ell\psi$
by $\bar gWr$, integrating, and using $\kcoef_\ell Wr=i\ell\,(\sgn\partial_r\bomega)\,r$, we get
$\lambda\|g\|_W^2=-i\ell\int\Omega|g|^2Wr\dd r+i\ell\int(\sgn\partial_r\bomega)\,\psi\bar g\,r\dd r$,
and the first term is imaginary. With $g_a=g\mathbf 1_{\mathcal U_+}$,
$g_b=g\mathbf 1_{\mathcal U_-}$,
\[
\int(\sgn\partial_r\bomega)\,\psi\bar g\,r\dd r=\langle\Green_\ell g_a,g_a\rangle-\langle\Green_\ell g_b,g_b\rangle
+\langle\Green_\ell g_b,g_a\rangle-\langle\Green_\ell g_a,g_b\rangle .
\]
The kernel $K_\ell$ of $\Green_\ell$ is real and symmetric (proof of
Lemma~\ref{lem:adjoint}(1)), so the first two terms are real
and the last is the conjugate of $X_{ab}=\langle\Green_\ell g_b,g_a\rangle$; hence
$\Re\lambda\,\|g\|_W^2=-2\ell\Im X_{ab}$. For $r\in\mathcal U_+$, $\rho\in\mathcal U_-$
we have $r<\rho$, so $|K_\ell(r,\rho)|=\frac1{2\ell}(r/\rho)^\ell$. Writing
$|g_a|r=(|g_a|^2Wr)^{1/2}|\partial_r\bomega|^{1/2}$, likewise for $g_b$, the Cauchy-Schwarz
inequality on $\mathcal U_+\times\mathcal U_-$ gives
$|X_{ab}|\le\frac1{2\ell}\Zedge_\ell^{1/2}\|g_a\|_W\|g_b\|_W\le\frac1{4\ell}\Zedge_\ell^{1/2}\|g\|_W^2$,
by the AM-GM inequality $\|g_a\|_W\|g_b\|_W\le\frac12(\|g_a\|_W^2+\|g_b\|_W^2)=\frac12\|g\|_W^2$
($g_a$ and $g_b$ have disjoint supports and $g=g_a+g_b$); so
$\Re\lambda\le\frac12\Zedge_\ell^{1/2}$. Since $r/\rho<1$ on $\mathcal U_+\times\mathcal U_-$, the
integrand of $\Zedge_\ell$ decreases in $\ell$.
\end{proof}

\begin{lemma}
\label{lem:closed-form}
Let $\chi$, $\prof$ and $\omc$ be as in Definition~\ref{def:ring}. Then:
\begin{enumerate}[label=\textup{(\arabic*)}]
\item $\partial_\xi\prof\ge0$ on $(-\infty,a_0/2)$ and $\partial_\xi\prof\le0$ on $(a_0/2,\infty)$; so
  $\omc$ satisfies the hypothesis of Lemma~\ref{lem:edge};
\item for every integer $\ell\ge1$ with $\ell d<1$, the quantity $\Zedge_\ell$ of
  Lemma~\ref{lem:edge} for $\bomega=\omc$ is $\Zedge_\ell=\Zedge^a_\ell\Zedge^b_\ell$ with
  \[
  \Zedge^a_\ell=\int_{\{\partial_\xi\prof>0\}}e^{2\ell(\xi-a_0/2)}\partial_\xi\prof(\xi)\dd\xi,\qquad
  \Zedge^b_\ell=\int_{\{\partial_\xi\prof<0\}}e^{-2\ell(\xi-a_0/2)}|\partial_\xi\prof(\xi)|\dd\xi,
  \]
  and $\Zedge^a_\ell,\Zedge^b_\ell\le e^{\ell a_0}\Thl{\ell}+\epsst$, where
  $\Thl{\ell}=\pi\ell d/\sin(\pi\ell d)$ and
  $\epsst=\max\{H((\xi_a-a_0)/d),\,H(-\xi_b/d)\}$.
\end{enumerate}
\end{lemma}
\begin{proof}
By Lemma~\ref{lem:ring-facts}, $a_0<0<d$ and $\xiL<\xia\le a_0/2\le\xib<\xiR$.
Put $\xi_c=a_0/2$ and $P(\xi)=H(\frac{\xi-a_0}d)H(-\frac\xi d)$, so $\prof=\chi P$.
Since $H(x)=\frac12(1+\tanh x)$, $\partial_xH(x)=2H(x)H(-x)=\frac12\operatorname{sech}^2x$ and
$\partial_\xi\log P(\xi)=\frac2d\big[H(-\frac{\xi-a_0}d)-H(\frac\xi d)\big]$, which is
positive if and only if $-(\xi-a_0)>\xi$, i.e.\ $\xi<\xi_c$, as $H$ is increasing.
For $\xi<\xi_c$, $\partial_\xi\chi\ge0$ (by monotonicity on $(-\infty,\xi_a]$, and $\partial_\xi\chi=0$
on $[\xi_a,\xi_c]$), $\chi\ge0$ and $\partial_\xi P>0$, so $\partial_\xi\prof=(\partial_\xi\chi)P+\chi\,\partial_\xi P\ge0$; for
$\xi>\xi_c$ likewise $\partial_\xi\prof\le0$. As $\partial_r\omc(r)=(\partial_\xi\prof)(\log r)/r$,
$\{\partial_r\omc>0\}\subset(0,e^{\xi_c}]$ and $\{\partial_r\omc<0\}\subset[e^{\xi_c},\infty)$ are
disjoint; this proves (1).

(2) In the double integral defining $\Zedge_\ell$ (Lemma~\ref{lem:edge}), with
$r\in\mathcal U_+$ and $\rho\in\mathcal U_-$, substitute $\xi=\log r$ and
$\xi'=\log\rho$. As $\partial_r\omc(r)=(\partial_\xi\prof)(\log r)/r$, the sets
$\mathcal U_\pm$ become $\{\pm\partial_\xi\prof>0\}$,
$|\partial_r\omc(r)|\dd r=|\partial_\xi\prof(\xi)|\dd\xi$,
$|\partial_r\omc(\rho)|\dd\rho=|\partial_\xi\prof(\xi')|\dd\xi'$ and
$r/\rho=e^{\xi-\xi'}$, so
\[
\Zedge_\ell=\int_{\{\partial_\xi\prof>0\}}\int_{\{\partial_\xi\prof<0\}}
e^{2\ell(\xi-\xi')}\,|\partial_\xi\prof(\xi)|\,|\partial_\xi\prof(\xi')|\dd\xi'\dd\xi ,
\]
and the factorization follows from
$e^{2\ell(\xi-\xi')}=e^{2\ell(\xi-\xi_c)}e^{-2\ell(\xi'-\xi_c)}$. For the bound on
$\Zedge^a_\ell$, write $\partial_\xi\prof=(\partial_\xi\chi)P+\chi\,\partial_\xi P$ and
bound the two terms on $(-\infty,\xi_c]$, which contains $\{\partial_\xi\prof>0\}$.
First,
$\partial_\xi P=\frac1d(\partial_xH)(\frac{\xi-a_0}d)H(-\frac\xi d)-\frac1dH(\frac{\xi-a_0}d)(\partial_xH)(-\frac\xi d)$.
The second term is $\le0$, as $H,\partial_xH\ge0$, and we drop it; in the first,
$H(-\frac\xi d)\le1$ and $\partial_xH=\frac12\operatorname{sech}^2$. So
$\partial_\xi P\le\frac1{2d}\operatorname{sech}^2\big(\frac{\xi-a_0}d\big)$ for every $\xi$,
and, as $0\le\chi\le1$, also
$\chi\,\partial_\xi P\le\frac1{2d}\operatorname{sech}^2\big(\frac{\xi-a_0}d\big)$.
Second, by Definition~\ref{def:ring}, $\chi$ is nondecreasing on $(-\infty,\xi_a]$ and
equal to $1$ on $[\xi_a,\xi_b]$, which contains $[\xi_a,\xi_c]$; so on
$(-\infty,\xi_c]$ we have $\partial_\xi\chi\ge0$, and $\partial_\xi\chi\ne0$ only for
$\xi<\xi_a$. For $\xi\le\xi_a$, $P(\xi)\le H(\frac{\xi-a_0}d)\le H(\frac{\xi_a-a_0}d)$, as
$H\le1$ and $H$ is increasing; so
$(\partial_\xi\chi)P\le H(\frac{\xi_a-a_0}d)\,\partial_\xi\chi$ on $(-\infty,\xi_c]$.
Both bounds are nonnegative. We integrate the first, times $e^{2\ell(\xi-\xi_c)}$,
over all of $\R$, and for the second we use $e^{2\ell(\xi-\xi_c)}\le1$ on
$(-\infty,\xi_c]$ and $\int_{-\infty}^{\xi_c}\partial_\xi\chi\dd\xi=\chi(\xi_c)=1$. The
substitution $\xi=a_0+dx$ gives
\[
\Zedge^a_\ell\le e^{2\ell(a_0-\xi_c)}\int_\R e^{2\ell dx}\,\tfrac12\operatorname{sech}^2x\dd x
+H\Big(\frac{\xi_a-a_0}d\Big)=e^{\ell a_0}\Thl{\ell}+H\Big(\frac{\xi_a-a_0}d\Big),
\]
by $\int_\R e^{\alpha x}\operatorname{sech}^2x\dd x=\pi\alpha/\sin(\pi\alpha/2)$ for
$0<|\alpha|<2$, with $\alpha=2\ell d$. As
$\operatorname{sech}^2$ is even,
$\int_\R e^{\alpha x}\operatorname{sech}^2x\dd x=2\int_0^\infty\cosh(\alpha x)(\cosh x)^{-2}\dd x$,
which is $2B(1+\frac\alpha2,1-\frac\alpha2)$ by
\cite[\href{https://dlmf.nist.gov/5.12.E7}{(5.12.7)}]{DLMF} with $a=1$ and
$b=\frac\alpha2$ (admissible as $|\alpha|<2$). By
\cite[\href{https://dlmf.nist.gov/5.12.E1}{(5.12.1)}]{DLMF},
$2B(1+\frac\alpha2,1-\frac\alpha2)=2\Gamma(1+\frac\alpha2)\Gamma(1-\frac\alpha2)
=\alpha\Gamma(\frac\alpha2)\Gamma(1-\frac\alpha2)$, and the reflection formula
$\Gamma(z)\Gamma(1-z)=\pi/\sin(\pi z)$
\cite[\href{https://dlmf.nist.gov/5.5.E3}{(5.5.3)}]{DLMF} gives $\pi\alpha/\sin(\pi\alpha/2)$. The bound for $\Zedge^b_\ell$ is the mirror
image: on $[\xi_c,\infty)$, which contains $\{\partial_\xi\prof<0\}$,
$-\partial_\xi P\le\frac1{2d}\operatorname{sech}^2\big(\frac\xi d\big)$ (now the first term
of $\partial_\xi P$ is $\ge0$ and is dropped, and $H(\frac{\xi-a_0}d)\le1$);
$\partial_\xi\chi\le0$ there, and $\partial_\xi\chi\ne0$ only for $\xi>\xi_b$, where
$P\le H(-\xi_b/d)$; and with $\xi=dy$,
$e^{-2\ell(dy-\xi_c)}=\nobreak e^{\ell a_0}e^{-2\ell dy}$.
\end{proof}

\subsection{Mode Bounds and the Semigroup}

For $\ell\in\Z$ and a profile $g$ supported in $[\rho_-,\rho_+]$ put
\[
\|g\|_{k,\ell}^2=\sum_{j=0}^{k}(1+\ell^2)^{k-j}\|\partial_r^jg\|^2_{L^2([\rho_-,\rho_+],r\dd r)},
\]
and let $\|g\|_{k,\ell;[\rho_-,\rho_+]}$ be the same expression for a profile $g$
on $(0,\infty)$ that need not be supported in $[\rho_-,\rho_+]$.

\begin{lemma}
\label{lem:norm-equiv}
In a mode setting, for every integer $k\ge0$ there are $c^{\rm eq}_k,C^{\rm eq}_k>0$, depending
only on $k$, $\rho_-$ and $\rho_+$, such that every
$w=\sum_\ell g_\ell(r)e^{i\ell\theta}\in Y^k$ satisfies
$c^{\rm eq}_k\sum_\ell\|g_\ell\|^2_{k,\ell}\le\|w\|^2_{H^k}\le C^{\rm eq}_k\sum_\ell\|g_\ell\|^2_{k,\ell}$.
\end{lemma}
\begin{proof}
Near $N_*$, $\partial_r=\cos\theta\,\partial_1+\sin\theta\,\partial_2$,
$\partial_\theta=-x_2\partial_1+x_1\partial_2$,
$\partial_1=\cos\theta\,\partial_r-\frac{\sin\theta}r\partial_\theta$ and
$\partial_2=\sin\theta\,\partial_r+\frac{\cos\theta}r\partial_\theta$. Iterating, each
$\partial^\alpha$, $|\alpha|=j$, is a combination of the
$\partial_r^i\partial_\theta^n$, $i+n\le j$, with coefficients bounded on $N_*$ in
terms of $j$, $\rho_\pm$, and conversely. As $\supp w\subset N_*$, $\|w\|_{H^k}^2$ is
comparable to $\sum_{i,n\ge0,\ i+n\le k}\|\partial_r^i\partial_\theta^nw\|^2_{L^2(\R^2)}
=2\pi\sum_\ell\sum_{i,n\ge0,\ i+n\le k}\ell^{2n}\|\partial_r^ig_\ell\|^2_{L^2(r\dd r)}$ (Parseval in
$\theta$). Finally $2^{-p}(1+x)^p\le\sum_{n=0}^{p}x^n\le(1+x)^p$ for $x\ge0$, by the binomial
theorem, as $1\le\binom pn\le2^p$; take
$x=\ell^2$, $p=k-i$.
\end{proof}

\begin{corollary}
\label{cor:K-bound}
In a mode setting, let $s\ge0$ be an integer. There is
$C_K=C_K(s,\rho_-,\rho_+,\|\partial_r\bomega\|_{C^s})$ such that
$\|\kcoef_\ell\Green_\ell g\|_{s,\ell}\le C_K(1+\ell^2)^{-1/2}\|g\|_{s,\ell}$ for every
$\ell\ne0$, every $G$, and every $g\in H^s((0,\infty))$ supported in
$[\rho_-,\rho_+]$.
\end{corollary}
\begin{proof}
Let $\ell\ne0$ and let $g\in H^s((0,\infty))$ be supported in $[\rho_-,\rho_+]$. We
first bound $\Green_\ell g$. Put $\xi=\log r$, $\tilde g(\xi)=g(e^\xi)$, $\tilde\psi(\xi)=(\Green_\ell g)(e^\xi)$ and
$\Xi=e^{2\xi}\tilde g\in H^s(\R)$, compactly supported. For
$r=e^\xi$ and $\rho=e^\eta$, the kernel of $\Green_\ell$ (proof of
Lemma~\ref{lem:adjoint}(1)) is
$K_\ell(r,\rho)=-\frac1{2|\ell|}\big(\min(r,\rho)/\max(r,\rho)\big)^{|\ell|}=-\frac1{2|\ell|}e^{-|\ell||\xi-\eta|}$,
since $\min(r,\rho)/\max(r,\rho)=e^{-|\xi-\eta|}$, and $\rho\dd\rho=e^{2\eta}\dd\eta$. So
$\tilde\psi=-\frac1{2|\ell|}e^{-|\ell||\cdot|}\ast\Xi$. We use the Fourier transform
$\hat u(\zeta)=\int_\R e^{-i\xi\zeta}u(\xi)\dd\xi$ on $\R$, for which
$\frac1{2|\ell|}e^{-|\ell||\xi|}$ has Fourier transform $(\ell^2+\zeta^2)^{-1}$; hence
the Fourier transform of $\tilde\psi$ is $\zeta\mapsto-\hat\Xi(\zeta)/(\ell^2+\zeta^2)$.
With this normalization, Plancherel's theorem reads
$\|u\|^2_{L^2(\R)}=\frac1{2\pi}\int_\R|\hat u|^2\dd\zeta$, which is where the single
factor $\frac1{2\pi}$ below comes from, and $\partial_\xi^ju$ has Fourier transform
$(i\zeta)^j\hat u$. For $u$ on $\R$ and an integer $k\ge0$ put
$|u|_{k,\ell}^2=\sum_{j=0}^{k}(1+\ell^2)^{k-j}\|\partial_\xi^ju\|^2_{L^2(\R)}$; then
\[
|u|^2_{k,\ell}=\frac1{2\pi}\int_\R\sum_{j=0}^{k}(1+\ell^2)^{k-j}\zeta^{2j}\,|\hat u(\zeta)|^2\dd\zeta .
\]
Multiplying the inequality
$2^{-k}(1+x)^k\le\sum_{n=0}^{k}x^n\le(1+x)^k$ of the proof of Lemma~\ref{lem:norm-equiv}
by $A^k$, with $x=B/A$, gives
$\sum_{j=0}^{k}A^{k-j}B^j\le(A+B)^k\le2^k\sum_{j=0}^{k}A^{k-j}B^j$ for $A>0$, $B\ge0$;
we take $A=1+\ell^2$ and $B=\zeta^2$. The first of these inequalities, with $k=s+2$, gives the
first step below; the formula for the Fourier transform of $\tilde\psi$ and
$1+\ell^2+\zeta^2\le2(\ell^2+\zeta^2)$ (as $\ell\ne0$) give the second; and the second
inequality, with $k=s$, gives the third:
\[
|\tilde\psi|^2_{s+2,\ell}\le\frac1{2\pi}\int(1+\ell^2+\zeta^2)^{s+2}\frac{|\hat\Xi|^2}{(\ell^2+\zeta^2)^2}\dd\zeta
\le\frac4{2\pi}\int(1+\ell^2+\zeta^2)^s|\hat\Xi|^2\dd\zeta\le4\cdot2^s|\Xi|^2_{s,\ell}.
\] On
$\tilde I=[\log\rho_-,\log\rho_+]$, $\partial_r=e^{-\xi}\partial_\xi$ and $\partial_\xi=r\partial_r$,
so $\partial_r^j$ is a combination of $\partial_\xi^i$, $i\le j$, with coefficients
bounded on $\tilde I$, and conversely; with $r\dd r=e^{2\xi}\dd\xi$ and
$(1+\ell^2)^{k-j}\le(1+\ell^2)^{k-i}$ for $i\le j$, for every integer $k\ge0$ the norms
$\|\cdot\|_{k,\ell;[\rho_-,\rho_+]}$ and $|\cdot|_{k,\ell}$ restricted to $\tilde I$ are
equivalent, with constants depending on $k$, $\rho_-$ and $\rho_+$. Multiplication
by $e^{2\xi}$ is bounded for $|\cdot|_{s,\ell}$ on functions supported in $\tilde I$,
by the Leibniz rule and the same monotonicity of the weights. Therefore
$\|\Green_\ell g\|_{s+2,\ell;[\rho_-,\rho_+]}\le C|\tilde\psi|_{s+2,\ell}\le C|\Xi|_{s,\ell}
\le C|\tilde g|_{s,\ell}\le C\|g\|_{s,\ell}$, with $C=C(s,\rho_-,\rho_+)$.

The operator $\kcoef_\ell\Green_\ell$ does not involve $G$, and
$\kcoef_\ell\Green_\ell g$ is supported in $I_A=[r_1,r_2]$. Since $1/r\le1/\rho_-$ on
$I_A$, the Leibniz rule gives
\[
\|\partial_r^j(r^{-1}(\partial_r\bomega)\psi)\|_{L^2(I_A,r\dd r)}\le C\sum_{i=0}^{j}\|\partial_r^i\psi\|_{L^2(I_A,r\dd r)}
\]
with $C$ depending on $s$, $\rho_-$ and $\|\partial_r\bomega\|_{C^s}$. Using
$(1+\ell^2)^{s-j}\le(1+\ell^2)^{s-i}$ for $i\le j$, and then
$(1+\ell^2)^{s+2-i}=(1+\ell^2)^2(1+\ell^2)^{s-i}$,
\begin{align*}
\|\kcoef_\ell\Green_\ell g\|_{s,\ell}^2&\le C\ell^2\sum_{i=0}^{s}(1+\ell^2)^{s-i}\|\partial_r^i\Green_\ell g\|^2_{L^2(I_A,r\dd r)}\\
&\le C\ell^2(1+\ell^2)^{-2}\|\Green_\ell g\|^2_{s+2,\ell;[\rho_-,\rho_+]}\le C(1+\ell^2)^{-1}\|g\|^2_{s,\ell},
\end{align*}
since $I_A\subset[\rho_-,\rho_+]$, and by the bound above.
\end{proof}

\begin{lemma}
\label{lem:semigroup-construct}
In a mode setting, let $s\ge0$ be an integer and $G_+\ge0$, and let $|G|\le G_+$. Put
$\Lin^{\rm tr}=-\Omega\partial_\theta$ and $\mathcal Kw=-(\KBS\ast w)\cdot\nabla\bomega$.
\begin{enumerate}[label=\textup{(\arabic*)}]
\item $(e^{t\Lin^{\rm tr}}w)(r,\theta)=w(r,\theta-\Omega(r;G)t)$ is a $C_0$-group on $Y^s$,
  and there is $C_s$, depending only on $s$, $\bomega$, $\rho_\pm$ and $G_+$, with
  $\|e^{t\Lin^{\rm tr}}\|_{Y^s\to Y^s}\le C_s(1+|t|)^s$ and
  $\|e^{-i\ell\Omega t}g\|_{s,\ell}\le C_s(1+|t|)^s\|g\|_{s,\ell}$ for all
  $\ell\in\Z$, $|G|\le G_+$ and $t\in\R$; $\mathcal K$ is bounded on $Y^s$.
\item The Dyson-Phillips series $e^{t\Lin}=\sum_{n\ge0}S_n(t)$, $S_0(t)=e^{t\Lin^{\rm tr}}$,
  $S_{n+1}(t)=\int_0^tS_0(t-\tau)\mathcal KS_n(\tau)\dd\tau$, converges in operator
  norm uniformly on compact $t$-sets and defines a $C_0$-semigroup on $Y^s$ that
  satisfies both Duhamel formulas and preserves each mode. Its generator
  $\hat\Lin$ satisfies $Y^{s+1}\subset D(\hat\Lin)$ and $\hat\Lin w=\Lin(G)w$ for
  $w\in Y^{s+1}$, with $\|\Lin(G)w\|_{Y^s}\le C\|w\|_{Y^{s+1}}$.
\item On each mode, $e^{t\Lin}|_{Y^s_\ell}=\exp(t\Lsec{\ell}(G))$, the norm-continuous
  semigroup generated by the bounded operator $\Lsec{\ell}(G)$ on $Y^s_\ell$.
\end{enumerate}
\end{lemma}
In (3), $Y^s_\ell$ is identified with $P_\ell Y^s$ through $g\mapsto ge^{i\ell\theta}$;
$\Lsec{\ell}(G)$ is bounded on $Y^s_\ell$ because multiplication by $\Omega$ is, and
$\kcoef_\ell\Green_\ell$ is by Corollary~\ref{cor:K-bound} and
Lemma~\ref{lem:norm-equiv}; so, by the argument in the proof of
Lemma~\ref{lem:sector}(4), $\Lsec{\ell}(G)$ generates a norm-continuous group
$\exp(t\Lsec{\ell}(G))$ on $Y^s_\ell$, with derivative
$\Lsec{\ell}(G)\exp(t\Lsec{\ell}(G))=\exp(t\Lsec{\ell}(G))\Lsec{\ell}(G)$.
\begin{proof}
(1) On $Y^s_\ell$, $e^{t\Lin^{\rm tr}}$ is multiplication by $e^{i\varphi}$,
$\varphi=-\ell\Omega(r;G)t$. By the Leibniz and Fa\`a di Bruno formulas,
$\partial_r^j(e^{i\varphi}g)=e^{i\varphi}\sum_{i=0}^{j}\binom jiB_i\,\partial_r^{j-i}g$,
where the complete Bell polynomial $B_i=B_i(i\partial_r\varphi,\dots,i\partial_r^i\varphi)$ is a sum
of at most $C_i$ products of at most $i$ factors $\partial_r^q\varphi=-\ell t\,\partial_r^q\Omega$.
On $[\rho_-,\rho_+]$,
$\|\Omega(\cdot;G)\|_{C^s}\le\|\Omega(\cdot;0)\|_{C^s}+G_+\|r^{-2}\|_{C^s}$, so
$|B_i|\le C(1+|\ell||t|)^i\le C2^{i/2}(1+|t|)^i(1+\ell^2)^{i/2}$ and
\[
(1+\ell^2)^{s-j}\|\partial_r^j(e^{i\varphi}g)\|^2\le C(1+|t|)^{2s}\sum_{i=0}^{j}(1+\ell^2)^{s-(j-i)}\|\partial_r^{j-i}g\|^2;
\]
summing over $j\le s$ gives the mode bound. The map
$(r,\theta)\mapsto(r,\theta-\Omega(r;G)t)$ is a smooth diffeomorphism near $N_*$
preserving $N_*$, so $e^{t\Lin^{\rm tr}}$ maps $Y^s$ to itself, and
Lemma~\ref{lem:norm-equiv} gives the bound on $Y^s$ after enlarging $C_s$ by the
factor $(C^{\rm eq}_s/c^{\rm eq}_s)^{1/2}$. For $w\in Y^s$ and each $\ell$,
$\|(e^{-i\ell\Omega t}-1)g_\ell\|_{s,\ell}\to0$ as $t\to0$ by dominated convergence in
the formula above, with the summable bound $(C_s(1+|t|)^s+1)^2\|g_\ell\|^2_{s,\ell}$;
so $e^{t\Lin^{\rm tr}}w\to w$ in $H^s$ by Lemma~\ref{lem:norm-equiv}. The group law
holds because $\Omega$ does not depend on $\theta$. By Lemma~\ref{lem:sector}(2),
$\mathcal K$ acts on mode profiles as $\kcoef_\ell\Green_\ell$, so
$\|\mathcal Kw\|_{H^s}\le(C^{\rm eq}_s/c^{\rm eq}_s)^{1/2}C_K\|w\|_{H^s}$ by
Corollary~\ref{cor:K-bound} and Lemma~\ref{lem:norm-equiv}, and
$\supp\mathcal Kw\subset\supp\partial_r\bomega\subset N_*$.

(2) By induction $\|S_n(t)\|\le C_s^{n+1}(1+t)^{s(n+1)}\|\mathcal K\|^nt^n/n!$, so the
series converges in operator norm uniformly on compact sets, and its sum $T(t)$ is
strongly continuous, as each $S_n$ is. Summing the recursion,
$T(t)=S_0(t)+\int_0^tS_0(t-\tau)\mathcal KT(\tau)\dd\tau$. Fix $t'\ge0$; we show
that both $T(\cdot+t')$ and $T(\cdot)T(t')$ solve
\[
\Psi(t)=S_0(t)T(t')+\int_0^tS_0(t-\tau)\mathcal K\Psi(\tau)\dd\tau\qquad(t\ge0).
\]
For $T(\cdot)T(t')$, multiply the identity for $T(t)$ on the right by $T(t')$. For
$T(\cdot+t')$, write the identity at time $t+t'$ and split the integral at $\tau=t'$.
On $[0,t']$, $S_0(t+t'-\tau)=S_0(t)S_0(t'-\tau)$ by the group law of $S_0$, and the
bounded operator $S_0(t)$ can be taken out of the integral; on $[t',t+t']$,
substitute $\tau=t'+\sigma$. This gives
\[
T(t+t')=S_0(t)\Big(S_0(t')+\int_0^{t'}S_0(t'-\tau)\mathcal KT(\tau)\dd\tau\Big)
+\int_0^tS_0(t-\sigma)\mathcal KT(\sigma+t')\dd\sigma ,
\]
and the bracket is $T(t')$, which is the equation for $\Psi=T(\cdot+t')$. The difference
$D$ of two strongly continuous, locally bounded solutions satisfies
$\|D(t)w\|\le C\int_0^t\|D(\tau)w\|\dd\tau$, so $D=0$ by iteration; hence $T$ is a
$C_0$-semigroup. By induction and Fubini's theorem, $S_n(t)$ is the integral of
$S_0(t-\tau_n)\mathcal KS_0(\tau_n-\tau_{n-1})\cdots\mathcal KS_0(\tau_1)$ over
$0<\tau_1<\dots<\tau_n<t$, so also $S_{n+1}(t)=\int_0^tS_n(t-\tau)\mathcal KS_0(\tau)\dd\tau$,
which gives the second Duhamel formula
$T(t)=S_0(t)+\int_0^tT(t-\tau)\mathcal KS_0(\tau)\dd\tau$. The $P_\ell$ commute with
$S_0$ and $\mathcal K$, which commute with rotations ($S_0$ because $\Omega$ is
radial, $\mathcal K$ by Lemma~\ref{lem:sector}(2)), hence with $T(t)$. For
$w\in Y^{s+1}$, $\|\Lin^{\rm tr}w\|_{H^s}\le C\|w\|_{H^{s+1}}$, as $\Omega$ and the
coefficients of $\partial_\theta$ are smooth near $N_*$, and
$(e^{h\Lin^{\rm tr}}w-w)/h=\frac1h\int_0^he^{\tau\Lin^{\rm tr}}\Lin^{\rm tr}w\dd\tau\to\Lin^{\rm tr}w$
in $Y^s$, since $\Omega\partial_\theta$ commutes with $e^{\tau\Lin^{\rm tr}}$. By the
second Duhamel formula,
$(T(h)w-w)/h=(e^{h\Lin^{\rm tr}}w-w)/h+\frac1h\int_0^hT(h-\tau)\mathcal Ke^{\tau\Lin^{\rm tr}}w\dd\tau
\to\Lin^{\rm tr}w+\mathcal Kw=\Lin(G)w$.

(3) On $Y^s_\ell$, $S_0(t)=\exp(-i\ell\Omega t)$ and $\mathcal K=\kcoef_\ell\Green_\ell$,
so $T(t)|_{Y^s_\ell}$ solves the first Duhamel equation for these bounded
operators. So does $\exp(t\Lsec{\ell})$, since
\[
\tfrac{\dd}{\dd\tau}\big[\exp(-(t-\tau)i\ell\Omega)\exp(\tau\Lsec{\ell})\big]
=\exp(-(t-\tau)i\ell\Omega)\kcoef_\ell\Green_\ell\exp(\tau\Lsec{\ell}).
\]
By uniqueness they coincide.
\end{proof}

\begin{proposition}
\label{prop:large-sectors}
In a mode setting, let $s\ge0$ be an integer, $\gamma>0$ and $G_+\ge0$. Let
$C_s$ be as in Lemma~\ref{lem:semigroup-construct}\textup{(1)}, $C_K$ as in
Corollary~\ref{cor:K-bound}, $c_s(\gamma)=\sup_{t\ge0}(1+t)^se^{-\gamma t}$, and let
$\mathcal M_\gamma\ge1$ satisfy $C_sC_K(1+\mathcal M_\gamma^2)^{-1/2}\Ip{s}{\gamma}\le\frac12$,
with $\Ip{s}{\gamma}$ as in Equation~\eqref{eq:Ip}. Then
for every $\ell\in\Z$ with $|\ell|\ge\mathcal M_\gamma$, every $G$ with $|G|\le G_+$ and
every $t\ge0$,
$\|e^{t\Lsec{\ell}(G)}\|_{s,\ell\to s,\ell}\le2C_s\,c_s(\gamma)\,e^{\gamma t}$.
\end{proposition}
\begin{proof}
Fix such $\ell$ and $G$. Since $\Lsec{\ell}$ is bounded for $\|\cdot\|_{s,\ell}$
(remark after Lemma~\ref{lem:semigroup-construct}),
$Y_T:=\sup_{0\le t\le T}e^{-\gamma t}\|e^{t\Lsec{\ell}}\|_{s,\ell\to s,\ell}$ is finite
for $T<\infty$. By Lemma~\ref{lem:semigroup-construct}(3) and the first Duhamel
formula,
$e^{t\Lsec{\ell}}=e^{-i\ell\Omega t}+\int_0^te^{-i\ell\Omega(t-\tau)}\kcoef_\ell\Green_\ell e^{\tau\Lsec{\ell}}\dd\tau$,
so for $t\le T$, by Lemma~\ref{lem:semigroup-construct}(1) and
Corollary~\ref{cor:K-bound},
\begin{align*}
e^{-\gamma t}\|e^{t\Lsec{\ell}}\|&\le C_s(1+t)^se^{-\gamma t}
+\int_0^tC_s(1+t-\tau)^se^{-\gamma(t-\tau)}\,\frac{C_K}{(1+\ell^2)^{1/2}}\,e^{-\gamma\tau}\|e^{\tau\Lsec{\ell}}\|\dd\tau\\
&\le C_sc_s(\gamma)+\tfrac12Y_T,
\end{align*}
by the choice of $\mathcal M_\gamma$ and $\int_0^t(1+u)^se^{-\gamma u}\dd u\le\Ip{s}{\gamma}$.
Taking the supremum over $t\le T$ gives $Y_T\le2C_sc_s(\gamma)$, for all $T$.
\end{proof}

For $\mathcal Z\subset\Z$ let $Y^s_{\mathcal Z}$ be the closed subspace of
$w\in Y^s$ with $g_\ell=0$ for $\ell\notin\mathcal Z$; it is invariant under
$e^{t\Lin}$ by Lemma~\ref{lem:semigroup-construct}(2).

\begin{lemma}
\label{lem:sector-sum}
In a mode setting, let $s\ge0$ be an integer, $G_+\ge0$, $|G|\le G_+$,
$\mathcal Z\subset\Z$, and let $\hat\Lin_{\mathcal Z}$ be the generator of the restriction of
$e^{t\Lin}$ to $Y^s_{\mathcal Z}$. Then:
\begin{enumerate}[label=\textup{(\arabic*)}]
\item $w\in D(\hat\Lin_{\mathcal Z})$ if and only if $w\in Y^s_{\mathcal Z}$ and
  $\sum_{\ell\in\mathcal Z}\|\Lsec{\ell}(G)g_\ell\|^2_{s,\ell}<\infty$, and then
  $P_\ell\hat\Lin_{\mathcal Z} w=(\Lsec{\ell}(G)g_\ell)e^{i\ell\theta}$ for every $\ell$;
\item let $\gamma>0$, let $\mathcal M_\gamma$ be as in Proposition~\ref{prop:large-sectors},
  and let $\Re\lambda>\gamma$. For $|\ell|\ge\mathcal M_\gamma$, $\lambda$ lies in the
  resolvent set of $\Lsec{\ell}(G)$ on $Y^s_\ell$, and
  $\|(\lambda-\Lsec{\ell}(G))^{-1}\|_{s,\ell\to s,\ell}\le2C_sc_s(\gamma)/(\Re\lambda-\gamma)$;
\item for $\Re\lambda>\gamma>0$: $\lambda$ lies in the resolvent set of $\hat\Lin_{\mathcal Z}$
  if and only if $\lambda$ lies in the resolvent set of $\Lsec{\ell}(G)$ on
  $Y^s_\ell$ for every $\ell\in\mathcal Z$ with $|\ell|<\mathcal M_\gamma$; and then
  $P_\ell(\lambda-\hat\Lin_{\mathcal Z})^{-1}w=((\lambda-\Lsec{\ell}(G))^{-1}g_\ell)e^{i\ell\theta}$
  for every $\ell\in\mathcal Z$;
\item if $\sigma_+\subset\{\Re\lambda>\gamma\}$ is a compact set that is open and
  closed in the spectrum of $\hat\Lin_{\mathcal Z}$, then the Riesz projection $\Riesz_{\sigma_+}$
  of $\hat\Lin_{\mathcal Z}$ for $\sigma_+$ satisfies
  $P_\ell\Riesz_{\sigma_+}w=(\Riesz_{\ell}g_\ell)e^{i\ell\theta}$, where $\Riesz_\ell$ is
  the Riesz projection of $\Lsec{\ell}(G)$ on $Y^s_\ell$ for
  $\sigma_+\cap\spec\Lsec{\ell}(G)$ \textup{(}and $\Riesz_\ell=0$ if this set is
  empty\textup{)}.
\end{enumerate}
\end{lemma}
\begin{proof}
(1) Since $P_\ell$ is bounded and commutes with $e^{t\Lin}$, for
$w\in D(\hat\Lin_{\mathcal Z})$ Lemma~\ref{lem:semigroup-construct}(3) gives
$P_\ell\hat\Lin_{\mathcal Z} w=\lim_{h\to0}h^{-1}(\exp(h\Lsec{\ell})g_\ell-g_\ell)e^{i\ell\theta}
=(\Lsec{\ell}g_\ell)e^{i\ell\theta}$, and
$\sum_\ell\|\Lsec{\ell}g_\ell\|^2_{s,\ell}\le(c^{\rm eq}_s)^{-1}\|\hat\Lin_{\mathcal Z} w\|^2_{H^s}$ by
Lemma~\ref{lem:norm-equiv}. Conversely, if the sum is finite,
$v:=\sum_\ell(\Lsec{\ell}g_\ell)e^{i\ell\theta}$ converges in $Y^s_{\mathcal Z}$ by
Lemma~\ref{lem:norm-equiv}, and for $t>0$
\[
e^{t\Lin}w-w=\sum_\ell\Big(\int_0^t\exp(\tau\Lsec{\ell})\Lsec{\ell}g_\ell\dd\tau\Big)e^{i\ell\theta}
=\int_0^te^{\tau\Lin}v\dd\tau;
\]
dividing by $t$ and letting $t\to0$ shows $w\in D(\hat\Lin_{\mathcal Z})$.

(2) We use the semigroup formula for the resolvent. Let $A$ be a bounded operator
with $\|e^{tA}\|\le Ke^{\gamma t}$ for $t\ge0$, and let $\Re\lambda>\gamma$. Then
$R=\int_0^\infty e^{-\lambda t}e^{tA}\dd t$ converges in operator norm,
$\|R\|\le K/(\Re\lambda-\gamma)$, and $R=(\lambda-A)^{-1}$: integrating
$\frac{\dd}{\dd t}\big(e^{-\lambda t}e^{tA}\big)=-e^{-\lambda t}e^{tA}(\lambda-A)
=-(\lambda-A)e^{-\lambda t}e^{tA}$ over $[0,\infty)$ gives $R(\lambda-A)=(\lambda-A)R=1$.
By Proposition~\ref{prop:large-sectors}, $A=\Lsec{\ell}(G)$ on $Y^s_\ell$ satisfies this
with $K=2C_sc_s(\gamma)$ when $|\ell|\ge\mathcal M_\gamma$, which gives (2).

(3) If $\lambda$ is in the resolvent set of $\hat\Lin_{\mathcal Z}$, then by (1) $P_\ell$ maps
$D(\hat\Lin_{\mathcal Z})$ into itself and commutes with $\hat\Lin_{\mathcal Z}$ there, so
$(\lambda-\hat\Lin_{\mathcal Z})^{-1}$ commutes with $P_\ell$ and its restriction to the mode
$\ell$ inverts $\lambda-\Lsec{\ell}$. Conversely, under the condition, (2) and the
finiteness of $\{\ell\in\mathcal Z:|\ell|<\mathcal M_\gamma\}$ give
$K:=\sup_{\ell\in\mathcal Z}\|(\lambda-\Lsec{\ell})^{-1}\|_{s,\ell\to s,\ell}<\infty$.
For $w\in Y^s_{\mathcal Z}$, $u=\sum_\ell((\lambda-\Lsec{\ell})^{-1}g_\ell)e^{i\ell\theta}$
converges in $Y^s_{\mathcal Z}$, with
$\|u\|_{H^s}\le(C^{\rm eq}_s/c^{\rm eq}_s)^{1/2}K\|w\|_{H^s}$
(Lemma~\ref{lem:norm-equiv}). Since $\Lsec{\ell}u_\ell=\lambda u_\ell-g_\ell$, the sum in
(1) is finite, so $u\in D(\hat\Lin_{\mathcal Z})$ and $(\lambda-\hat\Lin_{\mathcal Z})u=w$; and
$(\lambda-\hat\Lin_{\mathcal Z})u=0$ forces $(\lambda-\Lsec{\ell})u_\ell=0$ for all $\ell$ by (1),
so $u=0$.

(4) Let $\Gamma\subset\{\Re\lambda>\gamma\}$ be a contour in the resolvent set of
$\hat\Lin_{\mathcal Z}$ winding once around $\sigma_+$ and not around the rest of the spectrum.
Since $P_\ell$ is bounded,
$P_\ell\Riesz_{\sigma_+}w=\frac1{2\pi i}\oint_\Gamma P_\ell(\lambda-\hat\Lin_{\mathcal Z})^{-1}w\dd\lambda
=\big(\frac1{2\pi i}\oint_\Gamma(\lambda-\Lsec{\ell})^{-1}g_\ell\dd\lambda\big)e^{i\ell\theta}$
by (3). The spectrum of $\Lsec{\ell}$ inside $\Gamma$ is
$\sigma_+\cap\spec\Lsec{\ell}$, by (3), which gives the claim.
\end{proof}

\subsection{The Numerical Hypothesis}

The rest of the paper assumes the following statements, which are proved in
Appendix~A of Part~II~\cite{PartII} (Theorem~1.3 there).

\begin{hypothesis}
\label{hyp:N}
Let $a_0,d,\xiL,\xia,\xib,\xiR$, $H$, $\chi$ and $\prof$ be as in
Definition~\ref{def:ring}. For $G\ge0$ and $\xi\in\R$ put
\[
\Omega_G(\xi)=e^{-2\xi}\Big(G+\int_{-\infty}^{\xi}\prof(\xi')e^{2\xi'}\dd\xi'\Big).
\]
For an integer $\ell\ge1$, $G\ge0$ and $\lambda\in\C$ with $\Re\lambda>0$, let
$\psi=\psi_\ell(\cdot;\lambda,G)$ be the solution on $[\xiL,\xiR]$ of
\[
\partial_\xi^2\psi=\Big(\ell^2+\frac{i\ell\,\partial_\xi\prof(\xi)}{\lambda+i\ell\,\Omega_G(\xi)}\Big)\psi,\qquad
\psi(\xiL)=1,\quad\partial_\xi\psi(\xiL)=\ell,
\]
and put
\[
\Evans_\ell(\lambda;G)=e^{-\ell(\xiR-\xiL)}\big(\partial_\xi\psi(\xiR)+\ell\psi(\xiR)\big),
\]
a holomorphic function of $\lambda$ in $\{\Re\lambda>0\}$
(Lemma~\ref{lem:evans-basic}). Put $\Thl{\ell}=\pi\ell d/\sin(\pi\ell d)$ and
$\epsst=\max\{H((\xia-a_0)/d),\,H(-\xib/d)\}$. Then:
\begin{enumerate}[label=\textup{(N\arabic*)}]
\item for every $G\in[0,\frac1{50}]$, the function $\Evans_3(\cdot;G)$ has
  exactly one zero, counted with multiplicity, in the closed half-plane
  $\{\Re\lambda\ge\frac9{200}\}$;
\item for every $G\in[0,\frac1{50}]$, the function $\Evans_3(\cdot;G)$ has no
  zero in the closed strip $\{\frac9{200}\le\Re\lambda\le\frac{211}{2500}\}$;
\item for $\ell\in\{6,9\}$, $e^{\ell a_0}\Thl{\ell}+\epsst<\frac9{100}$.
\end{enumerate}
\end{hypothesis}

By Lemma~\ref{lem:ring-facts}, $6d<1$ and $9d<1$, so $\Thl{6}$ and $\Thl{9}$ are
finite positive numbers.

\subsection{Consequences of the Numerical Hypothesis}

We use the ring in both normalizations of Definition~\ref{def:ring}. In reference
units the ring is $\omc(x)=\prof(\log|x|)$, given directly by the explicit profile
$\prof$, and Hypothesis~\ref{hyp:N}, which Part~II verifies with computer
assistance, is stated in these units, for $G\in[0,\frac1{50}]$. In ring units the
ring $\Phi$ is supported in $\{a\le|y|\le1\}$, with outer radius $1$; this is the
normalization of the building block of the datum~\eqref{eq:datum}, whose rings are
the rescaled copies $A_k\Phi(\cdot/r_k)$, and of the spaces $X^s$ and the operator
$\Lin(G)$ that Part~II uses. Equation~\eqref{eq:evans-units} below translates
between the two.

In each case we fix a mode setting (Definition~\ref{def:sector-setting}); recall
that $(\xi_1,\xi_2)=(\log r_1,\log r_2)$, where $[r_1,r_2]$ contains the support of
$\partial_r\bomega$. In reference units, $\bomega=\omc$, whose profile $\prof$ is
supported in $[\xiL,\xiR]$; we take $(\xi_1,\xi_2)=(\xiL,\xiR)$ and
$N_*=\{\frac12e^{\xiL}\le|x|\le2e^{\xiR}\}$. In ring units, $\bomega=\Phi$; since
$\log|x|=\log|y|+\xiR$, the profile of $\Phi$ in $\xi=\log|y|$ is
$\prof(\xi+\xiR)$, supported in $[\xiL-\xiR,0]$, so we take
$(\xi_1,\xi_2)=(\xiL-\xiR,0)$, that is $(r_1,r_2)=(a,1)$, and $N_*=N$. Both are
mode settings by Lemma~\ref{lem:ring-facts}. We write $\Evans_\ell$ and $\Evans^\Phi_\ell$
for the two Evans functions; likewise, superscripts $\omc$ and $\Phi$ on $\Omega$
and $\nu$ indicate the profile. Since $\Phi=\omc(\Rc\,\cdot)$,
$\Circ_\Phi(r)=\Rc^{-2}\Circ_{\omc}(\Rc r)$, so
\[
\Omega^\Phi(r;G)=r^{-2}\Big(\frac{\Rc^{-2}\Circ_{\omc}(\Rc r)}{2\pi}+G\Big)
=(\Rc r)^{-2}\Big(\frac{\Circ_{\omc}(\Rc r)}{2\pi}+\Rc^2G\Big)=\Omega^{\omc}(\Rc r;\Rc^2G),
\]
and $\nu^\Phi(\xi)=\nu^{\omc}(\xi+\xiR)$, as $\Rc=e^{\xiR}$. Hence the Rayleigh
potential for $(\Phi,G)$ at $\xi$ is the one for $(\omc,\Rc^2G)$ at $\xi+\xiR$. So
if $\psi_\ell(\cdot;\lambda,\Rc^2G)$ is the solution of Definition~\ref{def:evans}
for $(\omc,\Rc^2G)$, normalized at $\xiL$, then $\psi_\ell(\cdot+\xiR;\lambda,\Rc^2G)$
is the one for $(\Phi,G)$, normalized at $\xiL-\xiR$, and the formula for the Evans
function gives
\begin{equation}
\label{eq:evans-units}
\Evans^\Phi_\ell(\lambda;G)=\Evans_\ell(\lambda;\Rc^2G)\qquad(\ell\ge1,\ \Re\lambda>0,\ G\in\R).
\end{equation}

\begin{corollary}
\label{cor:sector-3}
Assume \textup{(N1)} and \textup{(N2)}. For every $G\in[0,\frac1{50}]$, the
function $\Evans_3(\cdot;G)$ of Hypothesis~\ref{hyp:N} has exactly one zero
$\lamst^{\mathrm c}(G)$ in $\{\Re\lambda\ge\frac9{200}\}$; it is simple and
$\Re\lamst^{\mathrm c}(G)>\frac{211}{2500}$.
\end{corollary}
\begin{proof}
By (N1), $\Evans_3(\cdot;G)$ has exactly one zero $\lamst^{\mathrm c}(G)$ in
$\{\Re\lambda\ge\frac9{200}\}$, counted with multiplicity, so this zero is simple; by (N2),
$\Re\lamst^{\mathrm c}(G)>\frac{211}{2500}$.
\end{proof}

For the rest of the paper, $X^s$ denotes the space of $3$-fold rotationally symmetric
$w\in H^s(\R^2;\C)$ with $\supp w\subset N=\{\frac a2\le|y|\le2\}$; thus
$X^s=Y^s_{3\Z}$ for the ring in ring units, and $X^s_\ell=Y^s_\ell$ for $\ell\in3\Z$.
On $X^s$, $\Lin(G)$ denotes the linearization at $(\Phi,G)$ in ring units, that is,
the generator of the restriction of $e^{t\Lin}$ to $X^s$; by
Lemma~\ref{lem:semigroup-construct}(2) it extends the differential operator on
$X^{s+1}$, and Lemma~\ref{lem:sector-sum} describes it mode by mode.

\begin{proposition}
\label{prop:SU}
Assume \textup{(N1)}--\textup{(N3)}, and for $G\in[0,\Gone]$ let
$\lamst(G)=\lamst^{\mathrm c}(\Rc^2G)$ be as in Corollary~\ref{cor:sector-3}. Then
$\Re\lamst(G)>\frac{211}{2500}$, and for every $G\in[0,\Gone]$ and every integer
$s\ge0$: $\spec(\Lsec{3}(G)|_{X^s_3})\cap\{\Re\lambda>\frac9{200}\}=\{\lamst(G)\}$, and
$\lamst(G)$ is an algebraically simple eigenvalue of $\Lsec{3}(G)$ on $X^s_3$;
$\spec(\Lsec{-3}(G)|_{X^s_{-3}})\cap\{\Re\lambda>\frac9{200}\}=\{\overline{\lamst(G)}\}$;
$\spec(\Lsec{\ell}(G)|_{X^s_\ell})\cap\{\Re\lambda>\frac9{200}\}=\emptyset$ for
$\ell\in3\Z\setminus\{3,-3\}$;
\[
\spec(\Lin(G)|_{X^s})\cap\{\Re\lambda>\tfrac9{200}\}=\{\lamst(G),\overline{\lamst(G)}\};
\]
and the Riesz projection of $\Lin(G)$ on $X^s$ for this set acts as the Riesz
projection of $\Lsec{3}(G)$ at $\lamst(G)$ on the mode $3$, as the Riesz
projection of $\Lsec{-3}(G)$ at $\overline{\lamst(G)}$ on the mode $-3$, and as
$0$ on the other modes.
\end{proposition}
\begin{proof}
Fix $G\in[0,\Gone]$ and $s\ge0$; then $\Rc^2G\in[0,\frac1{50}]$. For $\ell=\pm3$: by
Lemma~\ref{lem:resolvent}, run in the mode setting of $\Phi$, the spectrum of
$\Lsec{\pm3}(G)$ on $X^s_{\pm3}$ in $\{\Re\lambda>0\}$ consists of the zeros of
$\Evans^\Phi_3(\cdot;G)$ and their conjugates respectively, all eigenvalues. By
\eqref{eq:evans-units} and the last assertion of Lemma~\ref{lem:evans-basic},
$\Evans^\Phi_3(\cdot;G)$ is the function $\Evans_3(\cdot;\Rc^2G)$ of
Hypothesis~\ref{hyp:N}. So by Corollary~\ref{cor:sector-3} the part of this
spectrum in $\{\Re\lambda\ge\frac9{200}\}$ is $\{\lamst(G)\}$ for $\Lsec{3}(G)$ and
$\{\overline{\lamst(G)}\}$ for $\Lsec{-3}(G)$, where
$\lamst(G)=\lamst^{\mathrm c}(\Rc^2G)$ is a simple zero of $\Evans^\Phi_3(\cdot;G)$
and $\Re\lamst(G)>\frac{211}{2500}$; and Theorem~\ref{thm:evans-simple}(2) shows
that $\lamst(G)$ is algebraically simple on $X^s_3$. Let $|\ell|=6$ or
$|\ell|\ge9$, and let $\ell'=6$ if $|\ell|=6$ and $\ell'=9$ if $|\ell|\ge9$. By
Lemma~\ref{lem:resolvent}, every point $\lambda$ of
$\spec(\Lsec{\ell}(G)|_{X^s_\ell})$ in $\{\Re\lambda>0\}$ is an eigenvalue, with an
eigenfunction in $X^s_\ell=Y^s_\ell\subset Y^0_\ell$. By
Lemma~\ref{lem:closed-form}(1), $\omc$ satisfies the hypothesis of
Lemma~\ref{lem:edge}, and so does $\Phi$, as $\partial_r\Phi(r)=\Rc\,(\partial_r\omc)(\Rc r)$; the
substitution $u=\Rc r$ shows that $\Zedge_\ell$ is the same for both. By
Lemma~\ref{lem:edge}, the monotonicity of $\Zedge_\ell$ (as $\ell'\le|\ell|$),
Lemma~\ref{lem:closed-form}(2) with $\ell'$ in place of $\ell$ (admissible since
$\ell'd<1$, Lemma~\ref{lem:ring-facts}) and (N3),
$\Re\lambda\le\tfrac12\Zedge_{|\ell|}^{1/2}\le\tfrac12\Zedge_{\ell'}^{1/2}
=\tfrac12(\Zedge^a_{\ell'}\Zedge^b_{\ell'})^{1/2}
\le\tfrac12\big(e^{\ell'a_0}\Thl{\ell'}+\epsst\big)<\tfrac12\cdot\tfrac9{100}$;
and $\Lsec{0}=0$. This proves the mode statements. Apply Lemma~\ref{lem:sector-sum} with
$\mathcal Z=3\Z$, $G_+=\Gone$, $\gamma=\frac9{200}$. Every $\lambda$ with
$\Re\lambda>\frac9{200}$ other than $\lamst(G)$, $\overline{\lamst(G)}$ is in the
resolvent set of every $\Lsec{\ell}(G)$, hence of $\Lin(G)$ by
Lemma~\ref{lem:sector-sum}(3), while $\lamst(G)$ and $\overline{\lamst(G)}$ are
eigenvalues with eigenfunctions $g_*e^{3i\theta}$ and $\overline{g_*}e^{-3i\theta}$,
where $g_*=\Mres_{\lamst(G)}\kcoef_3\,\psi_3(\log r;\lamst(G),G)$ is the
eigenfunction of Lemma~\ref{lem:resolvent}(2) in the mode setting of $\Phi$.
This finite set is open and closed in the spectrum, and
Lemma~\ref{lem:sector-sum}(4) gives its Riesz projection.
\end{proof}

The next definition collects the eigendata that Part~II uses. Two of them, $\beta(G)$
and $\pout_*(G)$, describe how a ring interacts with the other rings of the datum.
Let $g$ be a real vorticity supported in a disk $\{|y|\le\rho\}$ with $\rho<a$, inside
the ring. By the expansion following Definition~\ref{def:multipole}, the part of its
stream function in the angular modes $\pm3$ is
$-\Re\big(\frac13\pout_3[g]\,(\bar y)^{-3}\big)$ on the ring, and since
$(\bar y)^{-3}=r^{-3}e^{3i\theta}$, the corresponding velocity $h$ has radial
component $h\cdot\er=-r^{-1}\partial_\theta\big(-\Re(\frac13\pout_3[g]r^{-3}e^{3i\theta})\big)
=\Re\big(i\,\pout_3[g]\,r^{-4}e^{3i\theta}\big)$. The forcing
$-h\cdot\nabla\Phi=-(h\cdot\er)\,\partial_r\Phi$ that this field exerts on the ring
therefore has mode-$3$ part $\pout_3[g]\,b(r)e^{3i\theta}$, with
$b(r)=-\frac i2r^{-4}\partial_r\Phi(r)$. So $be^{3i\theta}$ is the forcing of the ring,
in the mode $3$, by an interior field of unit outer moment, and $\beta(G)$ below is
its component along the unstable eigenfunction. Conversely, $\pout_*(G)$ is the outer
moment of the unstable eigenfunction, which determines the field that a perturbed
ring exerts on the rings outside it. In the relay of Part~II, the kick that a ring
gives to the next ring outward has unstable component $\pout_3[g]\,\beta(G)$, where
$g$ is the inner ring rescaled to the units of the outer one; near saturation the
perturbation of the inner ring is a multiple of its unstable eigenfunction, so
$\pout_3[g]$ is a multiple of $\pout_*$. This is why Lemma~\ref{lem:uniform} bounds
$|\beta|\,|\pout_*|$ from below.

\begin{definition}
\label{def:eigendata}
Assume \textup{(N1)}--\textup{(N3)}, and let $G\in[0,\Gone]$. With the mode
setting of $\Phi$ in ring units and $\ell=3$, let $\lamst(G)$ be as in
Proposition~\ref{prop:SU}, $g_*=g_*(G)$ and $\Jsimp(G)$ as in
Theorem~\ref{thm:evans-simple}, $n_*(G)=\|g_*e^{3i\theta}\|_{L^2(\R^2)}$, and
\[
\phist(G)=n_*(G)^{-1}g_*(r)e^{3i\theta},\qquad \sigma(G)=\Re\lamst(G).
\]
Let $\ellst(G)$ be the functional on $X^0$ given by
$\ellst(G)(w)=n_*(G)\,\ellst(w_3)$, where $w_3$ is the mode-$3$ profile of $w$
and $\ellst$ is the functional of Theorem~\ref{thm:evans-simple}; so
$\ellst(G)(\phist(G))=1$ and $\ellst(G)(\overline{\phist(G)})=0$. The same
formula defines $\ellst(G)$ on $L^2(N)$. Put
$\beta(G)=\ellst(G)\big(b\,e^{3i\theta}\big)$ with $b(r)=-\frac i2r^{-4}\partial_r\Phi(r)$,
$\pout_*(G)=\pout_3[\phist(G)]$ (Definition~\ref{def:multipole}), let $\Riesz(G)$ be the Riesz projection of
Proposition~\ref{prop:SU}, so that $\Riesz(G)w=2\Re\big(\ellst(G)(w)\phist(G)\big)$
for real $w$, and put
\[
z=\lamst(0),\qquad \sigma=\Re z=\sigma(0),\qquad \sigb=\min_{G\in[0,\Gone]}\sigma(G).
\]
\end{definition}

The formula for $\Riesz(G)$ follows from Proposition~\ref{prop:SU} and
Theorem~\ref{thm:evans-simple}(2): on the mode $3$ the Riesz projection is
$w_3\mapsto\ellst(w_3)g_*=\ellst(G)(w)\,n_*^{-1}g_*$, and for real $w$ the mode
$-3$ part is the conjugate of the mode $3$ part (Lemma~\ref{lem:sector}(2)). By
the linearity of $\pout_3$ and $\ellst$, $\pout_*(G)=\pout_*/n_*(G)$ and
$\beta(G)=n_*(G)\beta$, with $\pout_*$ and $\beta$ of Lemma~\ref{lem:adjoint} for
$\Phi$, $\ell=3$ and $\lamst=\lamst(G)$. The minimum defining $\sigb$
exists by Proposition~\ref{prop:U}(a) below. With $\Mres_*=\Mres_{\lamst(G)}$, the
functional $\ellst(G)$ on $L^2(N)$ is represented by the restriction to $N$ of
$x\mapsto\frac{n_*}{2\pi\overline{\Jsimp}}\,\overline{\psi_*(\log|x|)\Mres_*(|x|)}\,e^{3i\theta}$,
which lies in $X^0$, so its norms on $L^2(N)$ and on $X^0$ agree.

\begin{proposition}
\label{prop:U}
Assume \textup{(N1)}--\textup{(N3)}. Then:
\begin{enumerate}[label=\textup{(\alph*)}]
\item $G\mapsto\lamst(G)$ is continuous on $[0,\Gone]$; in particular
  $G\mapsto\sigma(G)$ is continuous, $\sigb$ is attained, and $\sigb>\frac{211}{2500}$;
\item the functions $G\mapsto|\pout_*(G)|$, $|\beta(G)|$, $\|\ellst(G)\|_{(X^0)^*}$
  and $\|\phist(G)\|_{H^k}$ ($k\ge0$) are continuous on $[0,\Gone]$, and the first
  two are positive.
\end{enumerate}
\end{proposition}
\begin{proof}
(a) In reference units put $E(\lambda,G)=\Evans_3(\lambda;G)$, the function of
Hypothesis~\ref{hyp:N} by Lemma~\ref{lem:evans-basic}. In the Picard
iteration of Lemma~\ref{lem:evans-basic}(1), with $G\in[0,\frac1{50}]$ as a
parameter, the iterates are continuous in $(\xi,\lambda,G)$ and holomorphic in
$\lambda$, and the bound $C_A$ does not depend on $G$, by Lemma~\ref{lem:sector}(3).
So $E$ is continuous on $\{\Re\lambda>0\}\times[0,\frac1{50}]$ and holomorphic in
$\lambda$, and $\partial_\lambda E$ is continuous there by Cauchy's formula. Fix
$G_0\in[0,\frac1{50}]$ and the simple zero $\lambda_0=\lamst^{\mathrm c}(G_0)$ (the
eigenvalue in reference units of Corollary~\ref{cor:sector-3}, the superscript
$\mathrm c$ referring to $\omc$; in ring units $\lamst(G)=\lamst^{\mathrm c}(\Rc^2G)$) of
$E(\cdot,G_0)$, let $\delta_1>0$ be such that
$\{|\lambda-\lambda_0|\le\delta_1\}\subset\{\Re\lambda>\frac9{200}\}$ contains no
other zero of $E(\cdot,G_0)$, and let $0<\delta<\delta_1$. By continuity and
compactness there is $\epsilon>0$ with $E\ne0$ on
$\{|\lambda-\lambda_0|=\delta\}\times I_\epsilon$, where
$I_\epsilon=\{G\in[0,\frac1{50}]:|G-G_0|<\epsilon\}$. On the interval $I_\epsilon$
the number of zeros $\frac1{2\pi i}\oint_{|\lambda-\lambda_0|=\delta}\partial_\lambda E/E\dd\lambda$
of $E(\cdot,G)$ in the disk is continuous, integer-valued and $1$ at $G_0$, hence
$1$. This zero
lies in $\{\Re\lambda>\frac9{200}\}$, so it is $\lamst^{\mathrm c}(G)$ by (N1), and
$|\lamst^{\mathrm c}(G)-\lambda_0|<\delta$ for $G\in I_\epsilon$. So $\lamst^{\mathrm c}$
is continuous on $[0,\frac1{50}]$, and (a) follows from
$\lamst(G)=\lamst^{\mathrm c}(\Rc^2G)$, $\sigma(G)>\frac{211}{2500}$
(Proposition~\ref{prop:SU}) and the compactness of $[0,\Gone]$.

(b) In ring units, let $\psi(\xi;\lambda,G)$ be the solution of
Definition~\ref{def:evans} for $\Phi$, normalized at $\xi_1=\xiL-\xiR$. By the
Picard argument and the equation, $\psi$ and all its $\xi$-derivatives are
continuous in $(\xi,\lambda,G)$. Hence $\psi_*(G):=\psi(\cdot;\lamst(G),G)$ and
$g_*(G)=\Mres_{\lamst(G)}\kcoef_3\psi_*(G)$, supported in $[r_1,r_2]=[a,1]$, are
continuous in $G$ in every $C^k([a,1])$. So $n_*(G)$ and
$\Jsimp(G)=\int g_*\psi_*\Mres_*r\dd r$ are continuous, $n_*>0$ by
Lemma~\ref{lem:resolvent}(2), and $\Jsimp\ne0$ by Theorem~\ref{thm:evans-simple}(1),
as $\lamst(G)$ is a simple zero of $\Evans^\Phi_3(\cdot;G)$ (proof of
Proposition~\ref{prop:SU}). By
Lemma~\ref{lem:adjoint}(2) and (3) and the linearity of $\pout_3$ and
$\ellst$, with $\xi_2=0$,
\[
\pout_*(G)=-\frac{6\psi_*(\xi_2)}{n_*},\quad
\beta(G)=-\frac{n_*\,\pin_3[g_*e^{3i\theta}]}{6\Jsimp}=\frac{e^{-3\xi_1}n_*}{\Jsimp};
\]
these are continuous, and nonzero by Lemma~\ref{lem:adjoint}(4). Also
$\|\phist(G)\|_{H^k}$ is continuous, as $g_*(G)$ is continuous in $C^k([a,1])$ and
$n_*>0$, and so is
$\|\ellst(G)\|_{(X^0)^*}=\frac{n_*}{2\pi|\Jsimp|}\|\psi_*(\log|x|)\Mres_*(|x|)\|_{L^2(N)}$,
by the representation after Definition~\ref{def:eigendata}.
\end{proof}

\begin{proposition}
\label{prop:semigroup}
Assume \textup{(N1)}--\textup{(N3)}. Let $s\ge0$ be an integer and
$\sigma'\in(\frac9{200},\sigb)$. There is $C_S=C_S(s,\sigma')$, depending only on
$s$, $\sigma'$ and the ring, such that for every $G\in[0,\Gone]$ and $t\ge0$:
\[
\begin{gathered}
\|e^{t\Lin(G)}(1-\Riesz(G))\|_{X^s\to X^s}\le C_Se^{\sigma't},\qquad
\|\Riesz(G)\|_{X^s\to X^s}\le C_S,\\
\|e^{t\Lin(G)}\|_{X^s\to X^s}\le C_Se^{\sigma(G)t},
\end{gathered}
\]
and, for real $w\in X^s$ and for $w\in X^0$ respectively,
\[
e^{t\Lin(G)}\Riesz(G)w=2\Re\big(\ellst(G)(w)e^{\lamst(G)t}\phist(G)\big),\qquad
\ellst(G)(e^{t\Lin(G)}w)=e^{\lamst(G)t}\ellst(G)(w).
\]
\end{proposition}
\begin{proof}
We use the setting of $\Phi$ in ring units, with $G_+=\Gone$. By
Lemma~\ref{lem:semigroup-construct}(3), Proposition~\ref{prop:SU} and
Lemma~\ref{lem:norm-equiv} (which costs the factor
$(C^{\rm eq}_s/c^{\rm eq}_s)^{1/2}$), it suffices to bound
$\|\exp(t\Lsec{\ell})(1-\Riesz_\ell)\|_{s,\ell\to s,\ell}$ uniformly in $\ell\in3\Z$ and
$G$, where $\Riesz_{\pm3}$ are the mode Riesz projections and $\Riesz_\ell=0$
otherwise. For the fixed $\sigma'$, let $\mathcal M_{\sigma'}$ be the number
$\mathcal M_\gamma$ of Proposition~\ref{prop:large-sectors} with $\gamma=\sigma'$, and let
$\mathcal M_0=\max\{\mathcal M_{\sigma'},4\}$ be the larger of the two numbers
$\mathcal M_{\sigma'}$ and $4$; the $4$ ensures that $|\ell|\ge\mathcal M_0$ excludes the
modes $\pm3$. For $|\ell|\ge\mathcal M_0$ (so $\Riesz_\ell=0$) that
proposition gives the bound $2C_sc_s(\sigma')e^{\sigma't}$, and
$\exp(t\Lsec{0})=1$. For the finitely many $\ell\in3\Z$ with $0<|\ell|<\mathcal M_0$,
$\Lsec{\ell}(G)=\Lsec{\ell}(0)-i\ell Gr^{-2}$ is bounded and affine in $G$, so
$R_*:=1+\max_{0<|\ell|<\mathcal M_0}\sup_{G\in[0,\Gone]}\|\Lsec{\ell}(G)\|_{s,\ell}<\infty$. Fix
$\sigma''\in(\frac9{200},\sigma')$ and let $\Gamma$ be the boundary of
$[-R_*,\sigma'']\times[-R_*,R_*]$. By Proposition~\ref{prop:SU}, the spectrum of
$\Lsec{\ell}(G)$ in $\{\Re\lambda\ge\sigma''\}$ is empty, or $\{\lamst(G)\}$ for
$\ell=3$, or $\{\overline{\lamst(G)}\}$ for $\ell=-3$, with
$\Re\lamst(G)\ge\sigb>\sigma''$; the rest lies in
$\{|\lambda|<R_*,\ \Re\lambda<\sigma''\}$, inside $\Gamma$. So $\Gamma$ lies in the
resolvent set. For a bounded operator $A$ and $\varrho>\|A\|$,
$\frac1{2\pi i}\oint_{|\lambda|=\varrho}e^{t\lambda}(\lambda-A)^{-1}\dd\lambda=e^{tA}$ by
the Neumann series; take $A=\Lsec{\ell}(G)$, $\varrho=2R_*$, so that the circle encloses $\Gamma$. By Cauchy's theorem the
circle may be replaced by $\Gamma$ and a small circle around the eigenvalue
outside $\Gamma$, if any; on the range of the rank-one projection $\Riesz_{\pm3}$
the resolvent is $(\lambda-\lamst)^{-1}$ (or its conjugate), and on the range of
$1-\Riesz_{\pm3}$ it is holomorphic near $\lamst$. Hence
\[
\exp(t\Lsec{\ell})(1-\Riesz_\ell)=\frac1{2\pi i}\oint_\Gamma e^{t\lambda}(\lambda-\Lsec{\ell}(G))^{-1}\dd\lambda,
\]
of norm at most $\frac{|\Gamma|}{2\pi}e^{\sigma''t}\sup_{\Gamma}\|(\lambda-\Lsec{\ell}(G))^{-1}\|_{s,\ell}$,
and the supremum over $\Gamma\times[0,\Gone]$ and $0<|\ell|<\mathcal M_0$ is finite because
inversion is continuous on invertible operators and $\Gamma\times[0,\Gone]$ is compact. This gives the first bound.

Since $\Riesz_3g=\ellst(g)g_*$, the norm
$\|\Riesz(G)\|_{X^s\to X^s}$ is at most $C\|\ellst(G)\|_{(X^0)^*}\|\phist(G)\|_{H^s}$, which is
bounded on $[0,\Gone]$ by Proposition~\ref{prop:U}(b). The mode-$3$ part of
$e^{t\Lin}\Riesz w$ is $\exp(t\Lsec{3})\ellst(w_3)g_*=e^{\lamst t}\ellst(w_3)g_*$, and
for real $w$ the mode $-3$ part is its conjugate (Lemma~\ref{lem:sector}(2));
this gives the formula, and for
complex $w$ the bound $\|e^{t\Lin}\Riesz\|_{X^s\to X^s}\le Ce^{\sigma(G)t}$, so the
third bound follows as $\sigma'<\sigb\le\sigma(G)$. Finally
$\ellst(G)(e^{t\Lin}w)=n_*\ellst(\exp(t\Lsec{3})w_3)=e^{\lamst t}\ellst(G)(w)$ by
Lemma~\ref{lem:semigroup-construct}(3) and Lemma~\ref{lem:adjoint}(1). The constants
depend only on $s$, $\sigma'$ and the ring.
\end{proof}

\begin{lemma}
\label{lem:uniform}
Assume \textup{(N1)}--\textup{(N3)}. Here $\Phi$ is the ring in ring units
(Definition~\ref{def:ring}), $N$ and $N'$ are the annuli of Section~\ref{sec:ring-def},
and all maxima and minima are taken over $G\in[0,\Gone]$; they are attained, since the functions involved are continuous
on this compact interval (Proposition~\ref{prop:U}).
\begin{enumerate}[label=\textup{(\arabic*)}]
\item The constants
  \[
  c_N=\min|\beta|\cdot\min|\pout_*|,\qquad
  C_N=\max|\beta|\cdot\max|\pout_*|,\qquad
  \cout=2\pi\min|\pout_*|
  \]
  satisfy $0<c_N\le C_N<\infty$ and $\cout>0$.
\item The number $\|\ellst\|:=\max_G\|\ellst(G)\|_{(X^0)^*}$ is finite, and so is
  $\max_G\|\phist(G)\|_{H^k}$ for every integer $k\ge0$.
\item Put
  \[
  \begin{gathered}
  \Cls=\|\ellst\|\,\|\nabla\Phi\|_{L^2(\R^2)},\qquad \Cp=\frac4{\sqrt\pi},\\
  \Cnab(k)=2^k\Big(\frac{(k+1)(k+2)}2\Big)^{1/2}\|\Phi\|_{H^{k+1}}\qquad(k\ge0),
  \end{gathered}
  \]
  with the norms of $C^k(N')$ and $H^k$ fixed in Section~\ref{sec:prelim}. Then for
  every $G\in[0,\Gone]$, every integer $k\ge0$, every vector field
  $h\in C^k(N';\R^2)$ and every $g\in L^2(N)$,
  \[
  \begin{gathered}
  |\ellst(G)(-h\cdot\nabla\Phi)|\le\Cls\|h\|_{C^0(N')},\qquad
  |\pout_3[g]|\le\Cp\|g\|_{L^2(N)},\\
  \|h\cdot\nabla\Phi\|_{H^k}\le\Cnab(k)\|h\|_{C^k(N')} .
  \end{gathered}
  \]
  None of these constants depends on $G$; $\Cls$ and $\Cp$ do not depend on $k$,
  while $\Cnab(k)$ does.
\end{enumerate}
\end{lemma}
\begin{proof}
(1), (2) By Proposition~\ref{prop:U}(b), $|\beta|$, $|\pout_*|$, $\|\ellst(G)\|_{(X^0)^*}$ and
$\|\phist(G)\|_{H^k}$ are continuous on $[0,\Gone]$, and $|\beta|$, $|\pout_*|$ are
positive. Continuous functions on the compact interval $[0,\Gone]$ attain their
extrema; this gives $c_N>0$, $C_N<\infty$, $\cout>0$, $\|\ellst\|<\infty$ and the
bound on $\|\phist\|_{H^k}$.

(3) Since $\supp\Phi\subset\{a\le|y|\le1\}\subset N$,
\[
|\ellst(G)(-h\cdot\nabla\Phi)|\le\|\ellst\|\,\|h\cdot\nabla\Phi\|_{L^2(\R^2)}
\le\|\ellst\|\,\|\nabla\Phi\|_{L^2(\R^2)}\|h\|_{C^0(N')} .
\]
Next, $|\pout_3[g]|\le\frac1{2\pi}\|g\|_{L^2(N)}\||y|^3\|_{L^2(N)}$ and
$\||y|^3\|^2_{L^2(N)}=2\pi\int_{a/2}^2r^7\dd r\le2\pi\cdot2^8/8=64\pi$, so
$|\pout_3[g]|\le\frac{8\sqrt\pi}{2\pi}\|g\|_{L^2(N)}=\frac4{\sqrt\pi}\|g\|_{L^2(N)}$.
Finally, for $|\alpha|\le k$, since $\nabla\Phi$ is supported in $N'$ and
$|\partial^\beta h|\le\|h\|_{C^k(N')}$ there for $|\beta|\le k$, the Leibniz rule
and the Cauchy-Schwarz inequality in $\R^2$ give
\[
\|\partial^\alpha(h\cdot\nabla\Phi)\|_{L^2(\R^2)}\le\sum_{\beta\le\alpha}\binom\alpha\beta
\sup_{N'}|\partial^\beta h|\,\big\||\partial^{\alpha-\beta}\nabla\Phi|\big\|_{L^2(\R^2)}
\le2^{|\alpha|}\|h\|_{C^k(N')}\|\Phi\|_{H^{k+1}};
\]
summing the squares over the $(k+1)(k+2)/2$ multi-indices gives the last inequality.
\end{proof}

\section*{Statements and Declarations}
The author has no competing interests to declare.

\end{document}